\documentclass[12pt,a4paper]{article}
\usepackage{latexsym,amssymb,amsfonts,amsmath,amsthm,nccmath,float,enumitem,setspace,bm,authblk,longtable,cite}
\usepackage[usenames,dvipsnames]{xcolor}
\usepackage[hidelinks]{hyperref}
\usepackage[margin=2cm]{geometry}
\usepackage{tikz}

\setlist{topsep=3pt,partopsep=0pt,itemsep=1pt,parsep=0pt}

\hypersetup{
    colorlinks,
    linkcolor={red!80!black},
    citecolor={blue!80!black},
    urlcolor={blue!80!black}
}

\numberwithin{equation}{section}
\newtheorem{Theorem}{Theorem}[section]

\newtheorem{Corollary}[Theorem]{Corollary}

\newtheorem{Remark}[Theorem]{Remark}
\newtheorem{Lemma}[Theorem]{Lemma}
\newtheorem{Example}[Theorem]{Example}
\newtheorem{Definition}[Theorem]{Definition}

\newtheorem{Construction}[Theorem]{Construction}

\newcommand{\Z}{\mathbb{Z}}

\def \leq {\leqslant}
\def \geq {\geqslant}

\def \mod#1{{\:({\rm mod}\ #1)}}

\let\oldproofname=\proofname
\renewcommand{\proofname}{\rm\bf{\oldproofname}}

\allowdisplaybreaks[4]

\begin{document}

\title{Constructions of two-dimensional optical orthogonal codes of weight three}

\author[a]{Xiuling Shan}
\author[b]{Lidong Wang}
\author[c]{Yanxun Chang}
\author[d]{Xiaomiao Wang}

\affil[a]{School of Mathematical Sciences, Hebei Normal University, Shijiazhuang, 050024, China}
\affil[b]{School of Intelligence Policing, China People's Police University, Langfang, 065000, China}
\affil[c]{School of Mathematics and Statistics, Beijing Jiaotong University, Beijing 100044, China}
\affil[d]{School of Mathematics and Statistics, Ningbo University, Ningbo 315211, China}
\affil[ ]{xiulingshan@hebtu.edu.cn; wanglidong@cppu.edu.cn; yxchang@bjtu.edu.cn; wangxiaomiao@nbu.edu.cn}

\renewcommand*{\Affilfont}{\small\it}
\renewcommand\Authands{ and }
\date{}

\maketitle

\footnotetext{Supported by Natural Science Foundation of Hebei Province under Grant A2025205023 (X. Shan), S\&T Program of Hebei under Grant A$2023507001$ and NSFC under Grant $12571346$ (L. Wang), NSFC under Grant $12371326$ (Y. Chang), Ningbo Natural Science Foundation under Grant 2024J018 (X. Wang).}

\begin{abstract}
The study of optical orthogonal codes has been motivated by an application in an optical code-division multiple access system. This paper focuses on optimal two-dimensional optical orthogonal codes with autocorrelation and cross-correlation both equal to $1$.
By examining the structures of $n$-cyclic group divisible packings and semi-cyclic incomplete holey group divisible designs, we present new combinatorial constructions for two-dimensional $(m\times n,k,1)$-optical orthogonal codes. As a consequence, the exact number of codewords of an optimal two-dimensional $(m\times n,3,1)$-optical orthogonal code is determined for any positive integers $m$ and $n$.
\end{abstract}


\section{Introduction}\label{sec:1}

An optical orthogonal code is a family of sequences with good auto- and cross-correlation
properties. Its study has been motivated by an application in an optical code-division multiple access (OCDMA) system. OCDMA is one of the most important techniques supporting many simultaneous users in shared media so as to increase the transmission capacity of an optical fibre. For related details, the reader may refer to \cite{kpp,ml,mmc,sa,sb}.

\begin{Definition}
Let $m$, $n$, $k$, and $\lambda$ be positive
integers. A {\em two-dimensional $(m\times n,k,\lambda)$ optical
orthogonal code} $($briefly $2$-D $(m\times n,k,\lambda)$-OOC$)$,
$\cal{C}$, is a family of $m\times n$ $(0, 1)$-matrices $($called {\em
codewords}$)$ of Hamming weight $k$ satisfying the following two properties:
\begin{enumerate}
\item[$(1)$] the auto-correlation property: for each matrix ${\mathbf{A}}=(a_{ij})\in\cal{C}$ and each integer $\tau$, $\tau\not\equiv 0\pmod n$,
$\sum_{i=0}^{m-1}\sum_{j=0}^{n-1}a_{ij}a_{i,j+\tau}\leq\lambda;$
\item[$(2)$] the cross-correlation property: for each matrix ${\mathbf{A}}=(a_{ij})\in\cal{C}$, ${\mathbf{B}}=(b_{ij})\in\cal{C}$ with ${\mathbf{A}}\neq {\mathbf{B}}$, and each integer $\tau$, $\sum_{i=0}^{m-1}\sum_{j=0}^{n-1}a_{ij}b_{i,j+\tau}\leq\lambda$,
\end{enumerate}
where all subscripts are reduced modulo $n$. When $m=1$, a two-dimensional $(1\times n,k,\lambda)$ optical
orthogonal code is said to be a one-dimensional $(1\times n,k,\lambda)$ optical
orthogonal code, denoted by $1$-D $(n,k,\lambda)$-OOC.
\end{Definition}

The number of codewords of a $2$-D $(m\times n,k,\lambda)$-OOC is called its \emph{size}. From a practical point of view, a code with a large size is required \cite{sb}. Let $\Phi(m\times n,k,\lambda)$ denote the largest possible size of a $2$-D $(m\times n,k,\lambda)$-OOC. A $2$-D $(m\times n,k,\lambda)$-OOC with $\Phi(m\times n,k,\lambda)$ codewords is said to be \emph{optimal}. By the well-known Johnson bound, Yang and Kwong \cite{yk} presented the following upper bound.

\begin{Lemma}\rm{\cite{yk}}\label{Johnson bound}
$\Phi(m\times n,k,\lambda)\leq J(m\times n,k,\lambda)$, where $$J(m\times n,k,\lambda)=\left\lfloor\frac{m}{k}\left\lfloor\frac{(mn-1)}{k-1} \cdots\left\lfloor\frac{(mn-\lambda)}{k-\lambda}\right\rfloor\cdots\right\rfloor\right\rfloor.$$
\end{Lemma}

In an OCDMA system, performance analysis shows that OOCs with $\lambda\leq3$ are the most
desirable. Especially, from a multiaccess and synchronization point of view, the most desirable
on-of signature sequences are OOCs with $\lambda=1$ \cite{ms}. In this paper we only focus on the
case $\lambda=1$. The reader is referred to \cite{am1,am,bj,bt,cc1,cc,ck,dx,fcj,fcj1,fc,lwl,mms}
and \cite{am2,am3,am4,cc2,yang,okeb} for constructions of OOCs with $\lambda=2$ and general $\lambda$, respectively.

$1$-D $(n,3,1)$-OOCs were investigated systematically by Brickell and Wei \cite{bw} and Chung, Salehi and Wei \cite{csw}.

\begin{Lemma}\label{1D}\rm{\cite{bw,csw}}
Let $n$ be a positive integer. Then there exists an optimal $1$-D $(n,3,1)$-OOC with $J(1\times n,3,1)-\delta$ codewords, i.e., $\Phi(1\times n,3,1)=J(1\times n,3,1)-\delta$, where $$\delta=\left\{
\begin{array}{lll}
1, &  n\equiv 14,20\ ({\rm mod}\ 24),\\
0, &  {\rm otherwise.\ }\\
\end{array}
\right.
$$
\end{Lemma}

Since then there are many researches
on $1$-D OOCs (see, e.g., \cite{ab,be,b2,b3,b4,bp,bp1,bs,csfw,csfw1,cfm,cj,cm1,cy,cgz,fm,fmy,gms,gy,km,mc1,mc2,wc2,yin,yyl}). Despite the vast amount of energy spent on $1$-D OOCs, the problem of determining the sizes of optimal $1$-D $(n,k,1)$-OOCs with $k\geq4$ is far from
being settled, even for $k=4$. Very recently, by utilizing the efficient direct construction method presented in \cite{zfw}, it was shown in \cite{zzcfwz,zcf} that
there exists an optimal $1$-D $(n,4,1)$-OOC with $J(1\times n,4,1)$ codewords for any positive integer $n$ with the only definite exception of $n=25$.

If we turn to the case of $m>1$, relatively little is known. A special kind of optimal $2$-D $(m\times n,k,1)$-OOCs with size $m(mn-1)/k(k-1)$
are said to be {\em perfect}. The existence spectrum for
perfect $2$-D $(m\times n,3,1)$-OOCs was completely solved by Cao and Wei \cite{cw}.

\begin{Lemma}\label{perfect}\rm{\cite{cw}}
There exists a perfect $2$-D $(m\times n,3,1)$-OOC if and only if $m,n$ are odd and $m(mn-1)\equiv0\pmod6$.
\end{Lemma}

Clearly, the optimal codes provided by Lemma \ref{perfect} are limited. Wang, Shan and Yin \cite{wsy} focused on the combinatorial constructions for general optimal $2$-D $(m\times n,3,1)$-OOCs when $n$ is odd.

\begin{Lemma}\label{2D:m odd}\rm{\cite{wsy}}
Let $m$ be a positive integer and $n$ be an odd integer. Then there exists an optimal $2$-D $(m\times n,3,1)$-OOC with $J(m\times n,3,1)-\eta$ codewords, i.e., $\Phi(m\times n,3,1)=J(m\times n,3,1)-\eta$, where $$\eta=\left\{
\begin{array}{lll}
1, &  m\equiv 5\ ({\rm mod}\ 6)\ {\rm  and\ } n=1,\\
0, &  {\rm otherwise.\ }\\
\end{array}
\right.
$$
\end{Lemma}

This paper is a continuation of the study of \cite{wsy}. As the main result of this paper, we are to prove the following theorem.

\begin{Theorem}\label{main}
Let $m,n$ be positive integers. Then there exists an optimal $2$-D $(m\times n,3,1)$-OOC with $J(m\times n,3,1)-\mu$ codewords, i.e., $\Phi(m\times n,3,1)=J(m\times n,3,1)-\mu$, where
$$
    \mu=\left\{
          \begin{array}{ll}
               1,
               &{\rm{if}}\ m\equiv5\ ({\rm mod}\ 6),\ n=1;\\
               &{\rm{or}}\ m\equiv4\ ({\rm mod}\ 6), \ n=4;\\
               &{\rm{or}}\ m\equiv5,8\ ({\rm mod}\ 12), \ n=2;\\
               &{\rm{or}}\ n\equiv0\ ({\rm mod}\ 2), \ mn\equiv14,20\ ({\rm mod}\ 24);\\
               &{\rm{or}}\ n\equiv0\ ({\rm mod}\ 2), \ m\equiv0\ ({\rm mod}\ 3), \ mn\equiv6,12\ ({\rm mod}\ 24);\\
               0,
               &\rm{otherwise}.
          \end{array}
       \right.
$$
\end{Theorem}

The rest of the paper is organized as follows. Section $2$ gives an equivalent description of $2$-D $(m\times n,k,1)$-OOCs using set-theoretic notations. Section $3$ improves the upper bounds for the size of $2$-D $(m\times n,3,1)$-OOCs. Section $4$ introduces two special types of $2$-D OOCs with ``holes'' to build recursive constructions for $2$-D OOCs. The two special types of $2$-D OOCs are constructed in Sections $6$ and $7$, respectively. Before that, group divisible packings and incomplete holey group divisible designs are introduced in Section $5$. Our main result, Theorem \ref{main}, is proved in Section $8$. 

\section{Set-theoretic description}\label{sec:2}

Throughout this paper, we assume that $I_m=\{0,1,\ldots,m-1\}$ and denote by $\mathbb{Z}_n$ the additive
group of integers modulo $n$.

A convenient way of viewing two-dimensional optical orthogonal codes is from a set-theoretic perspective. Take a $2$-D $(m\times n,k,\lambda)$-OOC, $\cal C$. For each $m\times n$ $(0,1)$-matrix $M\in\cal C$, whose rows are indexed by $I_m$ and columns are indexed by $\mathbb{Z}_n$, construct a $k$-subset $B_M$ of $I_m\times\mathbb{Z}_n$ such that $(i,j)\in B_M$ if and only if $M$'s $(i,j)$ cell equals to $1$. Then $\{B_M: M\in {\cal C}\}$ is a set-theoretic representation of the $2$-D $(m\times n,k,\lambda)$-OOC.

\begin{Definition}\rm{\cite{hc}}
Let $m$, $n$ and $k$ be positive integers. A set $\cal C$ of $k$-subsets of $I_m\times\mathbb{Z}_n$ constitutes a $2$-D $(m\times n,k,\lambda)$-OOC if the following conditions hold:
\begin{enumerate}
\item[$(1')$] the auto-correlation property: $|A\bigcap(A+\tau)|\leq\lambda$ for any $A\in{\cal C}$ and any integer $\tau$, $\tau\not\equiv 0\pmod n$;
\item[$(2')$] the cross-correlation property: $|A\bigcap(B+\tau)|\leq\lambda$ for any $A,B\in{\cal C}$ with $A\neq B$ and any integer $\tau$,
\end{enumerate}
where $B+\tau=\{(i,x+\tau\pmod n):(i,x)\in B\}$.
\end{Definition}

\begin{Example}\label{(2,6)}
The following three $(0,1)$-matrices constitute a $2$-D $(2\times 6,3,1)$-OOC.
\begin{center}
\begin{tabular}{lll}
$\left(
  \begin{array}{cccccc}
    1 & 1& 0& 0& 0& 0 \\
    0 & 0& 1& 0& 0& 0 \\
  \end{array}
\right),$&
$\left(
  \begin{array}{cccccc}
    1 & 0& 1& 0& 0& 0 \\
    0 & 0& 0& 0& 0& 1 \\
  \end{array}
\right),$&

$\left(
  \begin{array}{cccccc}
    1 & 0& 0& 0& 0& 0 \\
    1 & 0& 0& 0& 1& 0 \\
  \end{array}
\right).$
\end{tabular}
\end{center}
In set notation, it consists of three $3$-subsets of $I_2\times\mathbb{Z}_6$:
\begin{center}
\begin{tabular}{lll}
$B_1=\{(0,0),(0,1),(1,2)\}$,
&$B_2=\{(0,0),(0,2),(1,5)\}$,
&$B_3=\{(0,0),(1,0),(1,4)\}$.
\end{tabular}
\end{center}
\end{Example}

For a given set $\mathcal{C}$ of $k$-subsets of $I_m\times\mathbb{Z}_n$, it is not convenient to check the correctness of Conditions $(1')$ and $(2')$. But fortunately, the pure and mixed difference method introduced by Bose \cite{b} is efficient to describe a $2$-D $(m\times n,k,1)$-OOC.

For $(i,x),(i,y)\in I_m\times\mathbb{Z}_n$ with $x\neq y$, the difference $x-y\pmod n$ is called a {\em pure $(i,i)$-difference}. For $(i,x),(j,y)\in I_m\times\mathbb{Z}_n$ with $i\neq j$, the difference $x-y\pmod n$ is called a {\em mixed $(i,j)$-difference}. Clearly, the total number of mixed differences and the total number of pure differences in $I_m\times\mathbb{Z}_n$ are $m(m-1)n$ and $m(n-1)$ respectively. Let $\mathcal{C}$ be a set of $k$-subsets of $I_m\times\mathbb{Z}_n$. Given $i,j\in I_m$, for every $B\in\mathcal{C}$, define a multi-set
$$\Delta_{ij}(B)=\{x-y \pmod n: (i,x),(j,y)\in B, (i,x)\neq(j,y)\},$$ and a multi-set $$\Delta_{ij}(\mathcal{C})=\bigcup_{B\in\mathcal{C}}\Delta_{ij}(B).$$
When $i=j$, $\Delta_{ii}(\mathcal{C})$ is the multi-set of all pure $(i,i)$-differences of $\mathcal{C}$. When $i\neq j$, $\Delta_{ij}(\mathcal{C})$ is the multi-set of all mixed $(i,j)$-differences of $\mathcal{C}$. Note that $\Delta_{ij}(\mathcal{C})$ is empty if $i$ or $j$ does not occur as the first component of the elements of any $B\in\mathcal{C}$.

\begin{Definition}\label{def-diff}
Let $m$, $n$ and $k$ be positive integers. A set $\cal C$ of $k$-subsets of $I_m\times\mathbb{Z}_n$ constitutes a $2$-D $(m\times n,k,1)$-OOC if
$\Delta_{ij}(\mathcal{C})$ covers every element of $\mathbb{Z}_n$ at most once for all $i,j\in I_m$.
\end{Definition}

\begin{Remark}
Let $\mathcal{C}$ be a $2$-D $(m\times n,k,1)$-OOC with $n\equiv0\pmod2$. We have
\begin{enumerate}
\item[$(1)$] The set $\mathbb{Z}_n\setminus\Delta_{ii}(\mathcal{C})$ must contain $\{0,n/2\}$ as a subset for all $i\in I_m$. Otherwise, $n/2$ will appear at least twice in a certain $\Delta_{ii}(\mathcal{C})$ for some $i\in I_m$.
\item[$(2)$] Since each codeword contributes $k(k-1)$ differences, then $$|\mathcal{C}|\leq\left\lfloor\frac{m(m-1)n+m(n-2)}{k(k-1)}\right\rfloor=\left\lfloor\frac{m(mn-2)}{k(k-1)}\right\rfloor.$$ Furthermore, if
    $|\mathcal{C}|=m(mn-2)/k(k-1)$, then $\mathbb{Z}_n\setminus\Delta_{ii}(\mathcal{C})=\{0,n/2\}$ for all $i\in I_m$.
\end{enumerate}
\end{Remark}

\begin{Example}
Using the pure and mixed difference method, we now check the three $3$-subsets of $I_2\times\mathbb{Z}_6$ shown in Example \ref{(2,6)}:
\begin{center}
\begin{tabular}{llll}
$\Delta_{00}(B_1)=\{1,5\}$,
&$\Delta_{11}(B_1)=\emptyset$,
&$\Delta_{01}(B_1)=\{4,5\}$,
&$\Delta_{10}(B_1)=\{1,2\}$,\\
$\Delta_{00}(B_2)=\{2,4\}$,
&$\Delta_{11}(B_2)=\emptyset$,
&$\Delta_{01}(B_2)=\{1,3\}$,
&$\Delta_{10}(B_2)=\{3,5\}$,\\
$\Delta_{00}(B_3)=\emptyset$,
&$\Delta_{11}(B_3)=\{2,4\}$,
&$\Delta_{01}(B_3)=\{0,2\}$,
&$\Delta_{10}(B_3)=\{0,4\}$.
\end{tabular}
\end{center}
It is readily checked that Definition \ref{def-diff} is satisfied. Thus the three codewords form a $2$-D $(2\times6,3,1)$-OOC.
\end{Example}

\section{Improved upper bounds for weight three}

In a $2$-D $(m\times n,3,1)$-OOC, each codeword is of the form $\{(i_1,x),(i_2,y),(i_3,z)\}$, where $i_1,i_2,i_3\in I_m$ and $x,y,z\in \mathbb{Z}_n$.
All its codewords can be divided into the following three types:
\begin{itemize}
\item Type $1$: $\{(i_t,0),(i_t,a_t),(i_t,b_t)\}$, $1\leq t\leq \alpha$;
\item Type $2$: $\{(i_{1t},0),(i_{1t},c_t),(i_{2t},d_t)\}$, $i_{1t}\neq i_{2t}$, $1\leq t\leq\beta$;
\item Type $3$: $\{(i_{1t},0),(i_{2t},e_t),(i_{3t},f_t)\}$, $i_{1t}<i_{2t}<i_{3t}$, $1\leq t\leq\gamma$.
\end{itemize}
Without loss of generality, we assume that $0<a_t,c_t<n/2$ and $a_t<b_t$. For Type $1$ and Type $2$, we can further assume that
\begin{center}
\begin{tabular}{lll}
$\left\{
\begin{array}{lll}
i_{1t}<i_{2t}, & {\rm{if}}\ 1\leq t\leq\beta_1;\\
i_{1t}>i_{2t}, & {\rm{if}}\ \beta_1+1\leq t\leq\beta.
\end{array}
\right.$&

$\left\{
\begin{array}{lll}
b_t<n/2, & {\rm{if}}\ 1\leq t\leq\alpha_1;\\
b_t>n/2, b_t-a_t<n/2, & {\rm{if}}\ \alpha_1+1\leq t\leq\alpha_2;\\
b_t>n/2, b_t-a_t>n/2, & {\rm{if}}\ \alpha_2+1\leq t\leq \alpha.
\end{array}
\right.$
\end{tabular}
\end{center}

\begin{Lemma}\label{bd1}
Assume that $n\equiv0\pmod{2}$. Then $\Phi(m\times n,3,1)\leq J(m\times n,3,1)-1$ if $m\equiv0\pmod3$ and $mn\equiv6,12\pmod{24}$, or
$mn\equiv14,20\pmod{24}$.
\end{Lemma}

\proof Suppose that there exists a $2$-D $(m\times n,3,1)$-OOC with $J(m\times n,3,1)$ codewords in both cases.
Now, we calculate the sum $A$ of all mixed $(i,j)$-differences in
such a code with $i,j\in I_m$ and $i>j$. Note that all
mixed differences are produced from codewords in Type $2$ and Type $3$. We have

\begin{eqnarray}
 A&=&\sum_{t=1}^{\beta_1}(d_t+d_t-c_t)+\sum_{t=\beta_1+1}^{\beta}(-d_t+c_t-d_t)+\sum_{t=1}^{\gamma}(e_t+f_t+f_t-e_t) \nonumber \\
 &=&2[\sum_{t=1}^{\beta_1}(d_t-c_t)-\sum_{t=\beta_1+1}^{\beta}d_t+\sum_{t=1}^{\gamma}f_t]+\sum_{t=1}^{\beta}c_t. \nonumber
\end{eqnarray}
Next, we calculate the sum $B$ of all pure $(i,i)$-differences less than $n/2$ in
such a code with $i\in I_m$. Note that all pure differences are produced from codewords in Type $1$ and Type $2$. We have

\begin{eqnarray}
 B&=&\sum_{t=1}^{\alpha_1}(a_t+b_t+b_t-a_t)+\sum_{t=\alpha_1+1}^{\alpha_2}(a_t-b_t+b_t-a_t)+\sum_{t=\alpha_2+1}^{\alpha}(a_t-b_t+a_t-b_t)+ \sum_{t=1}^{\beta}c_t\nonumber \\
 &=&2[\sum_{t=1}^{\alpha_1}b_t-\sum_{t=\alpha_2+1}^{\alpha}(a_t-b_t)]+\sum_{t=1}^{\beta}c_t. \nonumber
\end{eqnarray}
Note that $A$ and $B$ are calculated modulo $n$ and $n\equiv0\pmod2$, then we have $A\equiv B\pmod2$. On the other hand, since $J(m\times n,3,1)=m(mn-2)/6$, we know that there are exactly $m$ pure differences $n/2$ left. Thus
$$ A={m\choose2}\sum_{i=0}^{n-1}i=\frac{mn(m-1)(n-1)}{4}, B=m\sum_{i=1}^{(n-2)/2}i=\frac{mn(n-2)}{8}.$$
Therefore, $$\frac{mn(m-1)(n-1)}{4}\equiv\frac{mn(n-2)}{8}\pmod2,$$ which is
impossible when $m\equiv0\pmod{3}$ and $mn\equiv6,12\pmod{24}$, or $mn\equiv14,20\pmod{24}$. Hence, $\Phi(m\times n,3,1)\leq J(m\times n,3,1)-1$.
\qed

\begin{Lemma}\label{bd:n=2}
$\Phi(m\times 2,3,1)\leq J(m\times 2,3,1)-1$ if $m\equiv5,8\pmod{12}$.
\end{Lemma}

\proof Suppose there is a $2$-D $(m\times 2,3,1)$-OOC with $J(m\times 2,3,1)$ codewords.
Since $n=2$, there is no pure difference covered by any codeword and then all
codewords are Type $3$ and of the form $\{(i,0),(j,0),(k,0)\}$ or $\{(i,0),(j,0),(k,1)\}$ with distinct $i,j$ and $k$.

Consider the mixed $(i,j)$-differences with $i>j$ for all $i,j\in I_m$. Since the total number of such mixed differences is $m(m-1)$ and each codeword contributes three such mixed differences, then the total number of the unused mixed $(i,j)$-differences with $i>j$ for all $i,j\in I_m$ is $m(m-1)-3J(m\times 2,3,1)=2$ due to $m\equiv5,8\pmod{12}$.

For every $i_0\in I_m$, the number of mixed $(i,j)$-differences with $i>j$ and $i_0\in\{i,j\}$
is $2i_0+2(m-i_0-1)=2m-2$. On the other hand, each codeword contributes $0$ or $2$ such mixed differences, and then the number of unused mixed $(i,j)$-differences with $i>j$ and $i_0\in\{i,j\}$ must be even. Therefore, the two unused differences are all mixed $(i',j')$-differences for some fixed $i',j'\in I_m$, one of which is the mixed $(i',j')$-difference $0$, and the other is the mixed $(i',j')$-difference $1$.

Then the total number of mixed $(i,j)$-difference $1$ with $i>j$ covered by all codewords is $m(m-1)/2-1$, which is odd. However, each codeword of the form $\{(i,0),(j,0),(k,1)\}$ contributes $2$ mixed differences $1$, this creates a contradiction.\qed

\begin{Lemma}\label{bd:n=4}
$\Phi(m\times 4,3,1)\leq J(m\times 4,3,1)-1$ if $m\equiv4\pmod{6}$.
\end{Lemma}

\proof
Suppose that there is a $2$-D $(m\times 4,3,1)$-OOC with $J(m\times 4,3,1)=(2m^2-m-1)/3$ codewords.
Since $n=4$, there is no codeword of Type $1$. Without loss of generality, we assume that each codeword of Type $2$ is of the form $\{(i,0),(i,1),(j,x)\}$.

Clearly, the total number of mixed $(i,j)$-differences with $i>j$ is $2m(m-1)$ and the total number of pure $(i,i)$-differences less than $2$ is $m$. Since each codeword contributes $3$ such differences, then the number of unused differences are $2m(m-1)+m-(2m^2-m-1)=1$.

If the unused difference is a pure difference, then the number of codewords in Type $2$ must be $m-1$. Note that each codeword of Type $2$ contributes an odd mixed difference and an even mixed difference. Thus the total number of odd mixed differences produced from codewords of Type $3$ are $m(m-1)-(m-1)=(m-1)^2$, which is odd. However, each codeword of Type $3$ contributes $0$ or $2$ odd mixed differences, this creates a contradiction.

If the unused difference is a mixed $(i_0,j_0)$-difference with $i_0>j_0$, then the total number of mixed $(i_0,j)$-difference for all $i_0>j$ and
mixed $(j,i_0)$-difference for all $j>i_0$ produced from all codewords must be $4(i_0-1)+4(m-i_0)-1=4m-5$, which is odd. However, the number of such mixed differences contributed by each codeword is $0$ or $2$, this also creates a contradiction.
\qed

\begin{Lemma}\label{bd}
Assume that $n\equiv0\pmod{2}$. Then $\Phi(m\times n,3,1)\leq J^*(m\times n,3,1)$, where $$
    J^*(m\times n,3,1)=\left\{
          \begin{array}{ll}
               J(m\times n,3,1)-1,
               &{\rm{if}}\ mn\equiv14,20\ ({\rm mod}\ 24);\\
               &{\rm{or}}\ m\equiv4\ ({\rm mod}\ 6), \ n=4;\\
               &{\rm{or}}\ m\equiv5,8\ ({\rm mod}\ 12), \ n=2;\\
               &{\rm{or}}\ m\equiv0\ ({\rm mod}\ 3), \ mn\equiv6,12\ ({\rm mod}\ 24);\\
               J(m\times n,3,1),
               &\rm{otherwise}.
          \end{array}
       \right.
$$
\end{Lemma}

\proof Combine the results of Lemmas \ref{bd1}, \ref{bd:n=2} and \ref{bd:n=4} to complete the proof. \qed

\section{$2$-D $(m\times n,k,1)$-OOCs with ``holes''}

In this section, we introduce two special types of $2$-D $(m\times n,k,1)$-OOCs with ``holes'' to establish recursive constructions for $2$-D $(m\times n,k,1)$-OOCs.

\subsection{$2$-D $([m:r]\times n,k,1)$-OOCs}

Let $r>0$ and $R$ be any $r$-subset of $I_m$. If a $2$-D $(m\times n,k,1)$-OOC, $\cal B$, defined on $I_m\times\mathbb{Z}_n$, satisfies that for any $i,j\in R$, $\Delta_{ij}({\cal B})=\emptyset$, then we write $\cal B$ as a $2$-D $([m:r]\times n,k,1)$-OOC. For convenience, the set $R\times\mathbb{Z}_n$ is called {\em the hole} of the OOC. We also regard a $2$-D $(m\times n,k,1)$-OOC as a $2$-D $([m:0]\times n,k,1)$-OOC.
It is readily checked that the number of codewords of a
$2$-D $([m:r]\times n,k,1)$-OOC for $n\equiv0\pmod2$ does not exceed
$$\Theta(m,r,n,k):=\left\lfloor\frac{[m(m-1)-r(r-1)]n+(m-r)(n-2)}{k(k-1)}\right\rfloor=\left\lfloor\frac{(m^2-r^2)n-2(m-r)}{k(k-1)}\right\rfloor.$$

The following ``filling in hole" construction is straightforward.

\begin{Construction}\label{filling2}
Suppose that there exist
\begin{enumerate}
\item[$(1)$] a $2$-D $([m:r]\times n,k,1)$-OOC with $b_1$ codewords;
\item[$(2)$] a $2$-D $(r\times n,k,1)$-OOC with $b_2$ codewords.
\end{enumerate}
Then there exists a $2$-D $(m\times n,k,1)$-OOC with $b_1+b_2$ codewords.
\end{Construction}

We remark that Construction \ref{filling2} is the main tool of this paper.
In order to construct more $2$-D $([m:r]\times n,k,1)$-OOCs, in the next subsection we introduce our second type of $2$-D OOCs with ``holes''.

\subsection{$g$-regular $2$-D $([m:r]\times n,k,1)$-OOCs}

Let $H$ be the subgroup of order $g$ of $\mathbb{Z}_n$ and $R$ be any $r$-subset of $I_m$. If a $2$-D $(m\times n,k,1)$-OOC, $\mathcal{B}$, defined on $I_m\times\mathbb{Z}_n$, satisfies that
$$\Delta_{ij}(\mathcal{B})=\left\{
\begin{array}{lll}
\mathbb{Z}_{n}\setminus H, & \ \ {\rm if}\ \{i,j\}\not\subseteq R,\\
\emptyset, & \ \ {\rm otherwise},\\
\end{array}
\right.
$$
then we write $\cal B$ as a $g$-regular $2$-D $([m:r]\times n,k,1)$-OOC. For convenience, the set $R\times\mathbb{Z}_n$ is called {\em the hole}
and $I_m\times H$ is called {\em the forbidden set}. It is readily checked that the number of codewords of a
$g$-regular $2$-D $([m:r]\times n,k,1)$-OOC is
$$\Upsilon(m,r,n,g,k):=\frac{[m(m-1)-r(r-1)](n-g)+(m-r)(n-g)}{k(k-1)}=\frac{(m^2-r^2)(n-g)}{k(k-1)}.$$

When $R=\emptyset$, a $g$-regular $2$-D $([m:0]\times n,k,1)$-OOC will be denoted by a $g$-regular $2$-D $(m\times n,k,1)$-OOC.

\begin{Example}\label{m=6,2-regular}
There exists a $2$-regular $2$-D $(6\times n, 3, 1)$-OOC with $\Upsilon(6,0,n,2,3)$ codewords for any $n\equiv2\pmod{4}$ and $n\geq6$.
\end{Example}

\proof The required codes are constructed on $\mathbb{Z}_6\times\mathbb{Z}_n$. We list $n-2$ initial codewords below:

\begin{center}
\begin{tabular}{lll}
$\{(0,0),(1,i+1),(3,2+2i)\}$,
&$\{(0,0),(0,2i+2),(2,n/2+i+1)\}$, \\
$\{(0,0),(1,(n+2)/4+i),(3,1+2i)\}$,
&$\{(0,0),(0,2i+1),(1,(3n+2)/4+i)\}$,
\end{tabular}
\end{center}
where $0\leq i\leq (n-6)/4$. All the $\Upsilon(6,0,n,2,3)=6(n-2)$ codewords are obtained by developing these initial codewords by $(+1\pmod 6, -)$. \qed

\begin{Construction}\label{filling3}
Suppose that there exist
\begin{enumerate}
\item[$(1)$] a $g$-regular $2$-D $([m:r]\times n,k,1)$-OOC with $\Upsilon(m,r,n,g,k)$ codewords;
\item[$(2)$] a $2$-D $([m:r]\times g,k,1)$-OOC with $b$ codewords.
\end{enumerate}
Then there exists a $2$-D $([m:r]\times n,k,1)$-OOC with $\Upsilon(m,r,n,g,k)+b$ codewords.
\end{Construction}

\begin{proof}
Let $\mathcal{A}_1$ be a $g$-regular $2$-D $([m:r]\times n,k,1)$-OOC with $\Upsilon(m,r,n,g,k)$ codewords defined on $I_{m}\times\mathbb{Z}_{n}$ with $I_{r}\times\mathbb{Z}_{n}$ as the hole and $I_m\times H$ as the forbidden set, where $H$ is the subgroup of $\mathbb{Z}_n$ of order $g$.

Let $\mathcal{F}$ be a $2$-D $([m:r]\times g,k,1)$-OOC with $b$
codewords defined on $I_{m}\times\mathbb{Z}_g$ with $I_{r}\times\mathbb{Z}_g$ as the hole. For
each $D=\{(a_1,x_1),(a_2,x_2),\ldots,(a_k,x_k)\}\in\mathcal{F}$, let
$$D'=\{(a_1,x_1\cdot\frac{n}{g}),(a_2,x_2\cdot\frac{n}{g}),\ldots,(a_k,x_x\cdot\frac{n}{g})\}.$$ Let
$\mathcal{A}_2=\bigcup_{D\in\mathcal{F}}D'$.
Then $\mathcal{A}_1\bigcup\mathcal{A}_2$
forms the required $2$-D $([m:r]\times n,k,1)$-OOC, which is defined on $I_{m}\times\mathbb{Z}_{n}$ with $I_{r}\times\mathbb{Z}_{n}$ as the hole.
\end{proof}

The following example illustrates how to use Construction \ref{filling3}.

\begin{Example}\label{(6,n)}
There exists an optimal $2$-D $(6\times n,3,1)$-OOC with $6n-3$ codewords for any positive integer $n\equiv 2\pmod{4}$.
\end{Example}

\begin{proof}
For $n=2$, we list $3$ initial codewords on $\mathbb{Z}_6\times\mathbb{Z}_2$ below:
\begin{center}
\begin{tabular}{lll}
$\{(0,0),(1,0),(2,0)\}$,
&$\{(0,0),(3,0),(2,1)\}$,
&$\{(0,0),(3,1),(5,1)\}$.
\end{tabular}
\end{center}
All $9$ codewords are obtained by developing the initial codewords by $(+2\pmod 6, -)$. For any $n\equiv2\pmod{4}$ and $n\geq6$,
by Example \ref{m=6,2-regular}, there exists a $2$-regular $2$-D $(6\times n, 3, 1)$-OOC with $6(n-2)$ codewords.
Apply Construction \ref{filling3} with the aforementioned $2$-D $(6\times 2, 3, 1)$-OOC to obtain a $2$-D $(6\times n,3,1)$-OOC with $6n-3$ codewords, which is optimal by Lemma \ref{bd}.
\end{proof}

\subsection{More examples of $g$-regular $2$-D $([m:r]\times n,k,1)$-OOCs}

We introduce the concept of frame starters to present some more examples of $g$-regular $2$-D OOCs for immediately use in the next section.
We will still discuss recursive constructions for these codes in Section \ref{g-regular [m:r]}.

Let $G$ be an additive abelian group of order $n$ and $H$ a subgroup of $G$ of order $g$. A \emph{frame starter} of type $g^{n/g}$ in $G\setminus H$ is a set of $(n-g)/2$ pairs $\{\{x_i,y_i\}:1\leq i\leq(n-g)/2\}$ that satisfies the following two properties:

\begin{enumerate}
\item[$(1)$] $\bigcup_{i=1}^{(n-g)/2}\{x_i,y_i\}=G\setminus H$;
\item[$(2)$] $\{\pm(x_i-y_i):1\leq i\leq(n-g)/2\}=G\setminus H$.
\end{enumerate}


\begin{Lemma}\rm{\cite{rs}}\label{fs 2^n}
There exists a frame starter of type $2^{n/2}$ in $\mathbb{Z}_{n}\setminus\{0,n/2\}$ for any $n\equiv0,2\pmod{8}$.
\end{Lemma}

\begin{Example}\label{2regular [2:1]}
There exists a $2$-regular $2$-D $([2:1]\times n,3,1)$-OOC with $\Upsilon(2,1,n,2,3)$ codewords for $n\equiv0,2\pmod8$ and $n\geq8$.
\end{Example}

\begin{proof}
Let $H$ be a subgroup of $\mathbb{Z}_n$ of order $2$. By Lemma \ref{fs 2^n}, there exists a frame starter of type $2^{n/2}$ in $\mathbb{Z}_n\setminus H$. Let $S$ be the given frame starter. It is readily checked that $$\mathcal{B}=\{\{(0,0),(1,x),(1,y)\}:\{x,y\}\in S\}.$$
is a $2$-regular $2$-D $([2:1]\times n,3,1)$-OOC with $\Upsilon(2,1,n,2,3)=(n-2)/2$ codewords.
\end{proof}

\begin{Example}\label{m=3,regular}
There exists a $2$-regular $2$-D $(3\times n,3,1)$-OOC with $\Upsilon(3,0,n,2,3)$ codewords for $n\equiv0,2\pmod8$ and $n\geq8$.
\end{Example}

\begin{proof}
Let $H$ be a subgroup of $\mathbb{Z}_n$ of order $2$. By Lemma \ref{fs 2^n}, there exists a frame starter of type $2^{n/2}$ in $\mathbb{Z}_n\setminus H$. Let $S$ be the given frame starter. It is readily checked that $$\mathcal{B}=\{\{(i,0),(i+1,x),(i+1,y)\}:i\in\mathbb{Z}_3,\{x,y\}\in S\}.$$
is a $2$-regular $2$-D $(3\times n,3,1)$-OOC on $\mathbb{Z}_3\times\mathbb{Z}_n$ with forbidden set $\mathbb{Z}_3\times H$.
\end{proof}

\section{Auxiliary designs}

In order to apply Constructions \ref{filling2} and \ref{filling3}, we need to
construct $2$-D $([m:r], n, k,1)$-OOCs and $g$-regular $2$-D $([m:r]\times n,k,1)$-OOCs. To this end, we introduce three types of auxiliary designs: $n$-cyclic group divisible packings, $wt$-cyclic holey group divisible packings and semi-cyclic incomplete holey group divisible designs.

\subsection{$n$-cyclic group divisible packings}

Fix a set of positive integers $K$. A {\em $K$-group divisible packing} (or $K$-GDP) is a triple ($X, {\cal G},{\cal B}$) satisfying that (1) $\cal G$ is a partition of a finite set $X$ into subsets (called {\em groups}); (2) $\cal B$ is a set of subsets of $X$ (called {\em blocks}), each of size from $K$, such that every $2$-subset of $X$ is either contained in at most one block or in exactly one group, but not in both. Furthermore, if every two points from different groups appear in exactly one block, then the GDP is often called a {\em group divisible design} and denoted by a GDD.
The multi-set $T=\{|G|: G\in\mathcal{G}\}$ is called the \emph{group type} (or \emph{type}) of the GDP. We will use an ``exponential" notation to describe group types: type $g_1^{m_1}g_2^{m_2}\cdots g_r^{m_r}$ indicates that there are $m_i$ groups of size $g_i$ for $1\leq i\leq r$. If $K=\{k\}$, we simply write $k$-GDP instead of $K$-GDP.

We employ the pure and mixed difference method to construct a $k$-GDD of type $g^{u}$. Let $X=\mathbb{Z}_{gu}$ and $\mathcal{G}=\{\{i,u+i,\ldots,(g-1)u+i\}:0\leq i\leq u-1\}$. Let ${\cal B}^*$ be a set of $k$-subsets (called {\em base blocks}) of $X$. Define a multi-set
$\Delta(B)=\{x-y \pmod{gu}: x,y\in B, x\neq y\}$, and a multi-set $$\Delta(\mathcal{B}^*)=\bigcup_{B\in\mathcal{B}^*}\Delta(B).$$
If $\Delta(\mathcal{B}^*)=\mathbb{Z}_{gu}\setminus\{0,u,\ldots,(g-1)u\}$, then a $k$-GDD of type $g^u$ with the point set $X$ together with the group set $\cal G$ can be generated from ${\cal B}^*$. The required blocks are obtained by developing all base blocks of ${\cal B}^*$ by successively
adding $1$ to each point of these base blocks modulo $gu$. Usually a $k$-GDD obtained by this manner is said to be \emph{strictly-cyclic}.

\begin{Lemma}\label{3-sCGDD}{\rm\cite{wxc}}
There exists a strictly cyclic $3$-GDD of type $g^u$ if and only if
\begin{enumerate}
\item[$(1)$] $g(u-1)\equiv 0\ ({\rm mod}\ 6)$ and $u\geq 4;$
\item[$(2)$] $u\not\equiv 2,3\ ({\rm mod}\ 4)$ when $g\equiv 2\ ({\rm mod}\ 4).$
\end{enumerate}
\end{Lemma}

By the definitions of $g$-regular $2$-D OOCs and strictly-cyclic GDDs, it is readily checked that a $g$-regular $2$-D $(1\times n,k,1)$-OOC is equivalent to a strictly cyclic $k$-GDD of type $g^{n/g}$. Hence, we have the following corollary.

\begin{Corollary}\label{3-sCGDD to ooc}
There exists a $g$-regular $2$-D $(1\times n,3,1)$-OOC with $\Upsilon(1,0,n,g,3)$ codewords if and only if
\begin{enumerate}
\item[$(1)$] $n-g\equiv 0\ ({\rm mod}\ 6)$ and $n\geq 4g;$
\item[$(2)$] $n/g\not\equiv 2,3\ ({\rm mod}\ 4)$ when $g\equiv 2\ ({\rm mod}\ 4).$
\end{enumerate}
\end{Corollary}

We can also employ the pure and mixed difference method to construct a $K$-GDP of type $(v_{1}n)^{m_1} (v_{2}n)^{m_2}\cdots (v_{r}n)^{m_r}$. Let $\{R_{ij}: 1\leq i\leq r, 1\leq j\leq m_i\}$ be a partition of $I_m$ with $|R_{ij}|=v_i$.
Let $X=I_m\times {\mathbb Z}_{n}$ and $\mathcal{G}=\{R_{ij}\times\mathbb{Z}_{n}:1\leq i\leq r,1\leq j\leq m_i\}$. Let ${\cal B}^*$ be a set of subsets (called {\em base blocks}) of $X$.
If for any $(a,b)\in I_m\times I_m$,

$$\Delta_{ab}(\mathcal{B}^*)\subseteq\left\{
\begin{array}{lll}
{\mathbb Z}_{n}, & \ \ {\rm if}\ \{a,b\}\not\subseteq R_{ij}\ {\rm for\ all}\ 1\leq i\leq r \ {\rm and}\ 1\leq j\leq m_i,\\
\emptyset, & \ \ {\rm otherwise},\\
\end{array}
\right.
$$
then a $K$-GDP of type $(v_{1}n)^{m_1} (v_{2}n)^{m_2}\cdots (v_{r}n)^{m_r}$ with the point set $X$ together with the group set $\cal G$ can be generated from ${\cal B}^*$. The required blocks are obtained by developing all base blocks of ${\cal B}^*$ by successively
adding $1$ to the second component of each point of these base blocks modulo $n$. Usually a $K$-GDP obtained by this manner is said to be \emph{$n$-cyclic}. Naturally we have the notion of an $n$-cyclic group divisible design (GDD).

\begin{Example}\label{6^32^1}
There exists a $2$-cyclic $3$-GDD of type $6^32^1$.
\end{Example}

\proof Let $X=(I_{9}\bigcup\{\infty\})\times\mathbb{Z}_2$ with $\mathcal{G}=\{\{i,i+3,i+6\}\times\mathbb{Z}_2:0\leq i\leq 2\}\bigcup\{\{\infty\}\times\mathbb{Z}_2\}$. All the $24$ base blocks are listed below.

\begin{center}
\begin{tabular}{lllll}
$\{(0,0),(1,0),(2,0)\}$,
&$\{(0,0),(\infty,0),(1,1)\}$,
&$\{(0,0),(7,0),(8,0)\}$,
&$\{(0,0),(4,0),(5,0)\}$,\\
$\{(0,0),(2,1),(4,1)\}$,
&$\{(0,0),(8,1),(\infty,1)\}$,
&$\{(0,0),(5,1),(7,1)\}$,
&$\{(1,0),(3,0),(5,0)\}$,\\
$\{(1,0),(6,0),(8,0)\}$,
&$\{(1,0),(\infty,0),(2,1)\}$,
&$\{(1,0),(3,1),(8,1)\}$,
&$\{(1,0),(5,1),(6,1)\}$,\\
$\{(2,0),(3,0),(7,0)\}$,
&$\{(2,0),(\infty,0),(3,1)\}$,
&$\{(2,0),(6,0),(7,1)\}$,
&$\{(5,0),(7,1),(\infty,1)\}$,\\
$\{(3,0),(4,0),(5,1)\}$,
&$\{(3,0),(\infty,0),(7,1)\}$,
&$\{(3,0),(4,1),(8,1)\}$,
&$\{(4,0),(\infty,0),(8,1)\}$,\\
$\{(6,0),(7,0),(8,1)\}$,
&$\{(5,0),(\infty,0),(6,1)\}$,
&$\{(2,0),(4,1),(6,1)\}$,
&$\{(4,0),(6,1),(\infty,1)\}$.
\end{tabular}
\end{center}

\begin{Remark}\label{n-cyclic gdd and ooc}
Recall that $2$-D OOCs can be viewed from the set-theoretic perspective. According to the pure and mixed difference method, each base block of an $n$-cyclic $k$-GDP of type $(v_{1}n)^{m_1} (v_{2}n)^{m_2}\cdots (v_{r}n)^{m_r}$ $($defined on $I_m\times {\mathbb Z}_{n}$, where $m=\sum_{i=1}^r v_i m_i$$)$ can be seen as a codeword of a $2$-D $(m\times n,k,1)$-OOC. Actually by the definition of $n$-cyclic GDPs, an $n$-cyclic GDP cannot produce pure $(i,i)$-differences for any $i\in I_m$, and each of its mixed $(i,j)$-differences belongs to $\mathbb{Z}_{n}$ for any $i\neq j$ occurs at most once, which coincide with the condition in Definition \ref{def-diff}.
\end{Remark}

\begin{Lemma}\rm{\cite[Corollary 2.7]{wc}}\label{h-cyclic gdd}
There is an $n$-cyclic $3$-GDD of type $(vn)^m$ if and only if $(1)$ when $m=3$, $n$ is odd, or $n$ is even and $v$ is even; $(2)$ when $m\geq 4$, $(m-1)vn\equiv0 \pmod{2}$,
$m(m-1)vn\equiv0 \pmod{3}$, and
$v\equiv0 \pmod{2}$ if
$m\equiv2,3 \pmod{4}$ and
$n\equiv2 \pmod{4}$.
\end{Lemma}

\begin{Lemma}\rm{\cite[Theorem 5.18]{wc3}}\label{3-GDP}
There is an $n$-cyclic $3$-GDP of type $(vn)^{m}$ with
$\lfloor mv/3\lfloor vn(m-1)/2\rfloor\rfloor-\gamma$ base blocks for any positive
integers $v,n$ and $m\geq3$, where
$$
    \gamma=\left\{
          \begin{array}{ll}
               1, & {\rm{if}}\
               v\equiv1\hspace{-0.3cm}\pmod2,
               n\equiv0\hspace{-0.3cm}\pmod4,\\ & ~~~{\rm{and}}~ m=3,\\
                &~\hspace{-0.12cm}{\rm{or}}~v\equiv1\hspace{-0.3cm}\pmod2,n\equiv2\hspace{-0.3cm}\pmod4,
                \\ & ~~~{\rm{and}}~ m\equiv3,6,7,10\hspace{-0.3cm}\pmod{12},\\
                &~\hspace{-0.12cm}{\rm{or}}~vn\equiv6\hspace{-0.3cm}\pmod{12},n\equiv2\hspace{-0.3cm}\pmod{4},
                \\ & ~~~{\rm{and}}~ m\equiv2,11\hspace{-0.3cm}\pmod{12},\\
                &~\hspace{-0.12cm}{\rm{or}}~(m-1)v^2n\equiv4\hspace{-0.3cm}\pmod{6},
                \\ & ~~~{\rm{and}}~ m\equiv2\hspace{-0.3cm}\pmod{3},\\
                &~\hspace{-0.12cm}{\rm{or}}~n=2,m(m-1)v^2\equiv8\hspace{-0.3cm}\pmod{12},\\

               0, & \rm{otherwise}.
          \end{array}
       \right.
$$
\end{Lemma}

\begin{Construction}\label{SCIGDD-weighting}{\rm \cite[Lemma 4.8]{fwwzh}}
Suppose that there exists an $n$-cyclic $K$-GDD of type
$(v_{1}n)^{m_1} (v_{2}n)^{m_2}\cdots (v_{r}n)^{m_r}$. If there exists an $h$-cyclic $l$-GDD of type $(gh)^k$ for each $k\in K$, then there exists an $(hn)$-cyclic $l$-GDD of type ${(gv_{1}hn)}^{m_1}\cdots {(gv_{r}hn)}^{m_r}$.
\end{Construction}

The following construction is a minor modification of \cite[Construction 4.19]{wc}.

\begin{Construction}\label{GDD-filling}{\rm \cite[Construction 4.19]{wc}}
Suppose that there exists an $n$-cyclic $k$-GDD of type
$(v_{1}hn)^{m_1} (v_{2}hn)^{m_2}\cdots (v_{r}hn)^{m_r}$. If there exists an $n$-cyclic $k$-GDD of type $(hn)^{v_{i}}(tn)^1$ for each $1\leq i\leq r$, then there exists an $n$-cyclic $k$-GDD of type $(hn)^{\sum_{i=1}^rv_im_i}(tn)^1$.
\end{Construction}

\begin{Lemma}\label{3n^3n^1}
There exists an $n$-cyclic $3$-GDD of type $(3n)^{3}n^1$ for any positive integer $n\equiv2\pmod4$.
\end{Lemma}

\proof By Example \ref{6^32^1}, there exists a $2$-cyclic $3$-GDD of type $6^{3}2^1$. Apply Construction \ref{SCIGDD-weighting} with an $(n/2)$-cyclic $3$-GDD of type $(n/2)^3$ (from Lemma \ref{h-cyclic gdd}) to obtain an $n$-cyclic $3$-GDD of type $(3n)^{3}n^1$.
\qed

\begin{Lemma}\label{3n^4t9n^1}
Let $m\equiv r\pmod{12}$ where $r\in\{1,2,5,8,9\}$ and $m\geq13$. Then there exists an $n$-cyclic $3$-GDD of type $(3n)^{(m-r)/3}(rn)^1$ for any $n\equiv2\pmod4$.
\end{Lemma}

\proof Let $m=12t+r$. We first construct a $2$-cyclic $3$-GDD of type $6^{4t}(2r)^1$. For $t\in\{1,2\}$,
see Appendix \ref{6^4t12^1:t=1,2}. For $t\geq3$, apply Construction \ref{GDD-filling} with a $2$-cyclic $3$-GDD of type $24^t$ (from Lemma \ref{h-cyclic gdd}) and a $2$-cyclic $3$-GDD of type $6^{4}(2r)^1$ to obtain a $2$-cyclic $3$-GDD of type $6^{4t}(2r)^1$. Further apply Construction \ref{SCIGDD-weighting} with an $(n/2)$-cyclic $3$-GDD of type $(n/2)^3$ (from Lemma \ref{h-cyclic gdd}) to obtain an $n$-cyclic $3$-GDD of type $(3n)^{4t}(rn)^1$.
\qed

\begin{Lemma}\label{scigdd}
There exists a $2$-cyclic $3$-GDD of type $2^{m-r}(2r)^1$ for any $m\equiv r\pmod{12}$, $m\geq 13$ and $r\in\{1,2,4,5,7,8,10,11\}$.
\end{Lemma}

\proof
For $r\in\{1,2,4,5,7,10,11\}$, see \cite[Lemma 4.22]{wc}, \cite[Lemma 4.25]{wc3} and \cite[Lemma 4.11]{fwwzh}. Note that a $2$-cyclic $3$-GDD of type $2^{m-r}(2r)^1$ is often called a {\em semi-cyclic incomplete GDD} (briefly $3$-SCIGDD) of type $2^{(m,r)}$.

For $r=8$ and $m\in\{20,32\}$, see Appendix \ref{2^12t16^1:t=1,2}. For $m\equiv8\pmod{12}$ and $m\geq44$, apply Construction \ref{GDD-filling}
with a $2$-cyclic $3$-GDD of type $24^{(m-8)/12}$ (from Lemma \ref{h-cyclic gdd})
and a $2$-cyclic $3$-GDD of type $2^{12}16^1$ to obtain a $2$-cyclic $3$-GDD of type $2^{m-8}16^1$.
\qed

\subsection{$wt$-cyclic holey group divisible designs}

A {\em holey group divisible design} is a quadruple $(X,{\cal H},{\cal G},{\cal B})$ satisfying that (1) $X$ is a set of $mwt$ points;
(2) $\mathcal{G}$ is a partition of $X$ into $m$ subsets (called
\emph{groups}), each of size $wt$; (3) $\mathcal{H}$ is another partition of $X$ into $t$ subsets (called \emph{holes}), each of size $mw$ such that $|H\cap G|=w$
for each $H\in {\cal H}$ and $G\in {\cal G}$; (4) $\mathcal{B}$ is a collection of $k$-subsets of $X$ (called
\emph{blocks}), such that no block contains two distinct points of any group or any hole, but any other pair of distinct points of $X$ occurs in exactly one block of $\cal B$. Such a design is denoted by a $k$-HGDD of type $(m,w^t)$.

\begin{Lemma}\rm{\cite{w}}\label{hgdd}
There exists a $3$-HGDD of type $(m,w^t)$ if and
only if $m,t\geq 3$, $(m-1)(t-1)w\equiv0\ ({\rm mod}\ 2)$, and
$m(m-1)t(t-1)w^2\equiv0\ ({\rm mod}\ 6)$.
\end{Lemma}

We can employ the pure and mixed difference method to construct a $k$-HGDD of type $(m,(hw)^t)$. Let $H=\{0,t,\ldots,(w-1)t\}$ be the subgroup of order $w$ in $\mathbb{Z}_{wt}$, and $H_l=H+l$ be a coset of $H$ in $Z_{wt}$, $0\leq l\leq t-1$. Let $R_i$, $1\leq i\leq m$, be $m$ pairwise disjoint sets with $|R_i|=h$. Let $X=(\bigcup_{i=1}^{m}R_i)\times {\mathbb Z}_{wt}$, $\mathcal{G}=\{R_i\times {\mathbb Z}_{wt}:1\leq i\leq m\}$ and $\mathcal{H}=\{(\bigcup_{i=1}^{m}R_i)\times H_l:0\leq l\leq t-1\}$. Let ${\cal B}^*$ be a set of $k$-subsets (called {\em base blocks}) of $X$.
If for any $(a,b)\in (\bigcup_{i=1}^{m}R_i)\times(\bigcup_{i=1}^{m}R_i)$,

$$\Delta_{ab}(\mathcal{B}^*)=\left\{
\begin{array}{lll}
{\mathbb Z}_{wt}\setminus H, & \ \ {\rm if}\ \{a,b\}\not\subseteq R_i\ {\rm for\ all}\ 1\leq i\leq m,\\
\emptyset, & \ \ {\rm otherwise},\\
\end{array}
\right.
$$
then a $k$-HGDD of type $(m,(hw)^t)$ with the point set $X$ together with the group set $\cal G$ and the hole set $\cal H$ can be generated from ${\cal B}^*$. The required blocks are obtained by developing all base blocks of ${\cal B}^*$ by successively
adding $1$ to the second component of each point of these base
blocks modulo $wt$. Usually a $k$-HGDD obtained by this manner is said to be \emph{$wt$-cyclic}.
Note that each base block of ${\cal B}^*$ contributes $k(k-1)$ mixed differences, and since $|\Delta_{ab}({\cal B}^*)|=wt-w$ for any given $(a,b)\in (\bigcup_{i=1}^{m}R_i)\times(\bigcup_{i=1}^{m}R_i)$ and $\{a,b\}\not\subset R_i$ for $1\leq i\leq m$, the total number of mixed differences produced by ${\cal B}^*$ is $m(m-1)h^2(wt-w)$. Hence $|{\cal B}^*|=m(m-1)h^2w(t-1)/(k^2-k)$.



\begin{Construction}\label{regular OOC to HGDD}
Suppose that there exist
\begin{enumerate}
\item[$(1)$] a $w$-regular $2$-D $(h\times wt,k,1)$-OOC with $\Upsilon(h,0,wt,w,k)$ codewords, and
\item[$(2)$] a $k$-HGDD of type $(m,e^k)$ with $m(m-1)e^2$ blocks.
\end{enumerate}
Then there exists a $wt$-cyclic $k$-HGDD of type $(m,(hew)^t)$ with $m(m-1)(he)^2w(t-1)/(k^2-k)$ base blocks.
\end{Construction}

\begin{proof}
Let $\mathcal{B}$ be a $w$-regular $2$-D $(h\times wt,k,1)$-OOC with $\Upsilon(h,0,wt,w,k)=h^2w(t-1)/(k^2-k)$ codewords on $I_h\times\mathbb{Z}_{wt}$. For each $B\in\mathcal{B}$, construct a $k$-HGDD of type $(m,e^k)$ with $m(m-1)e^2$ blocks on $I_m\times I_e\times B$ with the group set $\{\{i\}\times I_e\times B: i\in I_m\}$ and the hole set $\{I_m\times I_e\times\{x\}:x\in B\}$. Denote its block set by $\mathcal{A}_B$.
It is readily checked that $\bigcup_{B\in\mathcal{B}}\mathcal{A}_B$ is
a $wt$-cyclic $k$-HGDD of type $(m,(hew)^t)$ on $I_m\times I_e\times I_h\times\mathbb{Z}_{wt}$ with the group set $\{\{i\}\times I_e\times I_h\times\mathbb{Z}_{wt}: i\in I_m\}$ and the hole set $\{I_m\times I_e\times I_h\times\{i,i+t,\ldots,i+(w-1)t\}:0\leq i\leq t-1\}$.
\end{proof}

\begin{Corollary}\label{h to mh}
Let $m$ and $e$ be positive integers with $m\geq3$. If there exists a $w$-regular $2$-D $(h\times wt,3,1)$-OOC with $\Upsilon(h,0,wt,w,3)$ codewords, then there exists a $wt$-cyclic $3$-HGDD of type $(m,(hew)^t)$.
\end{Corollary}

\begin{proof}
By Lemma \ref{hgdd}, there exists a $3$-HGDD of type $(m,e^3)$ for any any positive integer $m\geq3$. Then apply Constructions \ref{regular OOC to HGDD} to obtain a $wt$-cyclic $k$-HGDD of type $(m,(hew)^t)$ as desired.
\end{proof}

\begin{Lemma}\label{cyclic 3hgdd}
Let $m$ and $e$ be positive integers with $m\geq3$. Then there exists an $n$-cyclic $3$-HGDD of type $(m,(3e\times g)^{n/g})$ for $g=2$, $n\equiv0,2\pmod8$ and $n\geq8$.
\end{Lemma}

\begin{proof}
By Example \ref{m=3,regular}, there exists a $g$-regular $2$-D $(3\times n,3,1)$-OOC with $\Upsilon(3,0,n,g,3)$ codewords. Then apply Corollary \ref{h to mh} with $h=3$, $w=g$ and $t=n/g$.
\end{proof}

\subsection{Semi-cyclic incomplete holey group divisible designs}

An {\em incomplete holey group divisible design} is a quintuple $(X,Y,{\cal H},{\cal G},{\cal B})$ satisfying that (1) $X$ is a set of $mwt$ points;
(2) $\mathcal{G}$ is a partition of $X$ into $m$ subsets (called
\emph{groups}), each of size $wt$; (3) $\mathcal{H}$ is another partition of $X$ into $t$ subsets (called \emph{holes}), each of size $mw$ such that $|H\cap G|=w$
for each $H\in {\cal H}$ and $G\in {\cal G}$; (4) $Y$ is an union of $r$ groups of $\mathcal{G}$; (5) $\mathcal{B}$ is a collection of $k$-subsets of $X$ (called
\emph{blocks}), such that no block contains two distinct points of any group, any hole or $Y$, but any other pair of distinct points of $X$ occurs in exactly one block of $\cal B$. Such a design is denoted by a $k$-IHGDD of type $(m,r,w^t)$. When $r\in\{0,1\}$, an IHGDD of type $(m,r,w^t)$ is nothing but an HGDD of type $(m,w^t)$.

We can also employ the pure and mixed difference method to construct $k$-SCIHGDDs of type $(m,r,w^t)$. Let $H=\{0,t,\ldots,(w-1)t\}$ be the subgroup of order $w$ in $\mathbb{Z}_{wt}$, and $H_l=H+l$ be a coset of $H$ in $\mathbb{Z}_{wt}$, $0\leq l\leq t-1$. Let $X=I_m\times \mathbb{Z}_{wt}$, $Y=I_r\times \mathbb{Z}_{wt}$, ${\cal G}=\{\{i\}\times \mathbb{Z}_{wt}:i\in I_m\}$, and ${\cal H}=\{I_m\times H_l:0\leq l\leq t-1\}$. Take a family ${\cal B}^*$ of some $k$-subsets (base blocks) of $X$. If for $(i,j)\in I_m\times I_m$,
$$\Delta_{ij}({\cal B}^*)=\left\{
\begin{array}{lll}
\mathbb{Z}_{wt}\setminus H, & \ \ i\neq j\ {\rm and}\ \{i,j\}\not\subset I_r,\\
\emptyset, & \ \ {\rm otherwise},\\
\end{array}
\right.
$$
then a $k$-IHGDD of type $(m,r,w^t)$ with the point set $X$ together with $Y\subset X$, the group set $\cal G$ and the hole set $\cal H$ can be generated from ${\cal B}^*$. All blocks can be obtained by developing all base blocks of ${\cal B}^*$ by successively
adding $1$ to the second component of each point of these base
blocks modulo $wt$. Usually a $k$-IHGDD obtained by this
manner is said to be \emph{semi-cyclic}. Note that each base block of ${\cal B}^*$ contributes $k(k-1)$ mixed differences, and since $|\Delta_{ij}({\cal B}^*)|=wt-w$ for any given $(i,j)\in I_m\times I_m$, $i\neq j$ and $\{i,j\}\not\subset I_r$, the total number of mixed differences produced by ${\cal B}^*$ is $(m(m-1)-r(r-1))(wt-w)$. Hence $|{\cal B}^*|=(m(m-1)-r(r-1))w(t-1)/(k^2-k)$. Naturally we have the notion of a semi-cyclic holey group divisible design (SCHGDD).

\begin{Lemma}\label{schgdd}\rm{\cite{fwc,fwwei,wfpw}}
There exists a $3$-SCHGDD of type $(m,w^t)$ if and only if $m,t\geq
3$, $(t-1)(m-1)w\equiv 0\pmod {2}$ and $(t-1)m(m-1)w\equiv 0\pmod6$
except when
\begin{enumerate}
\item[$(1)$] $m=t=3$, $w\equiv0\pmod2$;
\item[$(2)$] $(m,w,t)\in\{(5,1,4),(6,1,3)\}$;
\item[$(3)$] $m=3$, $w\equiv1\pmod2$ and $t\equiv 0\pmod2$;
\item[$(4)$] $m\equiv3,7\pmod{12}$, $w\equiv1\pmod2$ and $t\equiv2\pmod4$;
\item[$(5)$] $m\equiv11\pmod{12}$, either $w\equiv3\pmod6$ and $t\equiv2\pmod4$, or $w\equiv1,5\pmod6$ and
$t\equiv10\pmod{12}$.
\end{enumerate}
\end{Lemma}

The following construction is a special case of \cite[Construction 4.3]{wc}.

\begin{Construction}\label{inflate}
Suppose that there exist a $k$-SCHGDD of type $(m,w^{t})$ and a $k$-GDD of type $h^k$.
Then there exists a $wt$-cyclic $k$-HGDD of type $(m,(hw)^{t})$.
\end{Construction}

\begin{Lemma}\label{inflate schgdd}
Let $(m,h)\in\{(3,3),(4,4),(7,2)\}$. Then there exists an $n$-cyclic $3$-HGDD of type $(m,(2h)^{n/2})$ for any $n\equiv0\pmod4$ and $n\geq8$.
\end{Lemma}

\proof
Apply Construction \ref{inflate} with a $3$-SCHGDD of type $(m,2^{n/2})$ (from Lemma \ref{schgdd}) and a $3$-GDD of type $h^3$ (from Lemma \ref{h-cyclic gdd}).
\qed

\begin{Lemma}\label{inflate schgdd1}
There exists an $n$-cyclic $3$-HGDD of type $(3,(3\times 4)^{n/4})$ for any $n\equiv4\pmod8$ and $n\geq20$.
\end{Lemma}

\proof
Apply Construction \ref{inflate} with a $3$-SCHGDD of type $(3,4^{n/4})$ (from Lemma \ref{schgdd}) and a $3$-GDD of type $3^3$ (from Lemma \ref{h-cyclic gdd}).
\qed

The following construction is a special case of \cite[Construction 2.15]{wfpw}.

\begin{Construction}\label{fill in cyclic hgdd}
Suppose that there exist a $wt$-cyclic $k$-HGDD of type $(m,(hw)^{t})$ and a $k$-SCIHGDD of type $(h+\delta,\delta,w^{t})$.
Then there exists a $k$-SCIHGDD of type $(mh+\delta,h+\delta,w^{t})$.
\end{Construction}

\begin{Lemma}\label{scihgdd1}
Let $m=6t+r$ where $r\in\{1,2,4,5\}$ and $t\in\{1,2\}$. Then there exists a $3$-SCIHGDD of type $(m-1,r-1,2^{n/2})$ for any $n\equiv0\pmod4$ and $n\geq8$.
\end{Lemma}

\begin{proof}

For $r\in\{1,2\}$, a $3$-SCIHGDD of type $(m-1,r-1,2^{n/2})$ is just a $3$-SCHGDD of type $(m-1,2^{n/2})$ from Lemma \ref{schgdd}.

For $r\in\{4,5\}$ and $t=1$, apply Construction \ref{fill in cyclic hgdd} with an $n$-cyclic $3$-HGDD of type $(3,(3\times2)^{n/2})$ (from Lemma \ref{inflate schgdd}) and a $3$-SCIHGDD of type $(3+r-1-3,r-1-3,2^{n/2})$ (from Lemma \ref{schgdd})
to obtain a $3$-SCIHGDD of type $(m-1,r-1,2^{n/2})$.

For $r=4$ and $t=2$, apply Construction \ref{fill in cyclic hgdd} with an $n$-cyclic $3$-HGDD of type $(7,(2\times2)^{n/2})$ (from Lemma \ref{inflate schgdd}) and a $3$-SCIHGDD of type $(3,1,2^{n/2})$ (from Lemma \ref{schgdd})
to obtain a $3$-SCIHGDD of type $(m-1,r-1,2^{n/2})$.

For $r=5$ and $t=2$, apply Construction \ref{fill in cyclic hgdd} with an $n$-cyclic $3$-HGDD of type $(4,(4\times2)^{n/2})$ (from Lemma \ref{inflate schgdd}) and a $3$-SCIHGDD of type $(4,0,2^{n/2})$ (from Lemma \ref{schgdd})
to obtain a $3$-SCIHGDD of type $(m-1,r-1,2^{n/2})$.
\end{proof}

\begin{Lemma}\label{scihgdd2}
Let $m\equiv r\pmod{12}$ where $r\in\{1,2,4,5,7,8,10,11\}$ and $m\geq13$. Then there exists a $3$-SCIHGDD of type $(m-1,r-1,2^{n/2})$ for any $n\equiv0,2\pmod8$ and $n\geq8$.
\end{Lemma}

\begin{proof}

For $r\in\{1,2\}$, a $3$-SCIHGDD of type $(m-1,r-1,2^{n/2})$ is just a $3$-SCHGDD of type $(m-1,2^{n/2})$ from Lemma \ref{schgdd}.

For $r\in\{4,5\}$, apply Construction \ref{fill in cyclic hgdd} with an $n$-cyclic $3$-HGDD of type $(\frac{m-r}{3}+1,(3\times2)^{n/2})$ (from Lemma \ref{cyclic 3hgdd}) and a $3$-SCIHGDD of type $(3+(r-1-3),r-1-3,2^{n/2})$ (from Lemma \ref{schgdd})
to obtain a $3$-SCIHGDD of type $(m-1,r-1,2^{n/2})$.

For $r\in\{7,8\}$, apply Construction \ref{fill in cyclic hgdd} with an $n$-cyclic $3$-HGDD of type $(\frac{m-r}{6}+1,(6\times2)^{n/2})$ (from Lemma \ref{cyclic 3hgdd}) and a $3$-SCIHGDD of type $(6+(r-1-6),r-1-6,2^{n/2})$ (from Lemma \ref{schgdd})
to obtain a $3$-SCIHGDD of type $(m-1,r-1,2^{n/2})$.

For $r\in\{10,11\}$, first apply Construction \ref{fill in cyclic hgdd} with an $n$-cyclic $3$-HGDD of type $(3,(3\times2)^{n/2})$ (from Lemma \ref{cyclic 3hgdd}) and a $3$-SCIHGDD of type $(3+(r-1-9),r-1-9,2^{n/2})$ (from Lemma \ref{schgdd})
to obtain a $3$-SCIHGDD of type $(6+3+(r-1-9),3+(r-1-9),2^{n/2})$. Further apply Construction \ref{fill in cyclic hgdd} with an $n$-cyclic $3$-HGDD of type $(\frac{m-r}{6}+1,(6\times2)^{n/2})$ (from Lemma \ref{cyclic 3hgdd}) to obtain a $3$-SCIHGDD of type $(m-1,r-1,2^{n/2})$.
\end{proof}

\section{Constructions for $g$-regular $2$-D $([m:r]\times n,k,1)$-OOCs}\label{g-regular [m:r]}

\begin{Construction}\label{inflationg:g-regular}
Suppose that there exist
\begin{enumerate}
\item[$(1)$] a $k$-GDD of type $h^k$, and
\item[$(2)$] a $g$-regular $2$-D $([m:r]\times n,k,1)$-OOC with $\Upsilon(m,r,n,g,k)$ codewords.
\end{enumerate}
Then there exists a $g$-regular $2$-D $([mh:rh]\times n,k,1)$-OOC with $\Upsilon(mh,rh,n,g,k)$ codewords.
\end{Construction}

\begin{proof}
Let $\mathcal{C}$ be a $g$-regular $2$-D $([m:r]\times n,k,1)$-OOC with $\Upsilon(m,r,n,g,k)$ codewords defined on $I_{m}\times\mathbb{Z}_{n}$ with $I_{r}\times\mathbb{Z}_{n}$ as the hole and $I_m\times H$ as the forbidden set, where $H$ is the subgroup of $\mathbb{Z}_n$ of order $g$.

For each $B\in\mathcal{C}$, construct a $k$-GDD of type $h^k$ on $I_h\times B$ with the group set $\{I_h\times\{x\}:x\in B\}$. Denote its block set by $\mathcal{A}_B$. It is readily checked that $\bigcup_{B\in\mathcal{C}}\mathcal{A}_B$ forms a $g$-regular $2$-D $([mh:rh]\times n,k,1)$-OOC with $h^2\Upsilon(m,r,n,g,k)=\Upsilon(mh,rh,n,g,k)$ codewords, which is defined on $I_{h}\times I_{m}\times\mathbb{Z}_{n}$ with $I_{h}\times I_{r}\times\mathbb{Z}_{n}$ as the hole and $I_{h}\times I_m\times H$ as the forbidden set.
\end{proof}

\begin{Lemma}\label{m=24t+6,6regular}
Let $m$ be a positive interger. There exists a $g$-regular $2$-D $(m\times n,3,1)$-OOC with $\Upsilon(m,0,n,g,3)$ codewords if the
following conditions hold:
\begin{enumerate}
\item[$(1)$] $n-g\equiv 0\ ({\rm mod}\ 6)$ and $n\geq 4g;$
\item[$(2)$] $n/g\not\equiv 2,3\ ({\rm mod}\ 4)$ when $g\equiv 2\ ({\rm mod}\ 4).$
\end{enumerate}
\end{Lemma}

\begin{proof}
By Corollary \ref{3-sCGDD to ooc}, there exists a $g$-regular $2$-D $(1\times n,3,1)$-OOC with $\Upsilon(1,0,n,g,3)$ codewords. Then apply Construction \ref{inflationg:g-regular} with $3$-GDD of type $m^3$ (from Lemma \ref{h-cyclic gdd}) to obtain a $g$-regular $2$-D $(m\times n,3,1)$-OOC with $\Upsilon(m,0,n,g,3)$ codewords.
\end{proof}

\begin{Lemma}\label{m=6t+3,regular}
There exists a $g$-regular $2$-D $(m\times n,3,1)$-OOC with $\Upsilon(m,0,n,g,3)$ codewords for
\begin{enumerate}
\item[$(1)$] $m\equiv3\pmod6$, $g=2$, $n\equiv0,2\pmod8$ and $n\geq8$;
\item[$(2)$] $m\equiv6\pmod{12}$, $g=2$, $n\equiv2\pmod4$ and $n\geq6$.
\end{enumerate}
\end{Lemma}

\begin{proof}
For Case $(1)$, by Example \ref{m=3,regular}, there exists a $g$-regular $2$-D $(3\times n,3,1)$-OOC with $\Upsilon(3,0,n,g,3)$ codewords. Then apply Construction \ref{inflationg:g-regular} with $3$-GDD of type $(m/3)^3$ (from Lemma \ref{h-cyclic gdd}) to obtain a $g$-regular $2$-D $(m\times n,3,1)$-OOC with $\Upsilon(m,0,n,g,3)$ codewords.

For Case $(2)$, by Example \ref{m=6,2-regular}, there exists a $g$-regular $2$-D $(6\times n,3,1)$-OOC with $\Upsilon(6,0,n,g,3)$ codewords. Then apply Construction \ref{inflationg:g-regular} with $3$-GDD of type $(m/6)^3$ (from Lemma \ref{h-cyclic gdd}) to obtain a $g$-regular $2$-D $(m\times n,3,1)$-OOC with $\Upsilon(m,0,n,g,3)$ codewords.
\end{proof}

\begin{Construction}\label{fill in ihgdd}
Suppose that there exist
\begin{enumerate}
\item[$(1)$] a $k$-SCIHGDD of type $(m-1,r-1,w^{t})$, and
\item[$(2)$] a $w$-regular $2$-D $([2:1]\times wt,k,1)$-OOC with $\Upsilon(2,1,wt,w,k)$ codewords.
\end{enumerate}
Then there exists a $w$-regular $2$-D $([m:r]\times wt,k,1)$-OOC with $\Upsilon(m,r,wt,w,k)=b_1+(m-r)\Upsilon(2,1,wt,w,k)$ codewords, where $b_1=(m(m-3)-r(r-3))w(t-1)/(k^2-k)$ is the number of base blocks of a $k$-SCIHGDD of type $(m-1,r-1,w^{t})$.
\end{Construction}

\begin{proof}
Let $H$ be the subgroup of order $w$ in $\mathbb{Z}_{wt}$ and $H_l=H+l$ be a coset of $H$ in $\mathbb{Z}_{wt}$.
Let $(X,Y,\mathcal{H},\mathcal{G},\mathcal{B})$ be the given
$k$-SCIHGDD of type $(m-1,r-1,w^{t})$, where $X=I_{m-1}\times\mathbb{Z}_{wt}$,
$Y=I_{r-1}\times \mathbb{Z}_{wt}$, $\mathcal{G}=\{\{x\}\times\mathbb{Z}_{wt}:x\in I_{m-1}\}$
and $\mathcal{H}=\{I_{m-1}\times H_l:0\leq l\leq t-1\}$. Let $\mathcal{A}_1$ be the set of base blocks of this design.

For each group $\{x\}\times\mathbb{Z}_{wt}$, $x\in I_{m-1}\setminus I_{r-1}$, construct a $w$-regular $2$-D $([2:1]\times wt,k,1)$-OOC with $\Upsilon(2,1,wt,w,k)$ codewords on $(\{x\}\bigcup\{\infty\})\times\mathbb{Z}_{wt}$ with $\{\infty\}\times\mathbb{Z}_{wt}$ as the hole and $(\{x\}\bigcup\{\infty\})\times H$ as the forbidden set. Denote by $\mathcal{C}_{x}$ the set of its codewords. Write $\mathcal{A}_2=\bigcup_{x\in I_{m-1}\setminus I_{r-1}}\mathcal{C}_x$.

Let $\mathcal{A}=\mathcal{A}_1\bigcup\mathcal{A}_2$. Then $\mathcal{A}$
forms the required $w$-regular $2$-D $([m:r]\times wt,k,1)$-OOC, which is defined on $(I_{m-1}\bigcup\{\infty\})\times\mathbb{Z}_{wt}$ with $(I_{r-1}\bigcup\{\infty\})\times\mathbb{Z}_{wt}$ as the hole and $(I_{m-1}\bigcup\{\infty\})\times H$ as the forbidden set.
\end{proof}

\begin{Remark}\label{2regular[4:1]}
The combination of Constructions \ref{filling3} and \ref{fill in ihgdd} will generate $2$-D $([m:r]\times n,3,1)$-OOCs
we need. For example, when $n\equiv0,2\pmod8$ and $n\geq8$, apply Construction \ref{fill in ihgdd} with a $3$-SCHGDD of type $(3,2^{n/2})$ and a
$2$-regular $2$-D $([2:1]\times n,3,1)$-OOC with $\Upsilon(2,1,n,2,3)$ codewords for Lemma \ref{2regular [2:1]} to obtain a $2$-regular $2$-D $([4:1]\times n,3,1)$-OOC with $\Upsilon(4,1,n,2,3)$ codewords. Further apply Construction \ref{filling3} with $2$-D $([4:1]\times 2,3,1)$-OOC with $\Theta(4,1,2,3)$ codewords $($which is just a $2$-cyclic $3$-GDD of type $2^4$ from Lemma \ref{h-cyclic gdd}$)$ to obtain a $2$-D $([4:1]\times n,3,1)$-OOC with $\Theta(4,1,n,3)$ codewords.
\end{Remark}

\begin{Lemma}\label{2regular[m:s]}
Let $m\equiv r\pmod{12}$ where $r\in\{1,2,4,5,7,8,10,11\}$ and $m\geq13$. Then there exists a $2$-regular $2$-D $([m:r]\times n,3,1)$-OOC with $\Upsilon(m,r,n,2,3)$ codewords for any $n\equiv0,2\pmod8$ and $n\geq8$.
\end{Lemma}

\begin{proof}
By Lemma \ref{scihgdd2}, there exists a $3$-SCIHGDD of type $(m-1,r-1,2^{n/2})$. Apply Construction \ref{fill in ihgdd} with
$2$-regular $2$-D $([2:1]\times n,3,1)$-OOC with $\Upsilon(2,1,n,2,3)$ codewords (from Example \ref{2regular [2:1]}) to obtain a $2$-regular $2$-D $([m:r]\times n,3,1)$-OOC with $\Upsilon(m,r,n,2,3)$ codewords.
\end{proof}

\section{Constructions for $2$-D $([m:r]\times n,k,1)$-OOCs}

\subsection{Skolem-type sequences}

A \emph{$k$-extended Skolem sequence} of order $v$ is a partition of $[1,2v+1]\setminus\{k\}$
into a collection of ordered pairs $(a_i,b_i)$ such that $b_i-a_i=i$
for $1\leq i\leq v$. When $k=2v$ and $2v+1$, a $k$-extended Skolem sequence of order $v$ is often said to be a \emph{hooked Skolem sequence} and \emph{Skolem sequence}, respectively.

\begin{Lemma}\label{Skolem}\rm{\cite{Baker}}
A $k$-extended Skolem sequence of order $v$
exists if and only if $v\equiv0, 1\pmod{4}$ and $k$ is odd, or
$v\equiv 2, 3\pmod{4}$ and $k$ is even.
\end{Lemma}

An \emph{$m$-near Skolem sequence} of order $v$ is a partition
of $[1,2v-2]$ into a collection of ordered pairs $(a_i,b_i)$ such that $\{b_i-a_i:1\leq i\leq v-1\}=[1,v]\setminus\{m\}$.

\begin{Lemma}\label{m near}\rm{\cite{s2}}
An $m$-near Skolem sequence of order $v$
exists if and only if $v\equiv0,1\pmod{4}$ and $m$ is odd, or
$v\equiv 2,3\pmod{4}$ and $m$ is even.
\end{Lemma}

A \emph{Langford sequence} of order $v$ and defect $d$ is a partition of $[1,2v]$ into a collection of ordered pairs $(a_i,b_i)$ such that $\{b_i-a_i:1\leq i\leq v\}=[d,d+v-1]$.

\begin{Lemma}\label{Langford}\rm{\cite{s}}
A Langford sequence of order $v$ and defect $d$
exists if and only if $(1)$ $v\geq2d-1$, $(2)$ $v\equiv0, 1\pmod{4}$ and $d$ is odd, or
$v\equiv 0, 3\pmod{4}$ and $d$ is even.
\end{Lemma}

\begin{Lemma}\rm{\cite[Lemma 1.8]{zc}}\label{sequence}
If $(s,d)\equiv(2,0),(2,1),(1,0),(3,1)\pmod{(4,2)}$
and $s(s-2d+1)+2\geq0$, then $[d,d+3s]\setminus\{d+3s-1\}$ can be
partitioned into triples $\{a_i,b_i,c_i\}$, $1\leq i\leq s$ such
that $a_i+b_i=c_i$.
\end{Lemma}

\subsection{Sporadic recursive constructions}

\begin{Lemma}\label{[3:0],n=0,6 mod8}
There exists an optimal $2$-D $([3:0]\times n, 3, 1)$-OOC with $\Theta(3,0,n,3)$ codewords for any $n\equiv0,6\pmod{8}$.
\end{Lemma}

\proof Let $S_t=\{(a_{ti},b_{ti}):1\leq i\leq (n-2)/2\}$ where $1\leq t\leq 2$, be a $k_t$-extended Skolem sequence of order $(n-2)/2$, where $k_1=2$ and $k_2=n-4$ (from Lemma \ref{Skolem}). Let $\mathcal{B}$ be the following $3(n-2)/2$ base blocks.
\begin{center}
\begin{tabular}{lll}
$\{(0,0),(0,i),(1,a_{1i}+i)\}$,
&$\{(1,0),(1,i),(2,a_{1i}+i)\}$,
&$\{(2,0),(2,i),(0,a_{2i}+i)\}$,
\end{tabular}
\end{center}
where $1\leq i\leq (n-2)/2$. It is readily checked that $\mathcal{B}$ together with $\{(0,0),(1,0),(2,0)\}$ and $\{(0,0),(1,2),(2,4)\}$ form a $2$-D $([3:0]\times n, 3, 1)$-OOC with $\Theta(3,0,n,3)$ codewords.
\qed

\begin{Lemma}\label{[5:2]}
There exists a $2$-D $([5:2]\times n,3,1)$-OOC with $\Theta(5,2,n,3)$ codewords for any $n\equiv0,6\pmod8$.
\end{Lemma}

\begin{proof}
By Lemma \ref{schgdd}, there exists a $3$-SCHGDD of type $(4,2^{n/2})$. Without loss of generality, assume that this design is constructed on $(I_3\bigcup\{\infty_0\})\times\mathbb{Z}_{n}$ with $\mathcal{G}=\{\{i\}\times {\mathbb Z}_{n}:i\in I_3\bigcup\{\infty_0\}\}$ and $\mathcal{H}=\{(I_3\bigcup\{\infty_0\})\times\{j,n/2+j\}:0\leq j\leq(n-2)/2\}$. Let ${\cal A}$ be a set of base blocks of this design.

Let $S_t=\{(a_{ti},b_{ti}):1\leq i\leq (n-2)/2\}$ where $0\leq t\leq 2$, be a $k_t$-extended Skolem sequence of order $(n-2)/2$, where $k_0=k_1=2$ and $k_2=4$ (from Lemma \ref{Skolem}). Let $\mathcal{B}$ be the following $3(n-2)/2$ codewords:
\begin{center}
\begin{tabular}{lll}
$\{(0,0),(0,i),(\infty_1,a_{0i}+i)\}$, &$\{(1,0),(1,i),(\infty_1,a_{1i}+i+2)\}$, &$\{(2,0),(2,i),(\infty_1,a_{2i}+i)\}$,
\end{tabular}
\end{center}
where $1\leq i\leq (n-2)/2$. Let $\mathcal{C}$ be the following $6$ codewords:
\begin{center}
\begin{tabular}{lll}
$\{(0,0),(\infty_0,0),(1,n/2)\}$,
&$\{(0,0),(2,n/2),(\infty_0,n/2)\}$,
&$\{(0,0),(1,0),(\infty_1,2)\}$,\\
$\{(1,0),(\infty_0,0),(2,n/2)\}$,
&$\{(0,0),(2,0),(\infty_1,0)\}$,
&$\{(1,0),(2,0),(\infty_1,4)\}$.
\end{tabular}
\end{center}
It is readily checked that $\mathcal{A}\bigcup\mathcal{B}\bigcup\mathcal{C}$ forms the desired code defined on $(I_3\bigcup\{\infty_0,\infty_1\})\times\mathbb{Z}_{n}$ with $\{\infty_0,\infty_1\}\times\mathbb{Z}_{n}$ as the hole.
\end{proof}

\begin{Lemma}\label{[9:3]}
There exists a $2$-D $([9:3]\times n,3,1)$-OOC with $\Theta(9,3,n,3)$ codewords for any $n\equiv4\pmod8$ and $n\geq12$.
\end{Lemma}

\begin{proof}
For $n=12$, the $2$-D $([9:3]\times12,3,1)$-OOC is defined on $(\mathbb{Z}_{6}\cup\{\infty_0,\infty_1,\infty_{2}\})\times \mathbb{Z}_{12}$ with the hole $\{\infty_0,\infty_1,\infty_{2}\}\times \mathbb{Z}_{12}$. All the $142$ codewords can be obtained by developing the following $48$ initial codewords by $(+2\pmod 6,-)$, where $\infty_i+2=\infty_i$ for $i\in I_3$. The initial codeword marked with a star only generates one codeword.
\begin{center}
\begin{longtable}{lll}
$\{(0,0),(1,0),(3,0)\}$,
&$\{(0,0),(5,0),(0,5)\}$,
&$\{(0,0),(\infty_0,0),(4,4)\}$,\\
$\{(0,0),(\infty_1,0),(1,6)\}$,
&$\{(0,0),(\infty_2,0),(1,7)\}$,
&$\{(0,0),(0,1),(3,3)\}$,\\
$\{(0,0),(1,1),(0,10)\}$,
&$\{(0,0),(2,1),(\infty_1,7)\}$,
&$\{(0,0),(3,1),(3,6)\}$,\\
$\{(0,0),(4,1),(\infty_2,9)\}$,
&$\{(0,0),(5,1),(3,7)\}$,
&$\{(0,0),(\infty_0,1),(2,7)\}$,\\
$\{(0,0),(\infty_1,1),(1,10)\}$,
&$\{(0,0),(\infty_2,1),(0,3)\}$,
&$\{(0,0),(1,2),(5,6)\}$,\\
$\{(0,0),(2,2),(\infty_0,9)\}$,
&$\{(0,0),(4,2),(\infty_2,6)\}$,
&$\{(0,0),(5,2),(1,9)\}$,\\
$\{(0,0),(\infty_0,2),(1,8)\}$,
&$\{(0,0),(\infty_1,2),(2,3)\}$,
&$\{(0,0),(\infty_2,2),(1,5)\}$,\\
$\{(0,0),(4,3),(\infty_1,8)\}$,
&$\{(0,0),(5,3),(\infty_2,7)\}$,
&$\{(0,0),(\infty_0,3),(1,11)\}$,\\
$\{(0,0),(\infty_1,3),(3,10)\}$,
&$\{(0,0),(\infty_2,3),(3,4)\}$,
&$\{(0,0),(0,4),(5,8)\}$,\\
$\{(0,0),(1,4),(5,11)\}$,
&$\{(0,0),(2,4),(5,9)\}$,
&$\{(0,0),(\infty_0,4),(3,8)\}$,\\
$\{(0,0),(\infty_1,4),(4,7)\}$,
&$\{(0,0),(5,5),(\infty_0,5)\}$,
&$\{(0,0),(\infty_2,5),(4,6)\}$,\\
$\{(0,0),(3,9),(\infty_0,11)\}$,
&$\{(0,0),(5,10),(\infty_1,10)\}$,
&$\{(0,0),(\infty_0,10),(3,11)\}$,\\
$\{(1,0),(\infty_2,0),(5,10)\}$,
&$\{(1,0),(1,1),(\infty_1,9)\}$,
&$\{(1,0),(3,1),(\infty_0,10)\}$,\\
$\{(1,0),(5,1),(\infty_2,7)\}$,
&$\{(1,0),(\infty_0,1),(1,8)\}$,
&$\{(1,0),(\infty_1,1),(1,9)\}$,\\
$\{(1,0),(\infty_2,1),(1,10)\}$,
&$\{(1,0),(5,2),(\infty_2,10)\}$,
&$\{(1,0),(\infty_1,2),(3,3)\}$,\\
$\{(1,0),(5,3),(\infty_1,10)\}$,
&$\{(1,0),(\infty_0,3),(5,8)\}$,
&$\{(0,0),(2,0),(4,0)\}^*$.
\end{longtable}
\end{center}

For $n\geq20$, by Lemma \ref{inflate schgdd1}, there exists an $n$-cyclic $3$-HGDD of type $(3,(3\times4)^{n/4})$. Without loss of generality, assume that this design is constructed on $I_9\times\mathbb{Z}_{n}$ with $\mathcal{G}=\{\{3i,3i+1,3i+2\}\times {\mathbb Z}_{n}:i\in I_3\}$ and $\mathcal{H}=\{I_9\times\{j,n/4+j,n/2+j,3n/4+j\}:0\leq j\leq(n-4)/4\}$. Let ${\cal A}$ be a set of base blocks of this design.
Let $S=\{(a_{t},b_{t}):1\leq t\leq (n-2)/2\}$ be a $(n/4)$-extended Skolem sequence of order $(n-2)/2$ (from Lemma \ref{Skolem}). Let $\mathcal{B}$ be the following $3(n-2)$ base blocks.
\begin{center}
\begin{tabular}{lll}
$\{(0,0),(0,t),(1,a_{t}+t)\}$, &$\{(1,0),(1,t),(2,a_{t}+t)\}$, &$\{(2,0),(2,t),(0,a_{t}+t)\}$,\\
$\{(3,0),(3,t),(4,a_{t}+t)\}$, &$\{(4,0),(4,t),(5,a_{t}+t)\}$, &$\{(5,0),(5,t),(3,a_{t}+t)\}$,\\
\end{tabular}
\end{center}
where $1\leq t\leq (n-2)/2$. Let $\mathcal{C}$ be the following $40$ codewords:
\begin{center}
\begin{tabular}{lll}

$\{(0,0),(1,0),(2,0)\}$,
&$\{(0,0),(3,0),(4,0)\}$,
&$\{(0,0),(5,0),(6,0)\}$,\\
$\{(0,0),(7,0),(1,n/4)\}$,
&$\{(0,0),(8,0),(3,n/4)\}$,
&$\{(0,0),(4,n/4),(6,n/4)\}$,\\
$\{(0,0),(5,n/4),(7,n/4)\}$,
&$\{(0,0),(8,n/4),(4,n/2)\}$,
&$\{(0,0),(3,n/2),(6,n/2)\}$,\\
$\{(0,0),(5,n/2),(8,n/2)\}$,
&$\{(0,0),(7,n/2),(5,3n/4)\}$,
&$\{(0,0),(2,3n/4),(6,3n/4)\}$,\\
$\{(0,0),(3,3n/4),(7,3n/4)\}$,
&$\{(0,0),(4,3n/4),(8,3n/4)\}$,
&$\{(1,0),(3,0),(5,0)\}$,\\
$\{(1,0),(4,0),(7,0)\}$,
&$\{(1,0),(6,0),(2,n/4)\}$,
&$\{(1,0),(8,0),(5,n/4)\}$,\\
$\{(1,0),(3,n/4),(8,n/4)\}$,
&$\{(1,0),(4,n/4),(8,n/2)\}$,
&$\{(1,0),(6,n/4),(3,n/2)\}$,\\
$\{(1,0),(7,n/4),(3,3n/4)\}$,
&$\{(1,0),(4,n/2),(6,3n/4)\}$,
&$\{(1,0),(5,n/2),(8,3n/4)\}$,\\
$\{(1,0),(6,n/2),(5,3n/4)\}$,
&$\{(1,0),(7,n/2),(4,3n/4)\}$,
&$\{(2,0),(3,0),(7,n/4)\}$,\\
$\{(2,0),(4,0),(5,0)\}$,
&$\{(2,0),(7,0),(5,n/2)\}$,
&$\{(2,0),(8,0),(3,n/2)\}$,\\
$\{(2,0),(3,n/4),(8,n/2)\}$,
&$\{(2,0),(4,n/4),(8,3n/4)\}$,
&$\{(2,0),(5,n/4),(7,n/2)\}$,\\
$\{(2,0),(6,n/4),(3,3n/4)\}$,
&$\{(2,0),(8,n/4),(5,3n/4)\}$,
&$\{(2,0),(4,n/2),(7,3n/4)\}$,\\
$\{(2,0),(6,n/2),(4,3n/4)\}$,
&$\{(3,0),(4,n/4),(7,3n/4)\}$,
&$\{(3,0),(6,n/4),(5,3n/4)\}$,\\
$\{(4,0),(5,n/4),(6,n/2)\}$.
\end{tabular}
\end{center}
It is readily checked that $\mathcal{A}\bigcup\mathcal{B}\bigcup\mathcal{C}$ forms the desired code defined on $I_9\times\mathbb{Z}_{n}$ with $\{6,7,8\}\times\mathbb{Z}_{n}$ as the hole.
\end{proof}

\begin{Lemma}\label{[10:4]}
Let $r\in\{1,2,4,5\}$ and $t\in\{1,2\}$. There exists a $2$-D $([6t+r:r]\times n,3,1)$-OOC with $\Theta(6t+r,r,n,3)$ codewords for any $n\equiv0\pmod4$ and $n\geq8$.
\end{Lemma}

\begin{proof}

By Lemma \ref{scihgdd1}, there exists a $3$-SCIHGDD of type $(6t+r-1,r-1,2^{n/2})$.
Let $Y=\{\infty_0,\infty_1,\ldots,\infty_{r-2}\}$ and $Y$ is an empty set if $r=1$. Assume that this design is constructed on $(I_{6t}\bigcup Y)\times\mathbb{Z}_{n}$ with $\mathcal{G}=\{\{i\}\times\mathbb{Z}_{n}:i\in I_{6t}\bigcup Y\}$ and $\mathcal{H}=\{(I_{6t}\bigcup Y)\times\{j,n/2+j\}:0\leq j\leq(n-2)/2\}$. Let $\mathcal{A}$ be a set of base blocks of this design.

Let $S=\{(a_{i},b_{i}):1\leq i\leq (n-2)/2\}$ be an $(n/4)$-extended Skolem sequence of order $(n-2)/2$ and
$T=\{(c_{i},d_{i}):1\leq i\leq (n-2)/2\}$ be a $(3n/4)$-extended Skolem sequence of order $(n-2)/2$ (from Lemma \ref{Skolem}). Let $\mathcal{B}$ be the following $3t(n-2)$ codewords:
\begin{center}
\begin{tabular}{lll}
$\{(l,0),(l,i),(\infty_{r-1},a_{i}+i)\}$, &$\{(e,0),(e,i),(\infty_{r-1},c_{i}+i+n/2)\}$, &$1\leq i\leq (n-2)/2$,
\end{tabular}
\end{center}
where $0\leq l\leq3t-1$ and $3t\leq e\leq6t-1$. Let $\mathcal{C}$ consist of $\{\{(2i,0),(2i+1,0),(\infty_{r-1},n/4)\}:0\leq i\leq3t-1\}$ and the $t(12t+4r-5)$ codewords in Appendix \ref{initial codewords of [10:4]}. It is readily checked that $\mathcal{A}\bigcup\mathcal{B}\bigcup\mathcal{C}$ forms the desired code on $(I_{6t}\bigcup\{\infty_0,\infty_1,\ldots,\infty_{r-1}\})\times\mathbb{Z}_{n}$ with $\{\infty_0,\infty_1,\ldots,\infty_{r-1}\}\times\mathbb{Z}_{n}$ as the hole.
\end{proof}

\subsection{Constructions from $n$-cyclic GDDs}

\begin{Construction}\rm{\cite[Construction 9]{fww}}\label{n-cyclic GDD}
Suppose that there exist
\begin{enumerate}
\item[$(1)$] an $n$-cyclic $k$-GDP of type $(v_{1}n)^{m_1} (v_{2}n)^{m_2}\cdots (v_{r}n)^{m_r}$ with $b$ base blocks;
\item[$(2)$] a $2$-D $([v_i+r:r]\times n,k,1)$-OOC with $f_i$ codewords for each $1\leq i\leq r$.
\end{enumerate}
Then there exists a $2$-D $([(\sum_{i=1}^r v_i m_i)+r:r]\times n,k,1)$-OOC with $(\sum_{i=1}^r m_if_i)+b$ codewords.
\end{Construction}

\begin{Lemma}\label{[10:1],n=8t+6}
There exists a $2$-D $([10:1]\times n,3,1)$-OOC with $\Theta(10,1,n,3)$ codewords for any $n\equiv6\pmod{8}$.
\end{Lemma}

\begin{proof}
There exists an $n$-cyclic $3$-GDD of type $(3n)^{3}n^1$ by Lemma \ref{3n^3n^1}.
Then apply Construction \ref{n-cyclic GDD} with $2$-D $(3\times n,3,1)$-OOC with $\Theta(3,0,n,3)$ codewords (from Lemma \ref{[3:0],n=0,6 mod8}) to obtain a $2$-D $([10:1]\times n,3,1)$-OOC with $\Theta(10,1,n,3)$ codewords.
\end{proof}

\begin{Lemma}\label{[11:2],n=8t+2}
There exists a $2$-D $([11:2]\times n,3,1)$-OOC with $\Theta(11,2,n,3)$ codewords for any $n\equiv2\pmod{8}$.
\end{Lemma}

\begin{proof}
There exists an $n$-cyclic $3$-GDD of type $(3n)^{3}n^1$ by Lemma \ref{3n^3n^1}.
Then apply Construction \ref{n-cyclic GDD} with $2$-D $([4:1]\times n,3,1)$-OOC with $\Theta(4,1,n,3)$ codewords (from Remark \ref{2regular[4:1]}) to obtain a $2$-D $([11:2]\times n,3,1)$-OOC with $\Theta(11,2,n,3)$ codewords.
\end{proof}

\begin{Lemma}\label{[m:r],n=4}
Let $m\equiv r\pmod{6}$ where $r\in\{1,2,3,4,5\}$ and $m\geq7$. Then there exists a $2$-D $([m:r]\times 4,3,1)$-OOC with $\Theta(m,r,4,3)$ codewords.
\end{Lemma}

\begin{proof}
Let $m=6t+r$. It follows from Appendix \ref{initial codewords of [m:r]} that the conclusion holds for $t\in\{1,2\}$ and $r\in\{1,2,3,4,5\}$. For $t\geq3$, by Lemma \ref{h-cyclic gdd}, there exists a $4$-cyclic $3$-GDD of type $24^{t}$ with $24t(t-1)$ base blocks. Then apply Construction \ref{n-cyclic GDD} with $2$-D $([6+r:r]\times4,3,1)$-OOC with $22+8r$ codewords to obtain a $2$-D $([6t+r:r]\times4,3,1)$-OOC with $\Theta(6t+r,r,4,3)=24t^2+8tr-2t$ codewords.
\end{proof}

\begin{Lemma}\label{[m:r],r=3}
Let $m\equiv 3\pmod{6}$ where $m\geq9$ and $m\neq15$. Then there exists a $2$-D $([m:3]\times n,3,1)$-OOC with $\Theta(m,3,n,3)$ codewords for any $n\equiv4\pmod{8}$ and $n\geq12$.
\end{Lemma}

\begin{proof}
It follows from Lemma \ref{[9:3]} that the conclusion holds for $m=9$. For $m\geq21$, by Lemma \ref{h-cyclic gdd}, there exists an $n$-cyclic $3$-GDD of type $(6n)^{(m-3)/6}$ with $(m-3)(m-9)n/6$ base blocks. Then apply Construction \ref{n-cyclic GDD} with $2$-D $([9:3]\times n,3,1)$-OOC with $12n-2$ codewords to obtain a $2$-D $([m:3]\times n,3,1)$-OOC with $\Theta(m,3,n,3)=[(m^2-9)n-2(m-3)]/6$ codewords.
\end{proof}

\begin{Lemma}\label{[m:r],n=4t}
Let $m\equiv r\pmod{6}$ where $r\in\{1,2,4,5\}$ and $m\geq7$. Then there exists a $2$-D $([m:r]\times n,3,1)$-OOC with $\Theta(m,r,n,3)$ codewords for any $n\equiv0\pmod{4}$ and $n\geq8$.
\end{Lemma}

\begin{proof}
Let $m=6t+r$. It follows from Lemma \ref{[10:4]} that the conclusion holds for $t\in\{1,2\}$. For $t\geq3$, by Lemma \ref{h-cyclic gdd}, there exists an $n$-cyclic $3$-GDD of type $(6n)^{t}$ with $6nt(t-1)$ base blocks. Then apply Construction \ref{n-cyclic GDD} with $2$-D $([6+r:r]\times n,3,1)$-OOC with $6n+2nr-2$ codewords to obtain a $2$-D $([6t+r:r]\times n,3,1)$-OOC with $\Theta(6t+r,r,n,3)=2t(3nt+nr-1)$ codewords.
\end{proof}

\begin{Lemma}\label{[m:r],n=8t+6}
Let $m\equiv r\pmod{12}$ where $r\in\{1,2,4,5,7,8,10,11\}$ and $m\geq13$. Then there exists a $2$-D $([m:r]\times n,3,1)$-OOC with $\Theta(m,r,n,3)$ codewords for any $n\equiv6\pmod{8}$.
\end{Lemma}

\begin{proof}
For $r\in\{1,2,5,8\}$, there exists an $n$-cyclic $3$-GDD of type $(3n)^{(m-r)/3}(rn)^1$ by Lemma \ref{3n^4t9n^1}.
Then apply Construction \ref{n-cyclic GDD} with $2$-D $(3\times n,3,1)$-OOC with $\Theta(3,0,n,3)$ codewords (from Lemma \ref{[3:0],n=0,6 mod8}) to obtain a $2$-D $([m:r]\times n,3,1)$-OOC with $\Theta(m,r,n,3)$ codewords.

For $r\in\{4,7,10,11\}$, there exists an $n$-cyclic $3$-GDD of type $(3n)^{(m-r)/3}((r-2)n)^1$ by Lemma \ref{3n^4t9n^1}. Then apply Construction \ref{n-cyclic GDD} with $2$-D $([5:2]\times n,3,1)$-OOC with $\Theta(5,2,n,3)$ codewords (from Lemma \ref{[5:2]}) to obtain a $2$-D $([m:r]\times n,3,1)$-OOC with $\Theta(m,r,n,3)$ codewords.
\end{proof}

\begin{Lemma}\label{[m:s]}
Let $m\equiv r\pmod{12}$ where $r\in\{1,2,4,5,7,8,10,11\}$ and $m\geq13$. Then there exists a $2$-D $([m:r]\times n,3,1)$-OOC with $\Theta(m,r,n,3)$ codewords for any $n\equiv2\pmod{8}$.
\end{Lemma}

\begin{proof}
For $n=2$, there exists a $2$-cyclic $3$-GDD of type $2^{m-r}(2r)^1$ (from Lemma \ref{scigdd}), which is also a $2$-D $([m:r]\times 2,3,1)$-OOC with $\Theta(m,r,n,3)=(m-r)(m+r-1)/3$ codewords.

For $n\geq10$, by Lemma \ref{2regular[m:s]}, there exists a $2$-regular $2$-D $([m:r]\times n,3,1)$-OOC with $\Upsilon(m,r,n,2,3)$ codewords. Then apply Construction \ref{filling3} with $2$-D $([m:r]\times 2,3,1)$-OOC with $\Theta(m,r,2,3)=(m-r)(m+r-1)/3$ codewords to obtain a $2$-D $([m:r]\times n,3,1)$-OOC with $\Theta(m,r,n,3)=[(m^2-r^2)n-2(m-r)]/6$ codewords.
\end{proof}

\section{Main results}

\subsection{The cases of $n=2,4,6,10$}

\begin{Lemma}\label{n=2}
There exists an optimal $2$-D $(m\times 2, 3, 1)$-OOC with $J^*(m\times 2,3,1)$ codewords for any positive integer $m$.
\end{Lemma}

\proof By Lemma \ref{3-GDP}, there exists a $2$-cyclic $3$-GDP of type $2^m$ with $J^*(m\times 2,3,1)$ base blocks, which is also a
$2$-D $(m\times 2, 3, 1)$-OOC with $J^*(m\times 2,3,1)$ codewords by Remark \ref{n-cyclic gdd and ooc}.
\qed

\begin{Lemma}\label{n=4,small m}
There exists an optimal $2$-D $(m\times 4, 3, 1)$-OOC with $J^*(m\times 4,3,1)$ codewords for $m\in\{2,3,4,5\}$.
\end{Lemma}

\begin{proof}
We construct the codes on $I_m\times \mathbb{Z}_4$ for $m\in\{2,3,4,5\}$. All the $J^*(m\times 4,3,1)$ codewords are listed below:
\begin{center}
\begin{longtable}{llll}
$m=2:$ &$\{(0,0),(1,0),(0,1)\}$,
&$\{(0,0),(1,1),(1,2)\}$.\\

$m=3:$ &$\{(0,0),(1,0),(2,0)\}$,
&$\{(0,0),(0,1),(1,2)\}$,
&$\{(0,0),(2,1),(1,3)\}$,\\
&$\{(0,0),(2,2),(2,3)\}$.\\

$m=4:$ & $\{(0,0),(1,0),(2,0)\}$,
&$\{(0,0),(3,0),(0,1)\}$,
&$\{(0,0),(1,1),(3,1)\}$,\\
&$\{(0,0),(2,1),(1,2)\}$,
&$\{(0,0),(2,2),(2,3)\}$,
&$\{(0,0),(3,2),(1,3)\}$,\\
&$\{(1,0),(1,1),(2,2)\}$,
&$\{(1,0),(3,1),(3,2)\}$.\\

$m=5:$ & $\{(0,0),(1,0),(2,0)\}$,
&$\{(0,0),(3,0),(4,0)\}$,
&$\{(0,0),(0,1),(1,2)\}$,\\
&$\{(0,0),(2,1),(3,1)\}$,
&$\{(0,0),(4,1),(2,2)\}$,
&$\{(0,0),(3,2),(1,3)\}$,\\
&$\{(0,0),(4,2),(3,3)\}$,
&$\{(0,0),(2,3),(4,3)\}$,
&$\{(1,0),(3,0),(2,1)\}$,\\
&$\{(1,0),(4,0),(1,1)\}$,
&$\{(1,0),(3,1),(3,2)\}$,
&$\{(1,0),(4,1),(4,2)\}$,\\
&$\{(1,0),(2,2),(2,3)\}$,
&$\{(2,0),(3,1),(4,2)\}$.
\end{longtable}
\end{center}
Note that $J^*(2\times 4,3,1)=\lfloor m(2m-1)/3\rfloor$, and $J^*(m\times 4,3,1)=\lfloor m(2m-1)/3\rfloor-1$ for $m\in\{3,4,5\}$.
\end{proof}

\begin{Lemma}\label{n=4}
There exists an optimal $2$-D $(m\times 4, 3, 1)$-OOC with $J^*(m\times 4,3,1)$ codewords for any positive integer $m\equiv1,2,3,4,5\pmod6$.
\end{Lemma}

\proof
Let $m=6t+r$ where $t\geq0$ and $r\in\{1,2,3,4,5\}$. For $(t,r)=(0,1)$, note that $J^*(1\times 4,3,1)$=0. For $t=0$ and $r\in\{2,3,4,5\}$,
by Lemma \ref{n=4,small m}, there exists an optimal $2$-D $(r\times4, 3, 1)$-OOC with $J^*(r\times 4,3,1)$ codewords. For $t\geq1$ and $r\in\{1,2,3,4,5\}$, by Lemma \ref{[m:r],n=4}, there exists a $2$-D $([6t+r:r]\times4,3,1)$-OOC with $24t^2+8tr-2t$ codewords.
Then apply Construction \ref{filling2} with the aforementioned optimal $2$-D $(r\times4, 3, 1)$-OOC with $J^*(r\times 4,3,1)$ codewords to obtain an optimal $2$-D $(m\times 4, 3, 1)$-OOC with $J^*(m\times 4,3,1)$ codewords.
\qed

\begin{Lemma}\label{n=6}
There exists an optimal $2$-D $(m\times 6, 3, 1)$-OOC with $J^*(m\times 6,3,1)$ codewords for $m\in\{7,8,10,11\}$.
\end{Lemma}

\proof
We construct the codes on $I_m\times \mathbb{Z}_6$. All the $J^*(m\times 6,3,1)$ codewords are listed in Appendix \ref{Appendix:n=6}.
\qed

\begin{Lemma}\label{n=10}
There exists an optimal $2$-D $(m\times10, 3, 1)$-OOC with $J^*(m\times10,3,1)$ codewords for $m\in\{7,10\}$.
\end{Lemma}

\proof
We construct the codes on $I_m\times \mathbb{Z}_{10}$. All the $J^*(m\times10,3,1)$ codewords are listed in Appendix \ref{Appendix:n=10}.
\qed

\subsection{The cases of $m\equiv1,2\pmod{3}$ and $m\leq11$}

\begin{Lemma}\label{Lem:2to2}\rm{\cite[Theorem 5]{am}}
Suppose that there exists a $2$-D $(m\times n, k, \lambda)$-OOC with $b$ codewords.
Then for any integer factorization $n=n_1n_2$, there exists a $2$-D $(mn_1\times n_2, k, \lambda)$-OOC
with $n_1b$ codewords.
\end{Lemma}

\begin{Lemma}\label{m=2}
There exists an optimal $2$-D $(2\times n, 3, 1)$-OOC with $J^*(2\times n,3,1)$ codewords for any $n\equiv0\pmod{2}$ and $n\geq4$.
\end{Lemma}

\proof For $n\equiv2\pmod6$ or $n\equiv4\pmod{12}$, the conclusion holds by Lemmas \ref{1D} and \ref{Lem:2to2}.
For $n=6$, the conclusion holds by Example \ref{(2,6)}. For the remaining cases, we construct as follows.

Let $S=\{(a_{i},b_{i}):1\leq i\leq (n-2)/2\}$ be a $k$-extended Skolem sequence of order $(n-2)/2$ (from Lemma \ref{Skolem}) and $$\mathcal{A}=\{\{(0,0),(0,i),(1,a_{i}+i)\}: 1\leq i\leq (n-2)/2\}.$$
Let $T=\{(c_{i},d_{i}):1\leq i\leq\lceil n/6\rceil-1\}$ be a $d$-extended Skolem sequence of order $\lceil n/6\rceil-1$ (from Lemma \ref{Skolem}) and
$$\mathcal{B}=\{\{(1,0),(1,i),(1,c_{i}+i+\lfloor n/6\rfloor)\}: 1\leq i\leq\lceil n/6\rceil-1\}.$$
For $n\equiv10\pmod{12}$, $\mathcal{A}\bigcup\mathcal{B}$ is a $2$-D $(2\times n, 3, 1)$-OOC with $J^*(2\times n,3,1)$ codewords. For $n\equiv0\pmod{6}$ and $n\geq12$, take $k=d+n/6$ and $\mathcal{C}=\{\{(0,0),(1,0),(1,k)\}\}$. Then $\mathcal{A}\bigcup\mathcal{B}\bigcup\mathcal{C}$ is a $2$-D $(2\times n, 3, 1)$-OOC with $J^*(2\times n,3,1)$ codewords.
\qed

\begin{Lemma}\label{m=4}
There exists an optimal $2$-D $(4\times n, 3, 1)$-OOC with $J^*(4\times n,3,1)$ codewords for any $n\equiv4,6\pmod{8}$.
\end{Lemma}

\begin{proof}
For $n\equiv14,20\pmod{24}$, the conclusion holds by Lemmas \ref{1D} and \ref{Lem:2to2}. For $n\equiv6,12\pmod{24}$, the conclusion holds by Lemmas \ref{m=2} and \ref{Lem:2to2}.

For $n=4$, the following is an optimal $2$-D $(4\times 4, 3, 1)$-OOC with $J^*(4\times4,3,1)=8$ codewords.
\begin{center}
\begin{tabular}{llll}
$\{(0,0),(1,0),(2,0)\}$,
&$\{(0,0),(3,0),(0,1)\}$,
&$\{(0,0),(1,1),(3,1)\}$,
&$\{(0,0),(2,1),(1,2)\}$,\\
$\{(0,0),(2,2),(2,3)\}$,
&$\{(0,0),(3,2),(1,3)\}$,
&$\{(1,0),(1,1),(2,2)\}$,
&$\{(1,0),(3,1),(3,2)\}$.
\end{tabular}
\end{center}
For $n\equiv4\pmod{24}$ and $n\geq28$, we construct the codes on $(I_3\cup\{\infty\})\times\mathbb{Z}_n$. All the codewords are divided into four parts.

The first part consists of $n-4$ codewords which are the base blocks of a $3$-SCHGDD of type $(3,4^{n/4})$ on $I_3\times\mathbb{Z}_n$ with the group set $\{\{i\}\times\mathbb{Z}_n:0\leq i\leq2\}$ and the hole set $\{I_3\times\{j,j+n/4,j+n/2,j+3n/4\}:0\leq j\leq n/4-1\}$ (from Lemma \ref{schgdd}).
The second part consists of the following $3(n-2)/2-1$ codewords:
\begin{center}
\begin{tabular}{llll}
$\{(i,0),(i,t),(\infty,a_{it}+t+\beta_i)\}$, &$1\leq t\leq (n-2)/2$, &$i=0,1,2$, &$(i,t)\neq(2,(n-4)/2)$,
\end{tabular}
\end{center}
where $\{(a_{it},b_{it}):1\leq t\leq (n-2)/2\}$ is a $k_i$-extended Skolem sequence of order $(n-2)/2$ with $(k_0,\beta_0)=(n/2-1,1)$,
$(k_1,\beta_1)=(n/2-3,1)$ and $(k_2,\beta_2)=(n-1,0)$ (from Lemma \ref{Skolem}). Note that, by the construction of Skolem sequence of order $(n-2)/2$ in \cite[Theorem 53.10]{s3}, we can take $(a_{it},b_{it})=(n/2,n-2)$ when $(i,t)=(2,(n-4)/2)$.
The third part consists of $\{(2,0),(\infty,0),(\infty,n-1)\}$ and the following $6$ codewords:
\begin{center}
\begin{tabular}{llll}
$\{(0,0),(2,0),(\infty,n/2)\}$,
&$\{(0,0),(1,n/2),(2,n/4)\}$,
&$\{(0,0),(1,3n/4),(2,3n/4)\}$,\\
$\{(0,0),(1,0),(\infty,1)\}$,
&$\{(0,0),(1,n/4),(2,n/2)\}$,
&$\{(2,0),(1,n/2),(\infty,n-2)\}$.
\end{tabular}
\end{center}
The fourth part consists of the following $s:=(n-4)/6$ codewords:
\begin{center}
\begin{tabular}{llll}
$\{(\infty,0),(\infty,c_t+s+1),(\infty,d_{t}+s+1)\}$, &$1\leq t\leq s$,
\end{tabular}
\end{center}
where $\{(c_t,d_t): 1\leq t\leq s\}$ is a $1$-near Skolem sequence of order $s+1$ (from Lemma \ref{m near}).

For $n\equiv22\pmod{24}$, we also construct the codes on $(I_3\cup\{\infty\})\times\mathbb{Z}_n$. All the codewords are divided into four parts.

The first part consists of $n-2$ codewords which are the base blocks of a $3$-SCHGDD of type $(3,2^{n/2})$ on $I_3\times\mathbb{Z}_n$ with the group set $\{\{i\}\times\mathbb{Z}_n:0\leq i\leq2\}$ and the hole set $\{I_3\times\{j,j+n/2\}:0\leq j\leq n/2-1\}$ (from Lemma \ref{schgdd}).
The second part consists of the following $3(n-2)/2-1$ codewords:
\begin{center}
\begin{tabular}{llll}
$\{(i,0),(i,t),(\infty,a_{it}+t+\beta_i)\}$, &$1\leq t\leq (n-2)/2$, &$0\leq i\leq2$, &$(i,t)\neq(0,(n-2)/2)$,
\end{tabular}
\end{center}
where $\{(a_{it},b_{it}):1\leq t\leq (n-2)/2\}$ is a $k_i$-extended Skolem sequence of order $(n-2)/2$ with $(k_0,\beta_0)=(n-2,0)$,
$(k_1,\beta_1)=(n/2-1,0)$ and $(k_2,\beta_2)=(n/2-1,n/2)$ (from Lemma \ref{Skolem}). Note that, by the construction of hooked Skolem sequence of order $(n-2)/2$ in \cite[Theorem 53.10]{s3}, we can take $(a_{it},b_{it})=(n/2,n-1)$ when $(i,t)=(0,(n-2)/2)$.
The third part consists of the following $5$ codewords:
\begin{center}
\begin{tabular}{llll}
$\{(0,0),(1,n/2),(\infty,n/2)\}$,
&$\{(1,0),(2,n/2),(\infty,(n-2)/2)\}$,
&$\{(0,0),(2,n/2),(\infty,0)\}$,\\
$\{(0,0),(1,0),(2,0)\}$,
&$\{(0,0),(\infty,n-2),(\infty,n-1)\}$.
\end{tabular}
\end{center}
The fourth part consists of the following $s:=(n-4)/6$ codewords:
\begin{center}
\begin{tabular}{llll}
$\{(\infty,0),(\infty,c_t+s+1),(\infty,d_{t}+s+1)\}$, &$1\leq t\leq s$,
\end{tabular}
\end{center}
where $\{(c_t,d_t): 1\leq t\leq s\}$ is a Langford sequence of order $s$ and defect $2$ (from Lemma \ref{Langford}).
\end{proof}

\begin{Lemma}\label{m<12,n=0mod8}
There exists an optimal $2$-D $(m\times n, 3, 1)$-OOC with $J^*(m\times n,3,1)$ codewords for $m\in\{4,5\}$ and $n\equiv0\pmod{8}$.
\end{Lemma}

\begin{proof}
For $m\in\{4,5\}$, let $m=3+r$, where $r\in\{1,2\}$. By Remark \ref{2regular[4:1]} and Lemma \ref{[5:2]}, there exists a $2$-D $([m:r]\times n,3,1)$-OOC with $\Theta(m,r,n,3)$ codewords. Apply Construction \ref{filling2} with a $2$-D $(r\times n,3,1)$-OOC with $J^*(r\times n,3,1)$ codewords (from Lemmas \ref{1D}, \ref{m=2}) to obtain a $2$-D $(m\times n, 3, 1)$-OOC with $J^*(m\times n,3,1)$ codewords.
\end{proof}

\begin{Lemma}\label{m<12,n=2mod8}
There exists an optimal $2$-D $(m\times n, 3, 1)$-OOC with $J^*(m\times n,3,1)$ codewords for $(1)$ $m=5$ and $n\equiv4\pmod{8}$; $(2)$ $m\in\{4,5,7,8,10,11\}$, $n\equiv2\pmod{8}$.
\end{Lemma}

\begin{proof}
For $n\in\{2,4\}$, Lemmas \ref{n=2} and \ref{n=4} provide the required codes. For $(m,n)=(5,12)$, the $J^*(5\times12,3,1)=48$ codewords are listed below.
\begin{center}
\begin{tabular}{llll}
$\{(0,0),(1,0),(2,0)\}$,
&$\{(0,0),(3,0),(1,8)\}$,
&$\{(0,0),(4,0),(2,9)\}$,
&$\{(0,0),(0,1),(4,6)\}$,\\
$\{(0,0),(1,1),(0,8)\}$,
&$\{(0,0),(2,1),(1,6)\}$,
&$\{(0,0),(3,1),(4,1)\}$,
&$\{(0,0),(0,2),(1,4)\}$,\\
$\{(0,0),(2,2),(3,3)\}$,
&$\{(0,0),(3,2),(2,4)\}$,
&$\{(0,0),(4,2),(1,9)\}$,
&$\{(0,0),(0,3),(2,6)\}$,\\
$\{(0,0),(1,3),(4,3)\}$,
&$\{(0,0),(3,4),(0,5)\}$,
&$\{(0,0),(4,4),(2,5)\}$,
&$\{(0,0),(3,5),(3,6)\}$,\\
$\{(0,0),(1,7),(3,7)\}$,
&$\{(0,0),(2,7),(1,11)\}$,
&$\{(0,0),(4,7),(3,8)\}$,
&$\{(0,0),(2,8),(4,8)\}$,\\
$\{(0,0),(3,9),(1,10)\}$,
&$\{(0,0),(4,9),(4,10)\}$,
&$\{(0,0),(2,10),(2,11)\}$,
&$\{(0,0),(3,10),(4,11)\}$,\\
$\{(1,0),(1,1),(2,10)\}$,
&$\{(1,0),(2,1),(2,4)\}$,
&$\{(1,0),(3,1),(1,4)\}$,
&$\{(1,0),(4,1),(2,11)\}$,\\
$\{(1,0),(1,2),(3,7)\}$,
&$\{(1,0),(2,2),(4,7)\}$,
&$\{(1,0),(3,2),(4,9)\}$,
&$\{(1,0),(4,2),(1,3)\}$,\\
$\{(1,0),(2,3),(3,3)\}$,
&$\{(1,0),(4,3),(1,5)\}$,
&$\{(1,0),(4,4),(4,6)\}$,
&$\{(1,0),(2,5),(3,8)\}$,\\
$\{(1,0),(2,6),(3,10)\}$,
&$\{(1,0),(3,6),(4,8)\}$,
&$\{(2,0),(4,1),(4,8)\}$,
&$\{(2,0),(2,2),(3,8)\}$,\\
$\{(2,0),(3,2),(4,6)\}$,
&$\{(2,0),(2,4),(3,9)\}$,
&$\{(2,0),(4,4),(2,7)\}$,
&$\{(2,0),(3,7),(4,10)\}$,\\
$\{(2,0),(4,7),(3,11)\}$,
&$\{(3,0),(3,2),(3,5)\}$,
&$\{(3,0),(3,4),(4,9)\}$,
&$\{(3,0),(4,6),(4,10)\}$.
\end{tabular}
\end{center}

For $m=4$, by Remark \ref{2regular[4:1]}, there exists a $2$-D $([4:1]\times n,3,1)$-OOC with $\Theta(4,1,n,3)$ codewords.
Apply Construction \ref{filling2} with a $1$-D $(n,3,1)$-OOC with $J^*(1\times n,3,1)$ codewords (from Lemma \ref{1D}) to obtain a $2$-D $(4\times n, 3, 1)$-OOC with $J^*(4\times n,3,1)$ codewords.

For $m=11$, there exists a $2$-D $([11:2]\times n,3,1)$-OOC with $\Theta(11,2,n,3)$ codewords (from Lemma \ref{[11:2],n=8t+2}).
Apply Construction \ref{filling2} with a $2$-D $(2\times n,3,1)$-OOC with $J^*(2\times n,3,1)$ codewords (from Lemma \ref{m=2}) to obtain a $2$-D $(11\times n, 3, 1)$-OOC with $J^*(11\times n,3,1)$ codewords.

For $m\in\{5,8\}$ and $n\equiv10\pmod{24}$, the conclusion holds by Lemmas \ref{1D} and \ref{Lem:2to2}. For $m\in\{7,10\}$ and $n=10$, Lemma \ref{n=10} provides the required codes.

For $m=5$ and $n\equiv4\pmod{24}$, apply Construction \ref{filling3} with a $2$-D $(5\times 4,3,1)$-OOC with $J^*(5\times 4,3,1)$ codewords (from Lemma \ref{n=4}) and a $4$-regular $2$-D $(5\times n,3,1)$-OOC with $\Upsilon(5,0,n,4,3)$ codewords (from Lemma \ref{m=24t+6,6regular}).

For $m\in\{7,10\}$ and $n\equiv2\pmod{24}$, apply Construction \ref{filling3} with a $2$-D $(m\times 2,3,1)$-OOC with $J^*(m\times 2,3,1)$ codewords (from Lemma \ref{n=2}) and a $2$-regular $2$-D $(m\times n,3,1)$-OOC with $\Upsilon(m,0,n,2,3)$ codewords (from Lemma \ref{m=24t+6,6regular}).

Now we deal with the remaining four cases: $(1)$ $m=5$, $n\equiv12,20\pmod{24}$ and $n\geq20$; $(2)$ $m\in\{5,8\}$, $n\equiv2\pmod{24}$ and $n\geq26$; $(3)$ $m\in\{7,10\}$, $n\equiv10\pmod{24}$ and $n\geq34$; $(4)$ $m\in\{5,7,8,10\}$ and $n\equiv18\pmod{24}$. We construct the codes on $(I_{m-1}\cup\{\infty\})\times\mathbb{Z}_n$. All the codewords are divided into four parts. Let $e_m=3\lfloor(m-3)/3\rfloor$ and $r_m=m\pmod3+6$.

The first part consists of $(m-1)(m-2)(n-2)/6$ codewords, which are the base blocks of a $3$-SCHGDD of type $(m-1,2^{n/2})$ on $I_{m-1}\times\mathbb{Z}_n$ with the group set $\{\{i\}\times\mathbb{Z}_n:i\in I_{m-1}\}$ and the hole set $\{I_{m-1}\times\{j,j+n/2\}:0\leq j\leq n/2-1\}$ (from Lemma \ref{schgdd}).

The second part consists of the following $(m-1)(n-2)/2$ codewords:
\begin{center}
\begin{tabular}{llll}
$\{(i,0),(i,t),(\infty,a_{it}+t)\}$, &$1\leq t\leq (n-2)/2$, &$0\leq i\leq m-2$,\\
\end{tabular}
\end{center}
where $\{(a_{it},b_{it}):1\leq t\leq (n-2)/2\}$ is a $k_i$-extended Skolem sequence of order $(n-2)/2$ with $k_i=n/2$ for $0\leq i\leq e_m-1$,
and $k_i=2(i-e_m)+1$ for $e_m\leq i\leq m-2$ (from Lemma \ref{Skolem}).

The third part consists of $\{(i,0),(\infty,0),(\infty,2(i-e_m)+1)\}$ for $e_m\leq i\leq m-2$ and the following codewords:
\begin{center}
\begin{longtable}{llll}
$m=5$:\\
$\{(0,0),(1,0),(2,0)\}$,
&$\{(0,0),(3,0),(1,n/2)\}$,
&$\{(0,0),(2,n/2),(3,n/2)\}$,\\
$\{(1,0),(3,0),(2,n/2)\}$.\\
$m=7$:\\
$\{(0,0),(1,0),(3,0)\}$,
&$\{(0,0),(2,0),(4,0)\}$,
&$\{(0,0),(5,0),(3,n/2)\}$,\\
$\{(0,0),(\infty,0),(1,n/2)\}$,
&$\{(0,0),(2,n/2),(\infty,n/2)\}$,
&$\{(0,0),(4,n/2),(5,n/2)\}$,\\
$\{(1,0),(2,0),(5,0)\}$,
&$\{(1,0),(4,0),(5,n/2)\}$,
&$\{(1,0),(\infty,0),(2,n/2)\}$,\\
$\{(1,0),(3,n/2),(4,n/2)\}$,
&$\{(2,0),(3,0),(4,n/2)\}$,
&$\{(2,0),(3,n/2),(5,n/2)\}$.\\
$m=8$:\\
$\{(0,0),(1,0),(2,0)\}$,
&$\{(0,0),(3,0),(4,0)\}$,
&$\{(0,0),(5,0),(6,0)\}$,\\
$\{(0,0),(\infty,0),(1,n/2)\}$,
&$\{(0,0),(2,n/2),(\infty,n/2)\}$,
&$\{(0,0),(3,n/2),(5,n/2)\}$,\\
$\{(0,0),(4,n/2),(6,n/2)\}$,
&$\{(1,0),(3,0),(6,0)\}$,
&$\{(1,0),(4,0),(3,n/2)\}$,\\
$\{(1,0),(5,0),(6,n/2)\}$,
&$\{(1,0),(\infty,0),(2,n/2)\}$,
&$\{(1,0),(4,n/2),(5,n/2)\}$,\\
$\{(2,0),(3,0),(5,n/2)\}$,
&$\{(2,0),(4,0),(6,n/2)\}$,
&$\{(2,0),(5,0),(4,n/2)\}$,\\
$\{(2,0),(6,0),(3,n/2)\}$.\\
$m=10$:\\
$\{(0,0),(1,0),(2,0)\}$,
&$\{(0,0),(3,0),(4,0)\}$,
&$\{(0,0),(5,0),(6,0)\}$,\\
$\{(0,0),(7,0),(8,0)\}$,
&$\{(0,0),(\infty,0),(1,n/2)\}$,
&$\{(0,0),(2,n/2),(3,n/2)\}$,\\
$\{(0,0),(4,n/2),(7,n/2)\}$,
&$\{(0,0),(5,n/2),(\infty,n/2)\}$,
&$\{(0,0),(6,n/2),(8,n/2)\}$,\\
$\{(1,0),(3,0),(5,0)\}$,
&$\{(1,0),(4,0),(6,0)\}$,
&$\{(1,0),(7,0),(2,n/2)\}$,\\
$\{(1,0),(8,0),(3,n/2)\}$,
&$\{(1,0),(\infty,0),(4,n/2)\}$,
&$\{(1,0),(5,n/2),(8,n/2)\}$,\\
$\{(1,0),(6,n/2),(7,n/2)\}$,
&$\{(2,0),(4,0),(8,0)\}$,
&$\{(2,0),(5,0),(6,n/2)\}$,\\
$\{(2,0),(6,0),(4,n/2)\}$,
&$\{(2,0),(7,0),(8,n/2)\}$,
&$\{(2,0),(\infty,0),(5,n/2)\}$,\\
$\{(2,0),(3,n/2),(\infty,n/2)\}$,
&$\{(3,0),(6,0),(7,n/2)\}$,
&$\{(3,0),(7,0),(5,n/2)\}$,\\
$\{(3,0),(8,0),(6,n/2)\}$,
&$\{(3,0),(4,n/2),(\infty,n/2)\}$,
&$\{(4,0),(5,0),(8,n/2)\}$,\\
$\{(4,0),(5,n/2),(7,n/2)\}$.
\end{longtable}
\end{center}

The fourth part consists of $\{(\infty,0),(\infty,2),(\infty,6)\}$ and
the following $s:=\lfloor(n-16)/6\rfloor$ codewords:
$$\left\{
\begin{array}{lll}
\{(\infty,0),(\infty,a_i),(\infty,b_i)\}, & \ \ {\rm if}\ n\equiv2,10\pmod{24},\\
\{(\infty,0),(\infty,c_i+r_m+s-1),(\infty,d_i+r_m+s-1)\}, & \ \ {\rm if}\ n\equiv12,18,20\pmod{24},\\
\end{array}
\right.
$$
where $\{\{a_i,b_i-a_i,b_i\}: 1\leq i\leq s\}$ is a partition of $[r_m,r_m+3s]\setminus\{r_m+3s-1\}$ and
$\{(c_i,d_i): 1\leq i\leq s\}$ is a Langford sequence of order $s$ and defect $r_m$ (from Lemmas \ref{Langford} and \ref{sequence}).
Note that Lemmas \ref{Langford} and \ref{sequence} work only when $s\geq2r_m-1$, the codewords of the fourth part for $s\leq2r_m-2$ are listed in Appendix \ref{fourthpart m=5}.
\end{proof}

\begin{Lemma}\label{m<12,n=6mod8}
There exists an optimal $2$-D $(m\times n, 3, 1)$-OOC with $J^*(m\times n,3,1)$ codewords for $m\in\{5,7,8,10,11\}$, $n\equiv6\pmod{8}$.
\end{Lemma}

\begin{proof}
For $m=5$, there exists a $2$-D $([5:2]\times n,3,1)$-OOC with $\Theta(5,2,n,3)$ codewords by Lemma \ref{[5:2]}.
Apply Construction \ref{filling2} with a $2$-D $(2\times n,3,1)$-OOC with $J^*(2\times n,3,1)$ codewords (from Lemma \ref{m=2}) to obtain a $2$-D $(5\times n, 3, 1)$-OOC with $J^*(5\times n,3,1)$ codewords.

For $m=10$, there exists a $2$-D $([10:1]\times n,3,1)$-OOC with $\Theta(10,1,n,3)$ codewords by Lemma \ref{[10:1],n=8t+6}.
Apply Construction \ref{filling2} with a $2$-D $(1\times n,3,1)$-OOC with $J^*(1\times n,3,1)$ codewords (from Lemma \ref{1D}) to obtain a $2$-D $(10\times n, 3, 1)$-OOC with $J^*(10\times n,3,1)$ codewords.

For $m\in\{7,8,11\}$ and $n\equiv6\pmod{24}$, apply Construction \ref{filling3} with a $2$-D $(m\times 6,3,1)$-OOC with $J^*(m\times 6,3,1)$ codewords (from Lemma \ref{n=6}) and a $6$-regular $2$-D $(m\times n,3,1)$-OOC with $\Upsilon(m,0,n,6,3)$ codewords (from Lemma \ref{m=24t+6,6regular}).

For either $m\in\{8,11\}$ and $n\equiv22\pmod{24}$, or $m=7$ and $n\equiv14\pmod{24}$, the conclusion holds by Lemmas \ref{1D} and \ref{Lem:2to2}.
For $m=8$ and $n\equiv14\pmod{24}$, the conclusion holds by Lemmas \ref{m=4} and \ref{Lem:2to2}.

For $m=7$ and $n\equiv22\pmod{24}$, all the $J^*(7\times n,3,1)=(49n-16)/6$ codewords are divided into two parts. The first part consists of $30$ codewords on $I_7\times\Z_n$:
\begin{center}
\begin{tabular}{lllll}
$\{(0,0),(1,0),(2,0)\}$,
&$\{(0,0),(5,0),(6,0)\}$,\\
$\{(0,0),(1,1),(3,1)\}$,
&$\{(0,0),(2,1),(3,2)\}$,\\
$\{(0,0),(0,n/2-1),(4,n-1)\}$,
&$\{(0,0),(1,n/2-1),(2,n/2)\}$,\\
$\{(0,0),(2,n/2-1),(3,n/2-1)\}$,
&$\{(0,0),(1,n/2),(6,n/2)\}$,\\
$\{(0,0),(3,n/2),(5,n/2+1)\}$,
&$\{(0,0),(5,n/2),(5,n-1)\}$,\\
$\{(0,0),(4,n/2+1),(6,n/2+1)\}$,
&$\{(0,0),(4,n-2),(6,n-1)\}$,\\
$\{(1,0),(3,1),(4,1)\}$,
&$\{(1,0),(4,2),(5,n/2+1)\}$,\\
$\{(1,0),(1,n/2-1),(6,n-1)\}$,
&$\{(1,0),(2,n/2-1),(5,n-2)\}$,\\
$\{(1,0),(3,n/2-1),(6,n/2+1)\}$,
&$\{(1,0),(4,n/2-1),(5,n-1)\}$,\\
$\{(1,0),(2,n/2),(4,n/2)\}$,
&$\{(1,0),(3,n/2),(5,n/2)\}$,\\
$\{(2,0),(4,1),(4,n/2)\}$,
&$\{(2,0),(5,1),(2,n/2+1)\}$,\\
$\{(2,0),(5,2),(6,n/2+1)\}$,
&$\{(2,0),(3,n/2-1),(6,n/2)\}$,\\
$\{(2,0),(4,n/2-1),(6,n-2)\}$,
&$\{(2,0),(3,n/2),(6,n-1)\}$,\\
$\{(3,0),(4,1),(3,n/2+1)\}$,
&$\{(3,0),(4,n/2-1),(5,n/2)\}$,\\
$\{(3,0),(5,n/2-1),(6,n/2)\}$,
&$\{(4,0),(5,0),(6,n/2)\}$.\\
\end{tabular}
\end{center}
The second part consist of $(49n-196)/6$ codewords, which are obtained by developing the following $(7n-28)/6$ codewords by $(+1\mod{7},-)$. When $n=22$, the $21$ codewords are:

\begin{center}
\begin{tabular}{lllll}
$\{(0,0),(3,0),(0,1)\}$,
&$\{(0,0),(5,1),(4,2)\}$,
&$\{(0,0),(0,2),(1,4)\}$,\\
$\{(0,0),(2,2),(0,4)\}$,
&$\{(0,0),(6,2),(0,5)\}$,
&$\{(0,0),(0,3),(2,6)\}$,\\
$\{(0,0),(3,3),(0,6)\}$,
&$\{(0,0),(5,3),(4,6)\}$,
&$\{(0,0),(2,4),(0,8)\}$,\\
$\{(0,0),(3,4),(0,9)\}$,
&$\{(0,0),(4,4),(5,10)\}$,
&$\{(0,0),(6,4),(3,12)\}$,\\
$\{(0,0),(1,5),(6,10)\}$,
&$\{(0,0),(2,5),(1,13)\}$,
&$\{(0,0),(3,5),(4,13)\}$,\\
$\{(0,0),(6,5),(2,13)\}$,
&$\{(0,0),(3,6),(3,13)\}$,
&$\{(0,0),(5,6),(6,13)\}$,\\
$\{(0,0),(6,6),(5,13)\}$,
&$\{(0,0),(2,7),(5,14)\}$,
&$\{(0,0),(4,7),(2,14)\}$.
\end{tabular}
\end{center}

When $n\equiv 22\pmod{24}$ and $n\geq 46$, the $(7n-28)/6$ codewords are divided into two pieces:

\noindent $\bullet$ The fist piece consists of $n-3$ codewords.
\begin{center}
\begin{tabular}{lllll}
$\{(0,0),(1,2+i),(3,4+2i)\}$, & $i\in[0,n/2-4]\setminus\{(n-10)/4\}$;\\
$\{(0,0),(1,n/2+2+i),(3,5+2i)\}$, & $i\in[0,n/2-5]\setminus\{(n-10)/4\}$,
\end{tabular}

\begin{tabular}{lllll}
$\{(0,0),(0,(n-6)/4),(1,n-2)\}$,
&$\{(0,0),(0,(n+6)/4),(1,n/2+1)\}$,\\
$\{(0,0),(0,4),(3,3)\}$,
&$\{(0,0),(1,n-1),(3,n-2)\}$,\\
$\{(0,0),(2,(n-2)/4),(4,0)\}$,
&$\{(0,0),(2,n/2+1),(4,3)\}$.\\
\end{tabular}
\end{center}
\noindent $\bullet$ The second piece consists of $(n-10)/6$ codewords.

If $n=46$, then take
\begin{center}
\begin{tabular}{lllll}
$\{(0,0),(0,1),(0,20)\}$,
&$\{(0,0),(0,2),(0,18)\}$,
&$\{(0,0),(0,3),(0,15)\}$,\\
$\{(0,0),(0,5),(0,11)\}$,
&$\{(0,0),(0,7),(0,21)\}$,
&$\{(0,0),(0,8),(0,17)\}$.\\
\end{tabular}
\end{center}

If $n\equiv 22\pmod{24}$ and $n\geq70$, then take

\begin{center}
\begin{tabular}{lllll}
$\{(0,0),(0,5+2i),(0,(5n+22)/12+i)\}$, & $i\in[0,(n-46)/12]\setminus\{(n-70)/24\}$;\\
$\{(0,0),(0,6+2i),(0,(n+14)/4+i)\}$, & $i\in[0,(n-46)/12]\setminus\{(n-46)/24\}$,
\end{tabular}
\begin{tabular}{lllll}
$\{(0,0),(0,1),(0,(5n-14)/12)\}$,
&$\{(0,0),(0,2),(0,(n+10)/4)\}$,\\
$\{(0,0),(0,3),(0,(n+26)/12)\}$,
&$\{(0,0),(0,(n+2)/6),(0,(3n-2)/8)\}$,\\
$\{(0,0),(0,(n-4)/6),(0,(7n+2)/12)\}$,
&$\{(0,0),(0,(n-2)/4),(0,(13n+26)/24)\}$.\\
\end{tabular}
\end{center}

For $m=11$ and $n\equiv14\pmod{24}$, we construct the code on $(I_{10}\cup\{\infty\})\times\mathbb{Z}_n$. All the codewords are divided into five parts.

Let $\mathcal{B}$ be the block set of a $3$-GDD of type $3^3$ on $I_9$ with group set $\{G_i:0\leq i\leq 2\}$ where $G_0=\{0,5,8\}$, $G_1=\{1,3,7\}$ and $G_2=\{2,4,6\}$ (from Lemma \ref{h-cyclic gdd}). For each $B\in\mathcal{B}$, construct an $n$-cyclic $3$-HGDD of type $(3,2^{n/2})$ on $B\times\mathbb{Z}_n$ with group set $\{\{x\}\times\mathbb{Z}_n:x\in B\}$ and hole set $\{B\times\{j,n/2+j\}:0\leq j\leq n/2-1\}$ (from Lemma \ref{schgdd}). Let $\mathcal{A}_B$ be the base block set of this design. Then $\mathcal{A}_1=\bigcup_{B\in\mathcal{B}}\mathcal{A}_B$ forms the first part of codewords of the required code.

For $0\leq i\leq 2$, let $\mathcal{C}_i$ be the base block set of an $n$-cyclic $3$-GDD of type $n^4$ on $(G_i\bigcup\{9\})\times\mathbb{Z}_n$ with group set $\{\{x\}\times\mathbb{Z}_n:x\in G_i\bigcup\{9\}\}$ (from Lemma \ref{h-cyclic gdd}). Note that, by the construction of \cite[Section 3]{gjl}, we can assume that $T_{0}=\{(0,0),(5,0),(8,0)\}\in\mathcal{C}_0$, $T_{1}=\{(1,0),(3,0),(7,0)\}\in\mathcal{C}_1$, $T_{2}=\{(2,0),(4,0),(6,0)\}$ and $T_{3}=\{(2,0),(4,2),(6,4)\}\in\mathcal{C}_2$. Then $\mathcal{A}_2=(\bigcup_{0\leq i\leq 2}\mathcal{C}_i)\setminus\{T_{0},T_{1},T_{2},T_{3}\}$ forms the second part of codewords of the required code.

For $0\leq i\leq 9$, let $\{(a_{it},b_{it}):1\leq t\leq (n-2)/2\}$ be a $k_i$-extended Skolem sequence of order $(n-2)/2$
with $k_0=k_5=k_7=4$, $k_1=k_2=k_3=k_6=k_8=2$, $k_4=n/2-3$ and $k_9=n/2-1$ (from Lemma \ref{Skolem}). Let $\mathcal{D}_4$ denote the following
$(n-4)/2$ codewords:
\begin{center}
\begin{tabular}{llll}
$\{(4,0),(4,t),(\infty,a_{4t}+t)\}$, &$1\leq t\leq (n-2)/2$, &$t\neq(n-4)/2$.
\end{tabular}
\end{center}
Assume that $a_{4t}+t=u$ with $t=(n-4)/2$. Let $\mathcal{D}_2$ denote the following
$(n-4)/2$ codewords:
\begin{center}
\begin{tabular}{llll}
$\{(2,0),(2,t),(\infty,a_{2t}+t+u+2)\}$, &$1\leq t\leq (n-2)/2$, &$t\neq4$.
\end{tabular}
\end{center}
Assume that $a_{2t}+t+u+2=v$ with $t=4$. Let $\mathcal{D}_6$ denote the following
$(n-4)/2$ codewords:
\begin{center}
\begin{tabular}{llll}
$\{(6,0),(6,t),(\infty,a_{6t}+t+v+2)\}$, &$1\leq t\leq (n-2)/2$, &$t\neq2$.
\end{tabular}
\end{center}
Assume that $a_{6t}+t+v+2=w$ with $t=2$. For $0\leq i\leq 9$ and $i\notin\{2,4,6\}$, let $\mathcal{D}_i$ denote the following
$(n-2)/2$ codewords:
\begin{center}
\begin{tabular}{llll}
$\{(i,0),(i,t),(\infty,a_{it}+t+\beta_i)\}$, &$1\leq t\leq (n-2)/2$,
\end{tabular}
\end{center}
where $\beta_0=v-4$, $\beta_1=v$, $\beta_3=w$, $\beta_5=v-n/2$, $\beta_7=w-2$, $\beta_8=u+2$ and $\beta_9=0$.
Then $\mathcal{A}_3=\bigcup_{0\leq i\leq 9}\mathcal{D}_i$ forms the third part of codewords of the required code.
The fourth part $\mathcal{A}_4$ consists of the following $(n-8)/6$ codewords:
\begin{center}
\begin{tabular}{llll}
$\{(\infty,0),(\infty,t),(\infty,c_{t}+t+(n-8)/6)\}$, &$1\leq t\leq (n-8)/6$,
\end{tabular}
\end{center}
where $\{(c_{t},d_{t}):1\leq t\leq(n-8)/6\}$ is a $(n-5)/3$-extended Skolem sequence of order $(n-8)/6$ (from Lemma \ref{Skolem}).
The fifth part $\mathcal{A}_5$ consists of the following $32$ codewords:
\begin{center}
\begin{longtable}{llll}
$\{(0,0),(6,n/2),(8,0)\}$,
&$\{(0,0),(5,0),(7,0)\}$,
&$\{(3,0),(7,0),(\infty,w+2)\}$,\\

$\{(0,0),(3,0),(4,n/2)\}$,
&$\{(1,0),(7,0),(8,0)\}$,
&$\{(0,0),(2,0),(\infty,v-4)\}$,\\

$\{(1,0),(4,0),(5,0)\}$,
&$\{(0,0),(1,0),(\infty,v)\}$,
&$\{(5,n/2),(6,0),(\infty,v+4)\}$,\\

$\{(4,0),(6,0),(6,2)\}$,
&$\{(2,0),(5,n/2),(\infty,v)\}$,
&$\{(1,0),(6,0),(\infty,v+2)\}$,\\

$\{(1,0),(2,n/2),(3,0)\}$,
&$\{(3,0),(6,0),(\infty,w)\}$,
&$\{(6,0),(7,0),(\infty,w-2)\}$,\\

$\{(3,0),(5,0),(8,0)\}$,
&$\{(2,0),(4,2),(\infty,u+2)\}$,
&$\{(4,0),(\infty,0),(\infty,n/2-3)\}$,\\

$\{(2,0),(4,0),(7,0)\}$,
&$\{(4,n/2),(8,0),(\infty,u+2)\}$,
&$\{(9,0),(\infty,0),(\infty,n/2-1)\}$,\\

$\{(2,0),(2,4),(6,4)\}$,
&$\{(2,0),(8,0),(\infty,u+4)\}$,
&$\{(0,n/2),(1,0),(2,0)\}$,\\

$\{(3,0),(4,0),(5,n/2)\}$,
&$\{(6,0),(7,n/2),(8,0)\}$,
&$\{(0,0),(4,0),(7,n/2)\}$,\\

$\{(1,n/2),(5,0),(6,0)\}$,
&$\{(2,0),(3,0),(8,n/2)\}$,
&$\{(0,0),(3,n/2),(6,0)\}$,\\

$\{(1,n/2),(4,0),(8,0)\}$,
&$\{(2,0),(5,0),(7,n/2)\}$.
\end{longtable}
\end{center}
\end{proof}

\subsection{The case of $m\equiv0\pmod{3}$}

\begin{Lemma}\label{m=3}
There exists an optimal $2$-D $(3\times n, 3, 1)$-OOC with $J^*(3\times n,3,1)$ codewords for any $n\equiv0\pmod{2}$.
\end{Lemma}

\proof For $n\equiv0,6\pmod{8}$, Lemma \ref{[3:0],n=0,6 mod8} provides the required codes. We deal with the cases of $n\equiv2,4\pmod{8}$.

For $n=2$, $\{(0,0),(1,0),(2,0)\}$ is a $2$-D $(3\times n, 3, 1)$-OOC with $J^*(3\times n,3,1)$ codewords.
For $n\geq4$, let $S=\{(a_{i},b_{i}):1\leq i\leq (n-2)/2\}$ be a $k$-extended Skolem sequence of order $(n-2)/2$ (from Lemma \ref{Skolem}). Let $\mathcal{B}$ be the following $3(n-2)/2$ base blocks.
\begin{center}
\begin{tabular}{lll}
$\{(l,0),(l,i),(l+1,a_{i}+i)\}: l\in\mathbb{Z}_3,1\leq i\leq (n-2)/2$.
\end{tabular}
\end{center}
Then $\mathcal{B}$ together with $\{(0,0),(1,0),(2,0)\}$ form a $2$-D $(3\times n, 3, 1)$-OOC with $J^*(3\times n,3,1)$ codewords.
\qed

\begin{Lemma}\label{mn=0,18mod24}
There exists an optimal $2$-D $(m\times n,3,1)$-OOC with $J^*(m\times n,3,1)$ codewords for any $m\equiv0\pmod{3}$, $n\equiv0\pmod{2}$ and $mn\equiv0,18\pmod{24}$.
\end{Lemma}
\begin{proof}
Combining the results of Lemmas \ref{Lem:2to2} and \ref{m=3}.
\end{proof}

\begin{Lemma}\label{mn=6,12mod24}
There exists an optimal $2$-D $(m\times n,3,1)$-OOC with $J^*(m\times n,3,1)$ codewords for any $m\equiv0\pmod{3}$, $n\equiv0\pmod{2}$ and $mn\equiv6,12\pmod{24}$.
\end{Lemma}

\begin{proof}
Clearly, $m\equiv0\pmod{3}$ and $mn\equiv6,12\pmod{24}$ can be divided into four cases:
\begin{enumerate}
\item[$(1)$] $m\equiv9\pmod{12}$, $n\equiv6\pmod8$; $(2)$ $m\equiv3\pmod{12}$, $n\equiv2\pmod8$;
\item[$(3)$] $m\equiv6\pmod{12}$, $n\equiv2\pmod4$; $(4)$ $m\equiv3\pmod6$, $n\equiv4\pmod8$.
\end{enumerate}

For Case $(1)$, by Lemma \ref{3-GDP}, there exists an $n$-cyclic $3$-GDP of type $(3n)^{m/3}$ with
$m(m-3)n/6-1$ base blocks. Apply Construction \ref{n-cyclic GDD} with a $2$-D $([3:0]\times n, 3, 1)$-OOC with $J^*(3\times n,3,1)$ codewords
(from Lemma \ref{[3:0],n=0,6 mod8}) to obtain a $2$-D $(m\times n,3,1)$-OOC with $J^*(m\times n,3,1)$ codewords.

For Cases $(2)$ and $(3)$, the conclusion holds for $n=2$ by Lemma \ref{n=2}. When $n>2$, by Lemma \ref{m=6t+3,regular}, there exists a $2$-regular $2$-D $(m\times n,3,1)$-OOC with $\Upsilon(m,0,n,2,3)=m^2(n-2)/6$ codewords. Then apply Construction \ref{filling3} with a $2$-D $(m\times 2, 3, 1)$-OOC with $J^*(m\times 2,3,1)$ codewords to obtain a $2$-D $(m\times n,3,1)$-OOC with $J^*(m\times n,3,1)$ codewords.

For Case $(4)$, the conclusion holds for $n=4$ by Lemma \ref{n=4}. For $n\geq12$ and $m=15$,
combining the results of Lemmas \ref{Lem:2to2} and \ref{m<12,n=2mod8}. For $n\geq12$, $m\equiv3\pmod6$, $m\geq9$ and $m\neq15$,
apply Construction \ref{filling2} with a $2$-D $([m:3]\times n,3,1)$-OOC with $\Theta(m,3,n,3)$ codewords (from Lemma \ref{[m:r],r=3}) and a $2$-D $(3\times n,3,1)$-OOC with $J^*(3\times n,3,1)$ codewords (from Lemma \ref{m=3}).
\end{proof}

\noindent
\textbf{Proof of Theorem~\ref{main}} We divide the proof into the following four cases:
\begin{enumerate}
\item[$(1)$] $m=1$ or $n\equiv1\pmod2$;~~~~~~~~~~~~~~~$(2)$ $m\equiv0\pmod{3}$, $n\equiv0\pmod2$;
\item[$(3)$] $m\equiv1,2\pmod3$, $n\equiv0\pmod4$;~~~$(4)$ $m\equiv1,2\pmod3$, $n\equiv2\pmod4$.
\end{enumerate}

For Cases $(1)$ and $(2)$, the conclusion holds by Lemmas \ref{1D}, \ref{2D:m odd}, \ref{mn=0,18mod24} and \ref{mn=6,12mod24}.

For Case $(3)$, when $n=4$, the conclusion holds by Lemma \ref{n=4}.
When $n\geq8$ and $m\in\{1,2,4,5\}$, the conclusion holds by Lemmas \ref{1D}, \ref{m=2}, \ref{m=4}, \ref{m<12,n=0mod8} and \ref{m<12,n=2mod8}. When $n\geq8$ and $m\geq7$, let $m\equiv r\pmod{6}$ where $r\in\{1,2,4,5\}$. Then apply Construction \ref{filling2} with a $2$-D $([m:r]\times n,3,1)$-OOC with $\Theta(m,r,n,3)$ codewords (from Lemma \ref{[m:r],n=4t}) and a $2$-D $(r\times n,3,1)$-OOC with $J^*(r\times n,3,1)$ codewords.

For Case $(4)$, when $m\in\{1,2,4,5,7,8,10,11\}$, the conclusion holds by Lemmas \ref{1D}, \ref{m=2}, \ref{m=4}, \ref{m<12,n=2mod8} and \ref{m<12,n=6mod8}. When $m\geq13$, let $m\equiv r\pmod{12}$ where $r\in\{1,2,4,5,7,8,10,11\}$. Then apply Construction \ref{filling2} with a $2$-D $([m:r]\times n,3,1)$-OOC with $\Theta(m,r,n,3)$ codewords (from Lemmas \ref{[m:r],n=8t+6} and \ref{[m:s]}) and a $2$-D $(r\times n,3,1)$-OOC with $J^*(r\times n,3,1)$ codewords. \qed



\appendix

\section{Appendix: Base blocks in the proof of Lemma \ref{3n^4t9n^1}}\label{6^4t12^1:t=1,2}

For $(t,r)\in\{(1,1),(1,2),(1,5),(2,1),(2,2),(2,5)\}$, take $X=(I_{12t}\bigcup\{\infty_0,\infty_1,\ldots,\infty_{r-1}\})\times\mathbb{Z}_2$ with the group set $\mathcal{G}=\{\{i,i+4t,i+8t\}\times\mathbb{Z}_2:0\leq i\leq 4t-1\}\bigcup\{\{\infty_0,\infty_1,\ldots,\infty_{r-1}\}\times\mathbb{Z}_2\}$. All the $4t(12t+2r-3)$ base blocks are listed below.

\begin{center}
\begin{longtable}{lllll}
$(t,r)=(1,1)$:
&$\{(0,0),(1,0),(2,0)\}$,
&$\{(0,0),(3,0),(2,1)\}$,
&$\{(0,0),(5,0),(7,0)\}$,\\
&$\{(0,0),(6,0),(3,1)\}$,
&$\{(0,0),(9,0),(10,0)\}$,
&$\{(0,0),(11,0),(\infty_0,0)\}$,\\
&$\{(0,0),(1,1),(7,1)\}$,
&$\{(0,0),(5,1),(6,1)\}$,
&$\{(0,0),(9,1),(11,1)\}$,\\
&$\{(0,0),(10,1),(\infty_0,1)\}$,
&$\{(1,0),(3,0),(\infty_0,1)\}$,
&$\{(1,0),(4,0),(11,0)\}$,\\
&$\{(1,0),(6,0),(4,1)\}$,
&$\{(1,0),(8,0),(\infty_0,0)\}$,
&$\{(1,0),(10,0),(3,1)\}$,\\
&$\{(1,0),(2,1),(7,1)\}$,
&$\{(1,0),(6,1),(11,1)\}$,
&$\{(1,0),(8,1),(10,1)\}$,\\
&$\{(2,0),(3,0),(8,0)\}$,
&$\{(2,0),(4,0),(5,1)\}$,
&$\{(2,0),(5,0),(11,0)\}$,\\
&$\{(2,0),(9,0),(8,1)\}$,
&$\{(2,0),(\infty_0,0),(11,1)\}$,
&$\{(2,0),(4,1),(\infty_0,1)\}$,\\
&$\{(2,0),(7,1),(9,1)\}$,
&$\{(3,0),(4,0),(5,0)\}$,
&$\{(3,0),(6,0),(8,1)\}$,\\
&$\{(3,0),(9,0),(\infty_0,0)\}$,
&$\{(3,0),(10,0),(5,1)\}$,
&$\{(3,0),(4,1),(9,1)\}$,\\
&$\{(4,0),(6,0),(9,1)\}$,
&$\{(4,0),(7,0),(\infty_0,1)\}$,
&$\{(4,0),(10,0),(7,1)\}$,\\
&$\{(4,0),(10,1),(11,1)\}$,
&$\{(5,0),(8,0),(11,1)\}$,
&$\{(5,0),(10,0),(\infty_0,1)\}$,\\
&$\{(5,0),(\infty_0,0),(6,1)\}$,
&$\{(5,0),(7,1),(8,1)\}$,
&$\{(6,0),(7,0),(\infty_0,0)\}$,\\
&$\{(6,0),(8,0),(7,1)\}$,
&$\{(6,0),(9,0),(11,1)\}$,
&$\{(7,0),(10,0),(9,1)\}$,\\
&$\{(8,0),(9,0),(\infty_0,1)\}$,
&$\{(8,0),(11,0),(10,1)\}$.\\

$(t,r)=(1,2):$
&$\{(0,0),(1,0),(2,0)\}$,
&$\{(0,0),(3,0),(2,1)\}$,
&$\{(0,0),(5,0),(6,1)\}$,\\
&$\{(0,0),(6,0),(3,1)\}$,
&$\{(0,0),(7,0),(\infty_0,0)\}$,
&$\{(0,0),(9,0),(7,1)\}$,\\
&$\{(0,0),(10,0),(\infty_1,0)\}$,
&$\{(0,0),(11,0),(10,1)\}$,
&$\{(0,0),(1,1),(11,1)\}$,\\
&$\{(0,0),(5,1),(\infty_1,1)\}$,
&$\{(0,0),(9,1),(\infty_0,1)\}$,
&$\{(1,0),(3,0),(8,0)\}$,\\
&$\{(1,0),(4,0),(\infty_0,1)\}$,
&$\{(1,0),(6,0),(\infty_0,0)\}$,
&$\{(1,0),(7,0),(6,1)\}$,\\
&$\{(1,0),(10,0),(4,1)\}$,
&$\{(1,0),(\infty_1,0),(8,1)\}$,
&$\{(1,0),(2,1),(11,1)\}$,\\
&$\{(1,0),(3,1),(\infty_1,1)\}$,
&$\{(1,0),(7,1),(10,1)\}$,
&$\{(2,0),(3,0),(\infty_1,1)\}$,\\
&$\{(2,0),(4,0),(7,1)\}$,
&$\{(2,0),(5,0),(4,1)\}$,
&$\{(2,0),(7,0),(9,0)\}$,\\
&$\{(2,0),(8,0),(\infty_1,0)\}$,
&$\{(2,0),(\infty_0,0),(9,1)\}$,
&$\{(2,0),(5,1),(11,1)\}$,\\
&$\{(2,0),(8,1),(\infty_0,1)\}$,
&$\{(3,0),(4,0),(9,1)\}$,
&$\{(3,0),(5,0),(8,1)\}$,\\
&$\{(3,0),(6,0),(9,0)\}$,
&$\{(3,0),(10,0),(5,1)\}$,
&$\{(3,0),(\infty_0,0),(10,1)\}$,\\
&$\{(3,0),(4,1),(\infty_0,1)\}$,
&$\{(4,0),(5,0),(10,0)\}$,
&$\{(4,0),(6,0),(11,1)\}$,\\
&$\{(4,0),(7,0),(\infty_1,1)\}$,
&$\{(4,0),(9,0),(6,1)\}$,
&$\{(4,0),(11,0),(\infty_1,0)\}$,\\
&$\{(5,0),(6,0),(\infty_1,1)\}$,
&$\{(5,0),(7,0),(\infty_0,1)\}$,
&$\{(5,0),(8,0),(7,1)\}$,\\
&$\{(5,0),(\infty_0,0),(11,1)\}$,
&$\{(6,0),(7,0),(\infty_1,0)\}$,
&$\{(6,0),(8,0),(\infty_0,1)\}$,\\
&$\{(6,0),(11,0),(8,1)\}$,
&$\{(7,0),(8,0),(10,1)\}$,
&$\{(8,0),(9,0),(10,0)\}$,\\
&$\{(8,0),(11,0),(9,1)\}$,
&$\{(9,0),(11,0),(\infty_1,1)\}$,
&$\{(9,0),(\infty_1,0),(10,1)\}$,\\
&$\{(10,0),(11,0),(\infty_0,0)\}$.\\

$(t,r)=(1,5):$
&$\{(0,0),(1,0),(2,0)\}$,
&$\{(0,0),(3,0),(9,0)\}$,
&$\{(0,0),(5,0),(6,0)\}$,\\
&$\{(0,0),(7,0),(10,0)\}$,
&$\{(0,0),(11,0),(\infty_2,0)\}$,
&$\{(0,0),(\infty_0,0),(10,1)\}$,\\
&$\{(0,0),(\infty_1,0),(1,1)\}$,
&$\{(0,0),(\infty_3,0),(3,1)\}$,
&$\{(0,0),(\infty_4,0),(2,1)\}$,\\
&$\{(0,0),(5,1),(\infty_0,1)\}$,
&$\{(0,0),(6,1),(\infty_3,1)\}$,
&$\{(0,0),(7,1),(\infty_1,1)\}$,\\
&$\{(0,0),(9,1),(\infty_2,1)\}$,
&$\{(0,0),(11,1),(\infty_4,1)\}$,
&$\{(1,0),(3,0),(\infty_0,1)\}$,\\
&$\{(1,0),(4,0),(6,0)\}$,
&$\{(1,0),(7,0),(\infty_2,1)\}$,
&$\{(1,0),(8,0),(10,0)\}$,\\
&$\{(1,0),(11,0),(\infty_3,1)\}$,
&$\{(1,0),(\infty_0,0),(4,1)\}$,
&$\{(1,0),(\infty_1,0),(2,1)\}$,\\
&$\{(1,0),(\infty_2,0),(3,1)\}$,
&$\{(1,0),(\infty_3,0),(6,1)\}$,
&$\{(1,0),(\infty_4,0),(7,1)\}$,\\
&$\{(1,0),(8,1),(11,1)\}$,
&$\{(1,0),(10,1),(\infty_4,1)\}$,
&$\{(2,0),(3,0),(4,0)\}$,\\
&$\{(2,0),(5,0),(\infty_2,1)\}$,
&$\{(2,0),(7,0),(5,1)\}$,
&$\{(2,0),(8,0),(\infty_0,1)\}$,\\
&$\{(2,0),(9,0),(11,1)\}$,
&$\{(2,0),(11,0),(\infty_0,0)\}$,
&$\{(2,0),(\infty_1,0),(3,1)\}$,\\
&$\{(2,0),(\infty_2,0),(4,1)\}$,
&$\{(2,0),(\infty_3,0),(7,1)\}$,
&$\{(2,0),(\infty_4,0),(8,1)\}$,\\
&$\{(2,0),(9,1),(\infty_3,1)\}$,
&$\{(3,0),(5,0),(8,0)\}$,
&$\{(3,0),(6,0),(\infty_0,0)\}$,\\
&$\{(3,0),(10,0),(\infty_4,1)\}$,
&$\{(3,0),(\infty_1,0),(4,1)\}$,
&$\{(3,0),(\infty_2,0),(8,1)\}$,\\
&$\{(3,0),(\infty_3,0),(5,1)\}$,
&$\{(3,0),(\infty_4,0),(6,1)\}$,
&$\{(3,0),(9,1),(10,1)\}$,\\
&$\{(4,0),(5,0),(7,0)\}$,
&$\{(4,0),(9,0),(\infty_4,0)\}$,
&$\{(4,0),(10,0),(\infty_0,0)\}$,\\
&$\{(4,0),(11,0),(\infty_1,0)\}$,
&$\{(4,0),(\infty_2,0),(10,1)\}$,
&$\{(4,0),(\infty_3,0),(9,1)\}$,\\
&$\{(4,0),(5,1),(\infty_3,1)\}$,
&$\{(4,0),(6,1),(11,1)\}$,
&$\{(4,0),(7,1),(\infty_4,1)\}$,\\
&$\{(5,0),(10,0),(\infty_2,0)\}$,
&$\{(5,0),(11,0),(\infty_0,1)\}$,
&$\{(5,0),(\infty_1,0),(10,1)\}$,\\
&$\{(5,0),(\infty_4,0),(11,1)\}$,
&$\{(5,0),(6,1),(\infty_1,1)\}$,
&$\{(5,0),(8,1),(\infty_4,1)\}$,\\
&$\{(6,0),(7,0),(\infty_2,0)\}$,
&$\{(6,0),(8,0),(11,1)\}$,
&$\{(6,0),(9,0),(\infty_1,1)\}$,\\
&$\{(6,0),(\infty_4,0),(9,1)\}$,
&$\{(6,0),(7,1),(\infty_0,1)\}$,
&$\{(6,0),(8,1),(\infty_2,1)\}$,\\
&$\{(7,0),(8,0),(\infty_1,1)\}$,
&$\{(7,0),(9,0),(10,1)\}$,
&$\{(7,0),(\infty_3,0),(8,1)\}$,\\
&$\{(7,0),(9,1),(\infty_0,1)\}$,
&$\{(8,0),(9,0),(\infty_1,0)\}$,
&$\{(8,0),(\infty_0,0),(9,1)\}$,\\
&$\{(8,0),(\infty_3,0),(10,1)\}$,
&$\{(9,0),(11,0),(\infty_2,1)\}$,
&$\{(10,0),(11,0),(\infty_3,0)\}$,\\
&$\{(10,0),(\infty_1,0),(11,1)\}$.\\

$(t,r)=(2,1):$
&$\{(0,0),(1,0),(2,0)\}$,
&$\{(0,0),(3,0),(17,1)\}$,
&$\{(0,0),(4,0),(17,0)\}$,\\
&$\{(0,0),(5,0),(9,1)\}$,
&$\{(0,0),(6,0),(7,0)\}$,
&$\{(0,0),(9,0),(19,0)\}$,\\
&$\{(0,0),(10,0),(\infty_0,0)\}$,
&$\{(0,0),(11,0),(4,1)\}$,
&$\{(0,0),(12,0),(13,0)\}$,\\
&$\{(0,0),(14,0),(15,0)\}$,
&$\{(0,0),(18,0),(20,0)\}$,
&$\{(0,0),(21,0),(22,0)\}$,\\
&$\{(0,0),(23,0),(20,1)\}$,
&$\{(0,0),(1,1),(6,1)\}$,
&$\{(0,0),(2,1),(3,1)\}$,\\
&$\{(0,0),(5,1),(7,1)\}$,
&$\{(0,0),(10,1),(19,1)\}$,
&$\{(0,0),(11,1),(21,1)\}$,\\
&$\{(0,0),(12,1),(18,1)\}$,
&$\{(0,0),(13,1),(14,1)\}$,
&$\{(0,0),(15,1),(22,1)\}$,\\
&$\{(0,0),(23,1),(\infty_0,1)\}$,
&$\{(1,0),(3,0),(4,0)\}$,
&$\{(1,0),(5,0),(8,0)\}$,\\
&$\{(1,0),(7,0),(10,0)\}$,
&$\{(1,0),(11,0),(18,1)\}$,
&$\{(1,0),(12,0),(7,1)\}$,\\
&$\{(1,0),(13,0),(22,0)\}$,
&$\{(1,0),(14,0),(16,0)\}$,
&$\{(1,0),(15,0),(18,0)\}$,\\
&$\{(1,0),(19,0),(20,1)\}$,
&$\{(1,0),(20,0),(22,1)\}$,
&$\{(1,0),(21,0),(2,1)\}$,\\
&$\{(1,0),(23,0),(3,1)\}$,
&$\{(1,0),(\infty_0,0),(13,1)\}$,
&$\{(1,0),(4,1),(5,1)\}$,\\
&$\{(1,0),(6,1),(8,1)\}$,
&$\{(1,0),(10,1),(11,1)\}$,
&$\{(1,0),(12,1),(\infty_0,1)\}$,\\
&$\{(1,0),(14,1),(19,1)\}$,
&$\{(1,0),(15,1),(21,1)\}$,
&$\{(1,0),(16,1),(23,1)\}$,\\
&$\{(2,0),(4,0),(11,0)\}$,
&$\{(2,0),(5,0),(22,0)\}$,
&$\{(2,0),(6,0),(8,1)\}$,\\
&$\{(2,0),(7,0),(\infty_0,0)\}$,
&$\{(2,0),(8,0),(14,0)\}$,
&$\{(2,0),(9,0),(16,1)\}$,\\
&$\{(2,0),(12,0),(15,0)\}$,
&$\{(2,0),(13,0),(16,0)\}$,
&$\{(2,0),(17,0),(19,0)\}$,\\
&$\{(2,0),(20,0),(\infty_0,1)\}$,
&$\{(2,0),(21,0),(19,1)\}$,
&$\{(2,0),(23,0),(4,1)\}$,\\
&$\{(2,0),(3,1),(5,1)\}$,
&$\{(2,0),(6,1),(12,1)\}$,
&$\{(2,0),(7,1),(9,1)\}$,\\
&$\{(2,0),(11,1),(13,1)\}$,
&$\{(2,0),(14,1),(17,1)\}$,
&$\{(2,0),(15,1),(20,1)\}$,\\
&$\{(2,0),(22,1),(23,1)\}$,
&$\{(3,0),(6,0),(16,0)\}$,
&$\{(3,0),(7,0),(10,1)\}$,\\
&$\{(3,0),(8,0),(15,1)\}$,
&$\{(3,0),(9,0),(18,1)\}$,
&$\{(3,0),(10,0),(13,1)\}$,\\
&$\{(3,0),(12,0),(14,0)\}$,
&$\{(3,0),(13,0),(16,1)\}$,
&$\{(3,0),(15,0),(20,1)\}$,\\
&$\{(3,0),(17,0),(\infty_0,0)\}$,
&$\{(3,0),(18,0),(21,0)\}$,
&$\{(3,0),(20,0),(21,1)\}$,\\
&$\{(3,0),(22,0),(4,1)\}$,
&$\{(3,0),(23,0),(5,1)\}$,
&$\{(3,0),(6,1),(9,1)\}$,\\
&$\{(3,0),(7,1),(8,1)\}$,
&$\{(3,0),(12,1),(22,1)\}$,
&$\{(3,0),(14,1),(\infty_0,1)\}$,\\
&$\{(4,0),(6,0),(\infty_0,1)\}$,
&$\{(4,0),(7,0),(5,1)\}$,
&$\{(4,0),(8,0),(23,0)\}$,\\
&$\{(4,0),(9,0),(10,1)\}$,
&$\{(4,0),(10,0),(13,0)\}$,
&$\{(4,0),(14,0),(18,0)\}$,\\
&$\{(4,0),(15,0),(16,1)\}$,
&$\{(4,0),(16,0),(19,0)\}$,
&$\{(4,0),(21,0),(6,1)\}$,\\
&$\{(4,0),(22,0),(7,1)\}$,
&$\{(4,0),(\infty_0,0),(14,1)\}$,
&$\{(4,0),(8,1),(13,1)\}$,\\
&$\{(4,0),(9,1),(15,1)\}$,
&$\{(4,0),(17,1),(18,1)\}$,
&$\{(4,0),(19,1),(21,1)\}$,\\
&$\{(5,0),(6,0),(23,0)\}$,
&$\{(5,0),(9,0),(11,1)\}$,
&$\{(5,0),(10,0),(6,1)\}$,\\
&$\{(5,0),(11,0),(20,1)\}$,
&$\{(5,0),(12,0),(8,1)\}$,
&$\{(5,0),(14,0),(10,1)\}$,\\
&$\{(5,0),(15,0),(\infty_0,0)\}$,
&$\{(5,0),(16,0),(12,1)\}$,
&$\{(5,0),(17,0),(14,1)\}$,\\
&$\{(5,0),(18,0),(22,1)\}$,
&$\{(5,0),(19,0),(20,0)\}$,
&$\{(5,0),(15,1),(19,1)\}$,\\
&$\{(5,0),(16,1),(17,1)\}$,
&$\{(5,0),(18,1),(\infty_0,1)\}$,
&$\{(6,0),(10,0),(16,1)\}$,\\
&$\{(6,0),(11,0),(\infty_0,0)\}$,
&$\{(6,0),(13,0),(17,0)\}$,
&$\{(6,0),(15,0),(9,1)\}$,\\
&$\{(6,0),(18,0),(19,0)\}$,
&$\{(6,0),(20,0),(17,1)\}$,
&$\{(6,0),(21,0),(7,1)\}$,\\
&$\{(6,0),(11,1),(20,1)\}$,
&$\{(6,0),(12,1),(19,1)\}$,
&$\{(6,0),(13,1),(15,1)\}$,\\
&$\{(6,0),(18,1),(23,1)\}$,
&$\{(7,0),(11,0),(\infty_0,1)\}$,
&$\{(7,0),(12,0),(16,0)\}$,\\
&$\{(7,0),(13,0),(8,1)\}$,
&$\{(7,0),(14,0),(20,0)\}$,
&$\{(7,0),(17,0),(21,0)\}$,\\
&$\{(7,0),(18,0),(13,1)\}$,
&$\{(7,0),(19,0),(18,1)\}$,
&$\{(7,0),(22,0),(19,1)\}$,\\
&$\{(7,0),(9,1),(14,1)\}$,
&$\{(7,0),(11,1),(16,1)\}$,
&$\{(7,0),(17,1),(20,1)\}$,\\
&$\{(8,0),(9,0),(19,1)\}$,
&$\{(8,0),(10,0),(12,0)\}$,
&$\{(8,0),(11,0),(14,1)\}$,\\
&$\{(8,0),(15,0),(17,0)\}$,
&$\{(8,0),(18,0),(20,1)\}$,
&$\{(8,0),(19,0),(17,1)\}$,\\
&$\{(8,0),(20,0),(\infty_0,0)\}$,
&$\{(8,0),(21,0),(22,1)\}$,
&$\{(8,0),(22,0),(23,1)\}$,\\
&$\{(8,0),(9,1),(10,1)\}$,
&$\{(8,0),(11,1),(18,1)\}$,
&$\{(8,0),(21,1),(\infty_0,1)\}$,\\
&$\{(9,0),(11,0),(22,0)\}$,
&$\{(9,0),(12,0),(13,1)\}$,
&$\{(9,0),(13,0),(20,1)\}$,\\
&$\{(9,0),(16,0),(21,1)\}$,
&$\{(9,0),(18,0),(23,1)\}$,
&$\{(9,0),(20,0),(14,1)\}$,\\
&$\{(9,0),(21,0),(\infty_0,1)\}$,
&$\{(9,0),(23,0),(12,1)\}$,
&$\{(9,0),(\infty_0,0),(22,1)\}$,\\
&$\{(10,0),(14,0),(15,1)\}$,
&$\{(10,0),(15,0),(\infty_0,1)\}$,
&$\{(10,0),(16,0),(20,0)\}$,\\
&$\{(10,0),(17,0),(11,1)\}$,
&$\{(10,0),(21,0),(23,1)\}$,
&$\{(10,0),(22,0),(12,1)\}$,\\
&$\{(10,0),(23,0),(17,1)\}$,
&$\{(10,0),(19,1),(22,1)\}$,
&$\{(10,0),(20,1),(21,1)\}$,\\
&$\{(11,0),(12,0),(15,1)\}$,
&$\{(11,0),(14,0),(16,1)\}$,
&$\{(11,0),(15,0),(21,1)\}$,\\
&$\{(11,0),(17,0),(22,1)\}$,
&$\{(11,0),(23,0),(13,1)\}$,
&$\{(11,0),(12,1),(23,1)\}$,\\
&$\{(12,0),(17,0),(18,1)\}$,
&$\{(12,0),(21,0),(17,1)\}$,
&$\{(12,0),(14,1),(21,1)\}$,\\
&$\{(12,0),(19,1),(\infty_0,1)\}$,
&$\{(13,0),(18,0),(14,1)\}$,
&$\{(13,0),(19,0),(15,1)\}$,\\
&$\{(13,0),(20,0),(23,0)\}$,
&$\{(13,0),(\infty_0,0),(19,1)\}$,
&$\{(13,0),(17,1),(22,1)\}$,\\
&$\{(14,0),(23,0),(19,1)\}$,
&$\{(14,0),(21,1),(23,1)\}$,
&$\{(15,0),(16,0),(17,1)\}$,\\
&$\{(15,0),(18,1),(22,1)\}$,
&$\{(16,0),(18,0),(\infty_0,1)\}$,
&$\{(16,0),(21,0),(18,1)\}$,\\
&$\{(16,0),(22,0),(\infty_0,0)\}$,
&$\{(16,0),(19,1),(23,1)\}$,
&$\{(16,0),(20,1),(22,1)\}$,\\
&$\{(17,0),(23,0),(\infty_0,1)\}$.\\

$(t,r)=(2,2):$
&$\{(0,0),(1,0),(2,0)\}$,
&$\{(0,0),(3,0),(15,1)\}$,
&$\{(0,0),(4,0),(7,0)\}$,\\
&$\{(0,0),(5,0),(17,1)\}$,
&$\{(0,0),(6,0),(7,1)\}$,
&$\{(0,0),(9,0),(10,1)\}$,\\
&$\{(0,0),(10,0),(20,0)\}$,
&$\{(0,0),(11,0),(12,1)\}$,
&$\{(0,0),(12,0),(5,1)\}$,\\
&$\{(0,0),(13,0),(1,1)\}$,
&$\{(0,0),(14,0),(\infty_0,0)\}$,
&$\{(0,0),(15,0),(11,1)\}$,\\
&$\{(0,0),(17,0),(18,0)\}$,
&$\{(0,0),(19,0),(13,1)\}$,
&$\{(0,0),(21,0),(6,1)\}$,\\
&$\{(0,0),(22,0),(3,1)\}$,
&$\{(0,0),(23,0),(\infty_1,0)\}$,
&$\{(0,0),(2,1),(9,1)\}$,\\
&$\{(0,0),(4,1),(14,1)\}$,
&$\{(0,0),(18,1),(19,1)\}$,
&$\{(0,0),(20,1),(21,1)\}$,\\
&$\{(0,0),(22,1),(\infty_1,1)\}$,
&$\{(0,0),(23,1),(\infty_0,1)\}$,
&$\{(1,0),(3,0),(6,1)\}$,\\
&$\{(1,0),(4,0),(6,0)\}$,
&$\{(1,0),(5,0),(7,0)\}$,
&$\{(1,0),(8,0),(22,1)\}$,\\
&$\{(1,0),(10,0),(11,0)\}$,
&$\{(1,0),(12,0),(7,1)\}$,
&$\{(1,0),(13,0),(14,1)\}$,\\
&$\{(1,0),(14,0),(16,1)\}$,
&$\{(1,0),(15,0),(16,0)\}$,
&$\{(1,0),(18,0),(20,1)\}$,\\
&$\{(1,0),(19,0),(23,1)\}$,
&$\{(1,0),(20,0),(\infty_0,1)\}$,
&$\{(1,0),(21,0),(8,1)\}$,\\
&$\{(1,0),(22,0),(15,1)\}$,
&$\{(1,0),(23,0),(5,1)\}$,
&$\{(1,0),(\infty_0,0),(3,1)\}$,\\
&$\{(1,0),(\infty_1,0),(2,1)\}$,
&$\{(1,0),(4,1),(10,1)\}$,
&$\{(1,0),(11,1),(21,1)\}$,\\
&$\{(1,0),(12,1),(18,1)\}$,
&$\{(1,0),(19,1),(\infty_1,1)\}$,
&$\{(2,0),(3,0),(15,0)\}$,\\
&$\{(2,0),(4,0),(17,0)\}$,
&$\{(2,0),(5,0),(6,0)\}$,
&$\{(2,0),(7,0),(16,0)\}$,\\
&$\{(2,0),(8,0),(6,1)\}$,
&$\{(2,0),(11,0),(12,0)\}$,
&$\{(2,0),(13,0),(23,1)\}$,\\
&$\{(2,0),(14,0),(19,0)\}$,
&$\{(2,0),(20,0),(22,0)\}$,
&$\{(2,0),(21,0),(14,1)\}$,\\
&$\{(2,0),(23,0),(3,1)\}$,
&$\{(2,0),(\infty_0,0),(11,1)\}$,
&$\{(2,0),(\infty_1,0),(4,1)\}$,\\
&$\{(2,0),(5,1),(12,1)\}$,
&$\{(2,0),(7,1),(8,1)\}$,
&$\{(2,0),(9,1),(19,1)\}$,\\
&$\{(2,0),(13,1),(15,1)\}$,
&$\{(2,0),(16,1),(17,1)\}$,
&$\{(2,0),(20,1),(\infty_0,1)\}$,\\
&$\{(2,0),(21,1),(22,1)\}$,
&$\{(3,0),(4,0),(13,0)\}$,
&$\{(3,0),(5,0),(4,1)\}$,\\
&$\{(3,0),(6,0),(7,0)\}$,
&$\{(3,0),(8,0),(9,0)\}$,
&$\{(3,0),(10,0),(12,0)\}$,\\
&$\{(3,0),(14,0),(\infty_1,0)\}$,
&$\{(3,0),(16,0),(13,1)\}$,
&$\{(3,0),(17,0),(20,0)\}$,\\
&$\{(3,0),(18,0),(12,1)\}$,
&$\{(3,0),(21,0),(10,1)\}$,
&$\{(3,0),(22,0),(16,1)\}$,\\
&$\{(3,0),(23,0),(14,1)\}$,
&$\{(3,0),(\infty_0,0),(17,1)\}$,
&$\{(3,0),(5,1),(18,1)\}$,\\
&$\{(3,0),(7,1),(\infty_1,1)\}$,
&$\{(3,0),(8,1),(20,1)\}$,
&$\{(3,0),(9,1),(21,1)\}$,\\
&$\{(4,0),(5,0),(8,0)\}$,
&$\{(4,0),(9,0),(6,1)\}$,
&$\{(4,0),(11,0),(7,1)\}$,\\
&$\{(4,0),(15,0),(16,1)\}$,
&$\{(4,0),(16,0),(17,1)\}$,
&$\{(4,0),(18,0),(14,1)\}$,\\
&$\{(4,0),(19,0),(8,1)\}$,
&$\{(4,0),(21,0),(23,0)\}$,
&$\{(4,0),(22,0),(23,1)\}$,\\
&$\{(4,0),(\infty_0,0),(9,1)\}$,
&$\{(4,0),(\infty_1,0),(10,1)\}$,
&$\{(4,0),(11,1),(13,1)\}$,\\
&$\{(4,0),(15,1),(18,1)\}$,
&$\{(4,0),(19,1),(21,1)\}$,
&$\{(4,0),(22,1),(\infty_0,1)\}$,\\
&$\{(5,0),(9,0),(20,0)\}$,
&$\{(5,0),(10,0),(22,1)\}$,
&$\{(5,0),(11,0),(15,0)\}$,\\
&$\{(5,0),(14,0),(\infty_0,1)\}$,
&$\{(5,0),(16,0),(6,1)\}$,
&$\{(5,0),(17,0),(7,1)\}$,\\
&$\{(5,0),(19,0),(14,1)\}$,
&$\{(5,0),(22,0),(\infty_1,1)\}$,
&$\{(5,0),(23,0),(8,1)\}$,\\
&$\{(5,0),(\infty_0,0),(15,1)\}$,
&$\{(5,0),(\infty_1,0),(16,1)\}$,
&$\{(5,0),(9,1),(18,1)\}$,\\
&$\{(5,0),(10,1),(19,1)\}$,
&$\{(5,0),(11,1),(20,1)\}$,
&$\{(6,0),(8,0),(10,0)\}$,\\
&$\{(6,0),(9,0),(13,0)\}$,
&$\{(6,0),(11,0),(16,0)\}$,
&$\{(6,0),(12,0),(13,1)\}$,\\
&$\{(6,0),(15,0),(10,1)\}$,
&$\{(6,0),(17,0),(21,0)\}$,
&$\{(6,0),(18,0),(20,0)\}$,\\
&$\{(6,0),(19,0),(23,0)\}$,
&$\{(6,0),(\infty_0,0),(12,1)\}$,
&$\{(6,0),(\infty_1,0),(15,1)\}$,\\
&$\{(6,0),(11,1),(23,1)\}$,
&$\{(6,0),(17,1),(19,1)\}$,
&$\{(6,0),(18,1),(\infty_0,1)\}$,\\
&$\{(6,0),(20,1),(\infty_1,1)\}$,
&$\{(7,0),(9,0),(14,0)\}$,
&$\{(7,0),(10,0),(\infty_0,1)\}$,\\
&$\{(7,0),(11,0),(17,0)\}$,
&$\{(7,0),(12,0),(19,0)\}$,
&$\{(7,0),(13,0),(22,1)\}$,\\
&$\{(7,0),(18,0),(8,1)\}$,
&$\{(7,0),(20,0),(19,1)\}$,
&$\{(7,0),(21,0),(\infty_0,0)\}$,\\
&$\{(7,0),(22,0),(9,1)\}$,
&$\{(7,0),(10,1),(21,1)\}$,
&$\{(7,0),(13,1),(20,1)\}$,\\
&$\{(7,0),(14,1),(16,1)\}$,
&$\{(7,0),(18,1),(\infty_1,1)\}$,
&$\{(8,0),(11,0),(9,1)\}$,\\
&$\{(8,0),(12,0),(\infty_1,1)\}$,
&$\{(8,0),(13,0),(22,0)\}$,
&$\{(8,0),(14,0),(15,1)\}$,\\
&$\{(8,0),(15,0),(12,1)\}$,
&$\{(8,0),(17,0),(10,1)\}$,
&$\{(8,0),(18,0),(21,0)\}$,\\
&$\{(8,0),(19,0),(\infty_0,0)\}$,
&$\{(8,0),(23,0),(11,1)\}$,
&$\{(8,0),(\infty_1,0),(20,1)\}$,\\
&$\{(8,0),(13,1),(\infty_0,1)\}$,
&$\{(8,0),(14,1),(17,1)\}$,
&$\{(9,0),(10,0),(12,1)\}$,\\
&$\{(9,0),(11,0),(\infty_1,0)\}$,
&$\{(9,0),(12,0),(23,0)\}$,
&$\{(9,0),(15,0),(22,0)\}$,\\
&$\{(9,0),(16,0),(18,1)\}$,
&$\{(9,0),(\infty_0,0),(19,1)\}$,
&$\{(9,0),(13,1),(23,1)\}$,\\
&$\{(9,0),(14,1),(20,1)\}$,
&$\{(9,0),(15,1),(21,1)\}$,
&$\{(9,0),(16,1),(\infty_1,1)\}$,\\
&$\{(10,0),(13,0),(\infty_1,0)\}$,
&$\{(10,0),(14,0),(20,1)\}$,
&$\{(10,0),(15,0),(17,0)\}$,\\
&$\{(10,0),(16,0),(19,1)\}$,
&$\{(10,0),(22,0),(23,0)\}$,
&$\{(10,0),(\infty_0,0),(13,1)\}$,\\
&$\{(10,0),(11,1),(14,1)\}$,
&$\{(10,0),(16,1),(23,1)\}$,
&$\{(11,0),(18,0),(13,1)\}$,\\
&$\{(11,0),(22,0),(17,1)\}$,
&$\{(11,0),(\infty_0,0),(22,1)\}$,
&$\{(11,0),(14,1),(18,1)\}$,\\
&$\{(11,0),(16,1),(20,1)\}$,
&$\{(11,0),(21,1),(\infty_1,1)\}$,
&$\{(12,0),(13,0),(17,0)\}$,\\
&$\{(12,0),(14,0),(15,0)\}$,
&$\{(12,0),(16,0),(21,0)\}$,
&$\{(12,0),(22,0),(19,1)\}$,\\
&$\{(12,0),(\infty_0,0),(16,1)\}$,
&$\{(12,0),(\infty_1,0),(21,1)\}$,
&$\{(12,0),(14,1),(23,1)\}$,\\
&$\{(12,0),(17,1),(22,1)\}$,
&$\{(13,0),(14,0),(\infty_1,1)\}$,
&$\{(13,0),(16,0),(19,0)\}$,\\
&$\{(13,0),(18,0),(17,1)\}$,
&$\{(13,0),(15,1),(20,1)\}$,
&$\{(14,0),(21,0),(17,1)\}$,\\
&$\{(15,0),(19,0),(21,1)\}$,
&$\{(15,0),(\infty_0,0),(18,1)\}$,
&$\{(15,0),(\infty_1,0),(17,1)\}$,\\
&$\{(15,0),(19,1),(20,1)\}$,
&$\{(16,0),(18,0),(23,1)\}$,
&$\{(16,0),(22,0),(20,1)\}$,\\
&$\{(16,0),(\infty_0,0),(21,1)\}$,
&$\{(17,0),(23,0),(20,1)\}$,
&$\{(17,0),(\infty_0,0),(23,1)\}$,\\
&$\{(17,0),(\infty_1,0),(19,1)\}$,
&$\{(18,0),(22,0),(21,1)\}$,
&$\{(18,0),(23,0),(\infty_1,1)\}$,\\
&$\{(18,0),(19,1),(22,1)\}$,
&$\{(20,0),(23,0),(21,1)\}$.\\

$(t,r)=(2,5):$
&$\{(0,0),(1,0),(2,0)\}$,
&$\{(0,0),(3,0),(12,1)\}$,
&$\{(0,0),(4,0),(23,0)\}$,\\
&$\{(0,0),(5,0),(12,0)\}$,
&$\{(0,0),(6,0),(23,1)\}$,
&$\{(0,0),(7,0),(1,1)\}$,\\
&$\{(0,0),(9,0),(10,0)\}$,
&$\{(0,0),(11,0),(\infty_0,1)\}$,
&$\{(0,0),(13,0),(14,0)\}$,\\
&$\{(0,0),(15,0),(21,1)\}$,
&$\{(0,0),(17,0),(18,1)\}$,
&$\{(0,0),(18,0),(21,0)\}$,\\
&$\{(0,0),(19,0),(20,0)\}$,
&$\{(0,0),(22,0),(\infty_0,0)\}$,
&$\{(0,0),(\infty_1,0),(2,1)\}$,\\
&$\{(0,0),(\infty_2,0),(14,1)\}$,
&$\{(0,0),(\infty_3,0),(3,1)\}$,
&$\{(0,0),(\infty_4,0),(9,1)\}$,\\
&$\{(0,0),(4,1),(\infty_1,1)\}$,
&$\{(0,0),(5,1),(6,1)\}$,
&$\{(0,0),(7,1),(13,1)\}$,\\
&$\{(0,0),(10,1),(11,1)\}$,
&$\{(0,0),(15,1),(22,1)\}$,
&$\{(0,0),(17,1),(\infty_2,1)\}$,\\
&$\{(0,0),(19,1),(\infty_3,1)\}$,
&$\{(0,0),(20,1),(\infty_4,1)\}$,
&$\{(1,0),(3,0),(\infty_1,1)\}$,\\
&$\{(1,0),(4,0),(13,0)\}$,
&$\{(1,0),(5,0),(18,0)\}$,
&$\{(1,0),(6,0),(2,1)\}$,\\
&$\{(1,0),(7,0),(\infty_1,0)\}$,
&$\{(1,0),(8,0),(12,1)\}$,
&$\{(1,0),(10,0),(5,1)\}$,\\
&$\{(1,0),(11,0),(22,1)\}$,
&$\{(1,0),(12,0),(6,1)\}$,
&$\{(1,0),(14,0),(15,0)\}$,\\
&$\{(1,0),(16,0),(\infty_4,1)\}$,
&$\{(1,0),(19,0),(13,1)\}$,
&$\{(1,0),(20,0),(\infty_0,0)\}$,\\
&$\{(1,0),(21,0),(22,0)\}$,
&$\{(1,0),(23,0),(\infty_2,1)\}$,
&$\{(1,0),(\infty_2,0),(3,1)\}$,\\
&$\{(1,0),(\infty_3,0),(4,1)\}$,
&$\{(1,0),(\infty_4,0),(8,1)\}$,
&$\{(1,0),(10,1),(14,1)\}$,\\
&$\{(1,0),(11,1),(15,1)\}$,
&$\{(1,0),(16,1),(18,1)\}$,
&$\{(1,0),(19,1),(21,1)\}$,\\
&$\{(1,0),(20,1),(\infty_3,1)\}$,
&$\{(1,0),(23,1),(\infty_0,1)\}$,
&$\{(2,0),(3,0),(\infty_0,0)\}$,\\
&$\{(2,0),(4,0),(21,0)\}$,
&$\{(2,0),(5,0),(19,0)\}$,
&$\{(2,0),(6,0),(5,1)\}$,\\
&$\{(2,0),(7,0),(16,1)\}$,
&$\{(2,0),(8,0),(\infty_3,1)\}$,
&$\{(2,0),(9,0),(11,0)\}$,\\
&$\{(2,0),(12,0),(21,1)\}$,
&$\{(2,0),(13,0),(15,0)\}$,
&$\{(2,0),(14,0),(16,0)\}$,\\
&$\{(2,0),(17,0),(\infty_4,0)\}$,
&$\{(2,0),(20,0),(22,0)\}$,
&$\{(2,0),(23,0),(8,1)\}$,\\
&$\{(2,0),(\infty_1,0),(4,1)\}$,
&$\{(2,0),(\infty_2,0),(9,1)\}$,
&$\{(2,0),(\infty_3,0),(13,1)\}$,\\
&$\{(2,0),(3,1),(7,1)\}$,
&$\{(2,0),(11,1),(12,1)\}$,
&$\{(2,0),(14,1),(17,1)\}$,\\
&$\{(2,0),(15,1),(\infty_2,1)\}$,
&$\{(2,0),(19,1),(\infty_0,1)\}$,
&$\{(2,0),(20,1),(23,1)\}$,\\
&$\{(2,0),(22,1),(\infty_4,1)\}$,
&$\{(3,0),(4,0),(22,0)\}$,
&$\{(3,0),(5,0),(\infty_3,0)\}$,\\
&$\{(3,0),(6,0),(\infty_4,1)\}$,
&$\{(3,0),(8,0),(20,0)\}$,
&$\{(3,0),(9,0),(\infty_2,0)\}$,\\
&$\{(3,0),(10,0),(21,0)\}$,
&$\{(3,0),(12,0),(14,0)\}$,
&$\{(3,0),(13,0),(6,1)\}$,\\
&$\{(3,0),(15,0),(16,0)\}$,
&$\{(3,0),(17,0),(20,1)\}$,
&$\{(3,0),(18,0),(23,0)\}$,\\
&$\{(3,0),(\infty_1,0),(5,1)\}$,
&$\{(3,0),(\infty_4,0),(10,1)\}$,
&$\{(3,0),(4,1),(7,1)\}$,\\
&$\{(3,0),(8,1),(22,1)\}$,
&$\{(3,0),(9,1),(13,1)\}$,
&$\{(3,0),(14,1),(18,1)\}$,\\
&$\{(3,0),(15,1),(\infty_0,1)\}$,
&$\{(3,0),(16,1),(17,1)\}$,
&$\{(3,0),(21,1),(23,1)\}$,\\
&$\{(4,0),(5,0),(22,1)\}$,
&$\{(4,0),(6,0),(7,1)\}$,
&$\{(4,0),(8,0),(5,1)\}$,\\
&$\{(4,0),(9,0),(18,1)\}$,
&$\{(4,0),(10,0),(6,1)\}$,
&$\{(4,0),(11,0),(21,1)\}$,\\
&$\{(4,0),(14,0),(19,0)\}$,
&$\{(4,0),(15,0),(8,1)\}$,
&$\{(4,0),(16,0),(17,1)\}$,\\
&$\{(4,0),(17,0),(18,0)\}$,
&$\{(4,0),(\infty_0,0),(9,1)\}$,
&$\{(4,0),(\infty_2,0),(10,1)\}$,\\
&$\{(4,0),(\infty_3,0),(11,1)\}$,
&$\{(4,0),(\infty_4,0),(13,1)\}$,
&$\{(4,0),(14,1),(23,1)\}$,\\
&$\{(4,0),(15,1),(\infty_4,1)\}$,
&$\{(4,0),(16,1),(\infty_0,1)\}$,
&$\{(4,0),(19,1),(\infty_2,1)\}$,\\
&$\{(5,0),(7,0),(20,0)\}$,
&$\{(5,0),(8,0),(14,0)\}$,
&$\{(5,0),(9,0),(7,1)\}$,\\
&$\{(5,0),(10,0),(\infty_4,0)\}$,
&$\{(5,0),(11,0),(14,1)\}$,
&$\{(5,0),(15,0),(17,0)\}$,\\
&$\{(5,0),(16,0),(22,0)\}$,
&$\{(5,0),(23,0),(17,1)\}$,
&$\{(5,0),(\infty_0,0),(15,1)\}$,\\
&$\{(5,0),(\infty_1,0),(9,1)\}$,
&$\{(5,0),(\infty_2,0),(16,1)\}$,
&$\{(5,0),(11,1),(20,1)\}$,\\
&$\{(5,0),(12,1),(\infty_2,1)\}$,
&$\{(5,0),(18,1),(\infty_0,1)\}$,
&$\{(5,0),(19,1),(\infty_4,1)\}$,\\
&$\{(5,0),(23,1),(\infty_3,1)\}$,
&$\{(6,0),(7,0),(\infty_4,0)\}$,
&$\{(6,0),(8,0),(\infty_2,0)\}$,\\
&$\{(6,0),(9,0),(8,1)\}$,
&$\{(6,0),(10,0),(15,0)\}$,
&$\{(6,0),(11,0),(\infty_1,0)\}$,\\
&$\{(6,0),(12,0),(16,0)\}$,
&$\{(6,0),(13,0),(17,0)\}$,
&$\{(6,0),(18,0),(\infty_3,0)\}$,\\
&$\{(6,0),(19,0),(23,0)\}$,
&$\{(6,0),(20,0),(11,1)\}$,
&$\{(6,0),(21,0),(\infty_3,1)\}$,\\
&$\{(6,0),(\infty_0,0),(16,1)\}$,
&$\{(6,0),(9,1),(18,1)\}$,
&$\{(6,0),(15,1),(19,1)\}$,\\
&$\{(6,0),(17,1),(\infty_0,1)\}$,
&$\{(6,0),(20,1),(\infty_1,1)\}$,
&$\{(6,0),(21,1),(\infty_2,1)\}$,\\
&$\{(7,0),(8,0),(17,0)\}$,
&$\{(7,0),(9,0),(12,0)\}$,
&$\{(7,0),(10,0),(16,0)\}$,\\
&$\{(7,0),(11,0),(\infty_2,1)\}$,
&$\{(7,0),(14,0),(21,0)\}$,
&$\{(7,0),(18,0),(19,0)\}$,\\
&$\{(7,0),(22,0),(17,1)\}$,
&$\{(7,0),(\infty_0,0),(8,1)\}$,
&$\{(7,0),(\infty_2,0),(12,1)\}$,\\
&$\{(7,0),(\infty_3,0),(10,1)\}$,
&$\{(7,0),(11,1),(21,1)\}$,
&$\{(7,0),(13,1),(\infty_0,1)\}$,\\
&$\{(7,0),(14,1),(20,1)\}$,
&$\{(7,0),(18,1),(\infty_4,1)\}$,
&$\{(7,0),(19,1),(\infty_1,1)\}$,\\
&$\{(7,0),(22,1),(\infty_3,1)\}$,
&$\{(8,0),(9,0),(11,1)\}$,
&$\{(8,0),(10,0),(\infty_1,1)\}$,\\
&$\{(8,0),(11,0),(17,1)\}$,
&$\{(8,0),(12,0),(10,1)\}$,
&$\{(8,0),(13,0),(\infty_1,0)\}$,\\
&$\{(8,0),(15,0),(\infty_2,1)\}$,
&$\{(8,0),(18,0),(13,1)\}$,
&$\{(8,0),(19,0),(14,1)\}$,\\
&$\{(8,0),(21,0),(\infty_0,0)\}$,
&$\{(8,0),(23,0),(18,1)\}$,
&$\{(8,0),(\infty_3,0),(19,1)\}$,\\
&$\{(8,0),(\infty_4,0),(22,1)\}$,
&$\{(8,0),(20,1),(21,1)\}$,
&$\{(9,0),(14,0),(\infty_0,0)\}$,\\
&$\{(9,0),(15,0),(\infty_3,1)\}$,
&$\{(9,0),(16,0),(19,0)\}$,
&$\{(9,0),(20,0),(22,1)\}$,\\
&$\{(9,0),(21,0),(14,1)\}$,
&$\{(9,0),(22,0),(23,0)\}$,
&$\{(9,0),(\infty_1,0),(12,1)\}$,\\
&$\{(9,0),(\infty_3,0),(23,1)\}$,
&$\{(9,0),(\infty_4,0),(15,1)\}$,
&$\{(9,0),(10,1),(20,1)\}$,\\
&$\{(9,0),(13,1),(19,1)\}$,
&$\{(9,0),(16,1),(21,1)\}$,
&$\{(10,0),(12,0),(13,1)\}$,\\
&$\{(10,0),(13,0),(23,1)\}$,
&$\{(10,0),(17,0),(19,1)\}$,
&$\{(10,0),(19,0),(21,1)\}$,\\
&$\{(10,0),(22,0),(15,1)\}$,
&$\{(10,0),(23,0),(20,1)\}$,
&$\{(10,0),(\infty_0,0),(14,1)\}$,\\
&$\{(10,0),(\infty_1,0),(22,1)\}$,
&$\{(10,0),(\infty_2,0),(17,1)\}$,
&$\{(10,0),(\infty_3,0),(16,1)\}$,\\
&$\{(10,0),(11,1),(\infty_0,1)\}$,
&$\{(11,0),(13,0),(15,1)\}$,
&$\{(11,0),(14,0),(12,1)\}$,\\
&$\{(11,0),(16,0),(\infty_2,0)\}$,
&$\{(11,0),(17,0),(\infty_3,0)\}$,
&$\{(11,0),(18,0),(16,1)\}$,\\
&$\{(11,0),(22,0),(23,1)\}$,
&$\{(11,0),(23,0),(\infty_1,1)\}$,
&$\{(11,0),(\infty_4,0),(18,1)\}$,\\
&$\{(11,0),(13,1),(\infty_4,1)\}$,
&$\{(12,0),(13,0),(18,0)\}$,
&$\{(12,0),(15,0),(\infty_1,0)\}$,\\
&$\{(12,0),(17,0),(23,0)\}$,
&$\{(12,0),(19,0),(\infty_0,1)\}$,
&$\{(12,0),(21,0),(\infty_4,1)\}$,\\
&$\{(12,0),(22,0),(\infty_3,1)\}$,
&$\{(12,0),(\infty_0,0),(22,1)\}$,
&$\{(12,0),(\infty_3,0),(17,1)\}$,\\
&$\{(12,0),(\infty_4,0),(19,1)\}$,
&$\{(12,0),(15,1),(18,1)\}$,
&$\{(12,0),(16,1),(23,1)\}$,\\
&$\{(13,0),(16,0),(\infty_3,0)\}$,
&$\{(13,0),(20,0),(\infty_2,0)\}$,
&$\{(13,0),(22,0),(\infty_2,1)\}$,\\
&$\{(13,0),(23,0),(\infty_0,1)\}$,
&$\{(13,0),(14,1),(\infty_1,1)\}$,
&$\{(13,0),(16,1),(20,1)\}$,\\
&$\{(13,0),(17,1),(22,1)\}$,
&$\{(14,0),(\infty_2,0),(18,1)\}$,
&$\{(14,0),(\infty_3,0),(20,1)\}$,\\
&$\{(14,0),(\infty_4,0),(17,1)\}$,
&$\{(14,0),(15,1),(\infty_3,1)\}$,
&$\{(14,0),(16,1),(\infty_1,1)\}$,\\
&$\{(14,0),(23,1),(\infty_4,1)\}$,
&$\{(15,0),(20,0),(19,1)\}$,
&$\{(15,0),(21,0),(16,1)\}$,\\
&$\{(15,0),(17,1),(20,1)\}$,
&$\{(15,0),(18,1),(\infty_1,1)\}$,
&$\{(16,0),(\infty_4,0),(20,1)\}$,\\
&$\{(16,0),(19,1),(22,1)\}$,
&$\{(16,0),(23,1),(\infty_1,1)\}$,
&$\{(17,0),(19,0),(\infty_1,1)\}$,\\
&$\{(17,0),(21,0),(\infty_0,1)\}$,
&$\{(17,0),(\infty_1,0),(21,1)\}$,
&$\{(18,0),(20,0),(\infty_0,1)\}$,\\
&$\{(18,0),(22,0),(19,1)\}$,
&$\{(18,0),(\infty_2,0),(20,1)\}$,
&$\{(18,0),(21,1),(\infty_3,1)\}$,\\
&$\{(18,0),(22,1),(\infty_1,1)\}$,
&$\{(19,0),(23,1),(\infty_2,1)\}$,
&$\{(20,0),(21,1),(\infty_1,1)\}$,\\
&$\{(21,0),(\infty_4,0),(23,1)\}$,
&$\{(21,0),(22,1),(\infty_2,1)\}$.
\end{longtable}
\end{center}
For $(t,r)=(1,8)$, take $X=(I_{3}\bigcup\{\infty_0,\infty_1\})\times\mathbb{Z}_8$ with the group set $\mathcal{G}=\{I_{3}\times\{i,i+4\}:0\leq i\leq 3\}\bigcup\{\{\infty_0,\infty_1\}\times\mathbb{Z}_8\}$. Only $25$ base blocks are listed below.

\begin{center}
\begin{longtable}{lllll}
$(t,r)=(1,8)$:
&$\{(0,0),(\infty_0,0),(0,1)\}$,
&$\{(0,0),(\infty_1,0),(1,1)\}$,
&$\{(0,0),(2,1),(\infty_1,6)\}$,\\
&$\{(0,0),(\infty_0,1),(0,5)\}$,
&$\{(0,0),(\infty_1,1),(0,6)\}$,
&$\{(0,0),(1,2),(2,5)\}$,\\
&$\{(0,0),(2,2),(\infty_0,5)\}$,
&$\{(0,0),(\infty_0,2),(1,6)\}$,
&$\{(0,0),(\infty_1,2),(1,7)\}$,\\
&$\{(0,0),(1,3),(\infty_0,6)\}$,
&$\{(0,0),(2,3),(\infty_1,5)\}$,
&$\{(0,0),(\infty_0,3),(2,7)\}$,\\
&$\{(0,0),(\infty_1,4),(2,6)\}$,
&$\{(0,0),(1,5),(\infty_1,7)\}$,
&$\{(1,0),(\infty_0,0),(2,7)\}$,\\
&$\{(1,0),(\infty_1,0),(1,3)\}$,
&$\{(1,0),(1,1),(\infty_0,6)\}$,
&$\{(1,0),(2,1),(\infty_0,7)\}$,\\
&$\{(1,0),(\infty_0,1),(2,2)\}$,
&$\{(1,0),(\infty_1,1),(2,6)\}$,
&$\{(1,0),(1,2),(\infty_1,6)\}$,\\
&$\{(1,0),(\infty_0,2),(2,5)\}$,
&$\{(2,0),(\infty_0,0),(2,6)\}$,
&$\{(2,0),(\infty_1,0),(2,7)\}$,\\
&$\{(2,0),(2,3),(\infty_1,7)\}$.
\end{longtable}
\end{center}
For $(t,r)=(2,8)$, take $X=(I_{3}\bigcup\{\infty\})\times\mathbb{Z}_{16}$ with $\mathcal{G}=\{I_{3}\times\{i,i+8\}:0\leq i\leq 7\}\bigcup\{\{\infty\}\times\mathbb{Z}_{16}\}$. Only $37$ base blocks are listed below.

\begin{center}
\begin{longtable}{lllll}
$(t,r)=(2,8)$:
&$\{(0,0),(\infty,0),(0,1)\}$,
&$\{(0,0),(1,1),(2,11)\}$,
&$\{(0,0),(2,1),(2,10)\}$,\\
&$\{(0,0),(\infty,1),(0,10)\}$,
&$\{(0,0),(0,2),(0,13)\}$,
&$\{(0,0),(1,2),(\infty,8)\}$,\\
&$\{(0,0),(2,2),(1,14)\}$,
&$\{(0,0),(\infty,2),(1,7)\}$,
&$\{(0,0),(1,3),(0,9)\}$,\\
&$\{(0,0),(2,3),(2,6)\}$,
&$\{(0,0),(\infty,3),(1,15)\}$,
&$\{(0,0),(0,4),(\infty,10)\}$,\\
&$\{(0,0),(1,4),(2,7)\}$,
&$\{(0,0),(2,4),(\infty,13)\}$,
&$\{(0,0),(\infty,4),(1,12)\}$,\\
&$\{(0,0),(1,5),(1,9)\}$,
&$\{(0,0),(2,5),(\infty,12)\}$,
&$\{(0,0),(\infty,5),(1,11)\}$,\\
&$\{(0,0),(1,6),(2,15)\}$,
&$\{(0,0),(2,9),(\infty,11)\}$,
&$\{(0,0),(\infty,9),(2,14)\}$,\\
&$\{(0,0),(2,12),(2,13)\}$,
&$\{(0,0),(1,13),(\infty,14)\}$,
&$\{(1,0),(\infty,0),(1,11)\}$,\\
&$\{(1,0),(1,1),(2,6)\}$,
&$\{(1,0),(2,1),(\infty,14)\}$,
&$\{(1,0),(1,2),(2,13)\}$,\\
&$\{(1,0),(2,2),(\infty,3)\}$,
&$\{(1,0),(\infty,2),(2,12)\}$,
&$\{(1,0),(1,3),(\infty,15)\}$,\\
&$\{(1,0),(1,6),(\infty,13)\}$,
&$\{(1,0),(1,7),(2,14)\}$,
&$\{(1,0),(\infty,9),(2,15)\}$,\\
&$\{(2,0),(\infty,0),(2,12)\}$,
&$\{(2,0),(2,2),(\infty,14)\}$,
&$\{(2,0),(\infty,3),(2,11)\}$,\\
&$\{(2,0),(\infty,5),(2,6)\}$.
\end{longtable}
\end{center}
For $(t,r)\in\{(1,9),(2,9)\}$, take $X=(I_{4t}\bigcup\{\infty_0,\infty_1,\infty_2\})\times\mathbb{Z}_{6}$ with
$\mathcal{G}=\{\{i\}\times\mathbb{Z}_{6}:i\in(I_{4t}\bigcup\{\infty_0,\infty_1,\infty_2\})\}$. Only $4t(4t+5)$ base blocks are listed below.

\begin{center}
\begin{longtable}{lllll}
$(t,r)=(1,9)$:
&$\{(0,0),(1,0),(\infty_0,1)\}$,
&$\{(0,0),(2,0),(\infty_0,3)\}$,
&$\{(0,0),(3,0),(\infty_2,2)\}$,\\
&$\{(0,0),(\infty_0,0),(2,4)\}$,
&$\{(0,0),(\infty_1,0),(3,5)\}$,
&$\{(0,0),(\infty_2,0),(1,4)\}$,\\
&$\{(0,0),(1,1),(\infty_0,5)\}$,
&$\{(0,0),(2,1),(\infty_2,3)\}$,
&$\{(0,0),(3,1),(\infty_1,4)\}$,\\
&$\{(0,0),(\infty_1,1),(1,5)\}$,
&$\{(0,0),(\infty_2,1),(3,4)\}$,
&$\{(0,0),(1,2),(\infty_1,5)\}$,\\
&$\{(0,0),(2,2),(\infty_2,5)\}$,
&$\{(0,0),(3,2),(\infty_0,2)\}$,
&$\{(0,0),(\infty_1,2),(3,3)\}$,\\
&$\{(0,0),(1,3),(\infty_2,4)\}$,
&$\{(0,0),(2,3),(\infty_0,4)\}$,
&$\{(0,0),(\infty_1,3),(2,5)\}$,\\
&$\{(1,0),(2,0),(\infty_2,0)\}$,
&$\{(1,0),(3,0),(\infty_0,5)\}$,
&$\{(1,0),(\infty_0,0),(3,4)\}$,\\
&$\{(1,0),(\infty_1,0),(3,2)\}$,
&$\{(1,0),(2,1),(\infty_1,4)\}$,
&$\{(1,0),(3,1),(\infty_2,5)\}$,\\
&$\{(1,0),(\infty_1,1),(2,2)\}$,
&$\{(1,0),(\infty_0,2),(3,5)\}$,
&$\{(1,0),(2,3),(\infty_2,4)\}$,\\
&$\{(1,0),(3,3),(\infty_2,3)\}$,
&$\{(1,0),(\infty_0,3),(2,5)\}$,
&$\{(1,0),(2,4),(\infty_1,5)\}$,\\
&$\{(2,0),(3,0),(\infty_2,5)\}$,
&$\{(2,0),(\infty_0,0),(3,5)\}$,
&$\{(2,0),(\infty_1,0),(3,4)\}$,\\
&$\{(2,0),(3,1),(\infty_0,5)\}$,
&$\{(2,0),(3,2),(\infty_1,2)\}$,
&$\{(2,0),(3,3),(\infty_2,4)\}$.\\

$(t,r)=(2,9)$:
&$\{(0,0),(1,0),(2,0)\}$,
&$\{(0,0),(3,0),(\infty_0,3)\}$,
&$\{(0,0),(4,0),(\infty_1,4)\}$,\\
&$\{(0,0),(5,0),(\infty_1,5)\}$,
&$\{(0,0),(6,0),(1,3)\}$,
&$\{(0,0),(7,0),(\infty_0,2)\}$,\\
&$\{(0,0),(\infty_0,0),(5,1)\}$,
&$\{(0,0),(\infty_1,0),(5,5)\}$,
&$\{(0,0),(\infty_2,0),(1,5)\}$,\\
&$\{(0,0),(1,1),(2,3)\}$,
&$\{(0,0),(2,1),(1,4)\}$,
&$\{(0,0),(3,1),(7,2)\}$,\\
&$\{(0,0),(4,1),(\infty_1,2)\}$,
&$\{(0,0),(6,1),(5,4)\}$,
&$\{(0,0),(7,1),(\infty_0,5)\}$,\\
&$\{(0,0),(\infty_0,1),(6,5)\}$,
&$\{(0,0),(\infty_1,1),(3,3)\}$,
&$\{(0,0),(\infty_2,1),(3,2)\}$,\\
&$\{(0,0),(1,2),(7,5)\}$,
&$\{(0,0),(2,2),(4,2)\}$,
&$\{(0,0),(5,2),(6,2)\}$,\\
&$\{(0,0),(\infty_2,2),(7,4)\}$,
&$\{(0,0),(4,3),(5,3)\}$,
&$\{(0,0),(6,3),(7,3)\}$,\\
&$\{(0,0),(\infty_1,3),(2,4)\}$,
&$\{(0,0),(\infty_2,3),(4,4)\}$,
&$\{(0,0),(3,4),(\infty_2,5)\}$,\\
&$\{(0,0),(6,4),(\infty_0,4)\}$,
&$\{(0,0),(\infty_2,4),(2,5)\}$,
&$\{(0,0),(3,5),(4,5)\}$,\\
&$\{(1,0),(3,0),(2,4)\}$,
&$\{(1,0),(4,0),(5,2)\}$,
&$\{(1,0),(5,0),(2,5)\}$,\\
&$\{(1,0),(6,0),(\infty_1,0)\}$,
&$\{(1,0),(7,0),(\infty_0,0)\}$,
&$\{(1,0),(\infty_2,0),(5,1)\}$,\\
&$\{(1,0),(2,1),(3,5)\}$,
&$\{(1,0),(3,1),(4,5)\}$,
&$\{(1,0),(4,1),(3,2)\}$,\\
&$\{(1,0),(6,1),(\infty_1,3)\}$,
&$\{(1,0),(7,1),(5,5)\}$,
&$\{(1,0),(\infty_0,1),(4,2)\}$,\\
&$\{(1,0),(\infty_1,1),(3,4)\}$,
&$\{(1,0),(6,2),(\infty_2,2)\}$,
&$\{(1,0),(7,2),(\infty_1,4)\}$,\\
&$\{(1,0),(\infty_0,2),(4,4)\}$,
&$\{(1,0),(\infty_1,2),(6,5)\}$,
&$\{(1,0),(3,3),(\infty_1,5)\}$,\\
&$\{(1,0),(4,3),(\infty_0,3)\}$,
&$\{(1,0),(5,3),(\infty_0,5)\}$,
&$\{(1,0),(\infty_2,3),(6,4)\}$,\\
&$\{(1,0),(5,4),(\infty_2,4)\}$,
&$\{(1,0),(7,4),(\infty_2,5)\}$,
&$\{(1,0),(\infty_0,4),(7,5)\}$,\\
&$\{(2,0),(3,0),(6,2)\}$,
&$\{(2,0),(5,0),(\infty_2,2)\}$,
&$\{(2,0),(6,0),(4,5)\}$,\\
&$\{(2,0),(7,0),(\infty_1,0)\}$,
&$\{(2,0),(\infty_0,0),(4,3)\}$,
&$\{(2,0),(\infty_2,0),(4,4)\}$,\\
&$\{(2,0),(3,1),(5,3)\}$,
&$\{(2,0),(4,1),(6,4)\}$,
&$\{(2,0),(6,1),(7,2)\}$,\\
&$\{(2,0),(7,1),(\infty_0,2)\}$,
&$\{(2,0),(\infty_0,1),(3,3)\}$,
&$\{(2,0),(\infty_1,1),(4,2)\}$,\\
&$\{(2,0),(\infty_2,1),(5,4)\}$,
&$\{(2,0),(5,2),(\infty_1,2)\}$,
&$\{(2,0),(6,3),(\infty_0,4)\}$,\\
&$\{(2,0),(7,3),(\infty_2,3)\}$,
&$\{(2,0),(\infty_0,3),(5,5)\}$,
&$\{(2,0),(\infty_1,3),(7,4)\}$,\\
&$\{(2,0),(\infty_1,4),(6,5)\}$,
&$\{(2,0),(\infty_2,4),(7,5)\}$,
&$\{(2,0),(3,5),(\infty_0,5)\}$,\\
&$\{(3,0),(5,0),(4,2)\}$,
&$\{(3,0),(6,0),(\infty_2,4)\}$,
&$\{(3,0),(7,0),(\infty_2,3)\}$,\\
&$\{(3,0),(\infty_1,0),(6,5)\}$,
&$\{(3,0),(\infty_2,0),(5,5)\}$,
&$\{(3,0),(4,1),(7,5)\}$,\\
&$\{(3,0),(5,1),(\infty_0,1)\}$,
&$\{(3,0),(6,1),(7,4)\}$,
&$\{(3,0),(\infty_1,1),(5,3)\}$,\\
&$\{(3,0),(7,2),(\infty_0,5)\}$,
&$\{(3,0),(\infty_0,2),(6,3)\}$,
&$\{(3,0),(\infty_2,2),(5,4)\}$,\\
&$\{(3,0),(4,3),(\infty_1,5)\}$,
&$\{(3,0),(7,3),(6,4)\}$,
&$\{(4,0),(6,0),(\infty_2,3)\}$,\\
&$\{(4,0),(7,0),(5,5)\}$,
&$\{(4,0),(\infty_1,0),(7,3)\}$,
&$\{(4,0),(\infty_2,0),(6,5)\}$,\\
&$\{(4,0),(5,1),(\infty_0,2)\}$,
&$\{(4,0),(7,1),(5,3)\}$,
&$\{(4,0),(\infty_0,1),(6,4)\}$,\\
&$\{(4,0),(\infty_2,1),(7,5)\}$,
&$\{(4,0),(6,2),(\infty_2,4)\}$,
&$\{(4,0),(7,2),(\infty_1,3)\}$,\\
&$\{(5,0),(7,0),(6,2)\}$,
&$\{(5,0),(6,1),(7,3)\}$,
&$\{(5,0),(\infty_1,2),(6,4)\}$,\\
&$\{(5,0),(\infty_0,3),(6,5)\}$,
&$\{(5,0),(\infty_1,3),(7,5)\}$.
\end{longtable}
\end{center}

\section{Appendix: Base blocks in the proof of Lemma \ref{scigdd}}\label{2^12t16^1:t=1,2}

Let $t\in\{1,2\}$ and take $X=(\mathbb{Z}_{12t}\bigcup\{\infty_0,\infty_1,\ldots,\infty_{7}\})\times\mathbb{Z}_2$ with the group set $\mathcal{G}=\{\{i\}\times\mathbb{Z}_2:0\leq i\leq 12t-1\}\bigcup\{\{\infty_0,\infty_1,\ldots,\infty_{7}\}\times\mathbb{Z}_2\}$. All the $12t(4t+5)$ base blocks are obtained by developing the following $4t(4t+5)$ initial codewords by $(+4t\pmod{12t}, -)$.

\begin{center}
\begin{longtable}{lllll}
$t=1$:
&$\{(0,0),(1,0),(\infty_1,0)\}$,
&$\{(0,0),(2,0),(\infty_0,0)\}$,
&$\{(0,0),(3,0),(\infty_0,1)\}$,\\
&$\{(0,0),(4,0),(5,1)\}$,
&$\{(0,0),(5,0),(\infty_6,0)\}$,
&$\{(0,0),(6,0),(\infty_3,1)\}$,\\
&$\{(0,0),(7,0),(\infty_2,0)\}$,
&$\{(0,0),(9,0),(\infty_4,0)\}$,
&$\{(0,0),(10,0),(\infty_4,1)\}$,\\
&$\{(0,0),(11,0),(\infty_1,1)\}$,
&$\{(0,0),(\infty_3,0),(9,1)\}$,
&$\{(0,0),(\infty_5,0),(10,1)\}$,\\
&$\{(0,0),(\infty_7,0),(8,1)\}$,
&$\{(0,0),(2,1),(\infty_5,1)\}$,
&$\{(0,0),(3,1),(11,1)\}$,\\
&$\{(0,0),(6,1),(\infty_2,1)\}$,
&$\{(0,0),(7,1),(\infty_6,1)\}$,
&$\{(1,0),(2,0),(\infty_1,1)\}$,\\
&$\{(1,0),(3,0),(\infty_4,1)\}$,
&$\{(1,0),(5,0),(10,1)\}$,
&$\{(1,0),(6,0),(\infty_2,1)\}$,\\
&$\{(1,0),(7,0),(\infty_7,0)\}$,
&$\{(1,0),(10,0),(\infty_3,0)\}$,
&$\{(1,0),(11,0),(\infty_7,1)\}$,\\
&$\{(1,0),(\infty_0,0),(5,1)\}$,
&$\{(1,0),(\infty_2,0),(11,1)\}$,
&$\{(1,0),(\infty_5,0),(3,1)\}$,\\
&$\{(1,0),(2,1),(\infty_6,1)\}$,
&$\{(1,0),(7,1),(\infty_5,1)\}$,
&$\{(2,0),(3,0),(\infty_4,0)\}$,\\
&$\{(2,0),(6,0),(3,1)\}$,
&$\{(2,0),(7,0),(\infty_1,0)\}$,
&$\{(2,0),(11,0),(\infty_6,1)\}$,\\
&$\{(2,0),(\infty_7,0),(10,1)\}$,
&$\{(2,0),(7,1),(\infty_0,1)\}$,
&$\{(3,0),(\infty_3,0),(7,1)\}$.\\

$t=2$:
&$\{(0,0),(1,0),(2,0)\}$,
&$\{(0,0),(3,0),(9,1)\}$,
&$\{(0,0),(4,0),(3,1)\}$,\\
&$\{(0,0),(5,0),(\infty_6,1)\}$,
&$\{(0,0),(6,0),(7,0)\}$,
&$\{(0,0),(8,0),(1,1)\}$,\\
&$\{(0,0),(9,0),(12,1)\}$,
&$\{(0,0),(10,0),(22,0)\}$,
&$\{(0,0),(11,0),(14,1)\}$,\\
&$\{(0,0),(12,0),(18,0)\}$,
&$\{(0,0),(13,0),(17,0)\}$,
&$\{(0,0),(14,0),(19,0)\}$,\\
&$\{(0,0),(15,0),(18,1)\}$,
&$\{(0,0),(20,0),(21,1)\}$,
&$\{(0,0),(21,0),(\infty_3,1)\}$,\\
&$\{(0,0),(23,0),(10,1)\}$,
&$\{(0,0),(\infty_0,0),(11,1)\}$,
&$\{(0,0),(\infty_1,0),(19,1)\}$,\\
&$\{(0,0),(\infty_2,0),(6,1)\}$,
&$\{(0,0),(\infty_3,0),(4,1)\}$,
&$\{(0,0),(\infty_4,0),(5,1)\}$,\\
&$\{(0,0),(\infty_5,0),(7,1)\}$,
&$\{(0,0),(\infty_6,0),(15,1)\}$,
&$\{(0,0),(\infty_7,0),(8,1)\}$,\\
&$\{(0,0),(2,1),(\infty_5,1)\}$,
&$\{(0,0),(13,1),(\infty_1,1)\}$,
&$\{(0,0),(20,1),(\infty_4,1)\}$,\\
&$\{(0,0),(22,1),(\infty_0,1)\}$,
&$\{(0,0),(23,1),(\infty_2,1)\}$,
&$\{(1,0),(3,0),(17,0)\}$,\\
&$\{(1,0),(4,0),(5,0)\}$,
&$\{(1,0),(6,0),(20,0)\}$,
&$\{(1,0),(7,0),(2,1)\}$,\\
&$\{(1,0),(10,0),(13,0)\}$,
&$\{(1,0),(12,0),(\infty_6,1)\}$,
&$\{(1,0),(14,0),(\infty_3,1)\}$,\\
&$\{(1,0),(15,0),(18,0)\}$,
&$\{(1,0),(19,0),(22,0)\}$,
&$\{(1,0),(23,0),(\infty_7,0)\}$,\\
&$\{(1,0),(\infty_0,0),(5,1)\}$,
&$\{(1,0),(\infty_1,0),(6,1)\}$,
&$\{(1,0),(\infty_2,0),(3,1)\}$,\\
&$\{(1,0),(\infty_3,0),(18,1)\}$,
&$\{(1,0),(\infty_4,0),(7,1)\}$,
&$\{(1,0),(\infty_5,0),(9,1)\}$,\\
&$\{(1,0),(\infty_6,0),(10,1)\}$,
&$\{(1,0),(11,1),(12,1)\}$,
&$\{(1,0),(13,1),(\infty_4,1)\}$,\\
&$\{(1,0),(14,1),(\infty_2,1)\}$,
&$\{(1,0),(15,1),(20,1)\}$,
&$\{(1,0),(21,1),(\infty_7,1)\}$,\\
&$\{(1,0),(22,1),(\infty_1,1)\}$,
&$\{(1,0),(23,1),(\infty_0,1)\}$,
&$\{(2,0),(3,0),(5,1)\}$,\\
&$\{(2,0),(4,0),(6,0)\}$,
&$\{(2,0),(7,0),(\infty_7,1)\}$,
&$\{(2,0),(10,0),(11,1)\}$,\\
&$\{(2,0),(11,0),(18,1)\}$,
&$\{(2,0),(12,0),(\infty_2,0)\}$,
&$\{(2,0),(13,0),(14,1)\}$,\\
&$\{(2,0),(15,0),(\infty_3,0)\}$,
&$\{(2,0),(19,0),(\infty_4,1)\}$,
&$\{(2,0),(21,0),(\infty_1,1)\}$,\\
&$\{(2,0),(22,0),(\infty_4,0)\}$,
&$\{(2,0),(\infty_0,0),(4,1)\}$,
&$\{(2,0),(\infty_1,0),(12,1)\}$,\\
&$\{(2,0),(\infty_6,0),(22,1)\}$,
&$\{(2,0),(\infty_7,0),(6,1)\}$,
&$\{(2,0),(13,1),(\infty_5,1)\}$,\\
&$\{(2,0),(20,1),(\infty_0,1)\}$,
&$\{(2,0),(21,1),(\infty_2,1)\}$,
&$\{(3,0),(5,0),(7,0)\}$,\\
&$\{(3,0),(11,0),(15,1)\}$,
&$\{(3,0),(12,0),(14,1)\}$,
&$\{(3,0),(13,0),(\infty_7,1)\}$,\\
&$\{(3,0),(14,0),(\infty_6,0)\}$,
&$\{(3,0),(15,0),(13,1)\}$,
&$\{(3,0),(20,0),(\infty_7,0)\}$,\\
&$\{(3,0),(21,0),(\infty_0,0)\}$,
&$\{(3,0),(23,0),(\infty_5,0)\}$,
&$\{(3,0),(\infty_1,0),(23,1)\}$,\\
&$\{(3,0),(\infty_2,0),(12,1)\}$,
&$\{(3,0),(\infty_3,0),(11,1)\}$,
&$\{(3,0),(\infty_4,0),(22,1)\}$,\\
&$\{(3,0),(20,1),(\infty_5,1)\}$,
&$\{(3,0),(21,1),(\infty_6,1)\}$,
&$\{(4,0),(7,0),(\infty_1,0)\}$,\\
&$\{(4,0),(12,0),(15,1)\}$,
&$\{(4,0),(13,0),(\infty_5,1)\}$,
&$\{(4,0),(15,0),(\infty_6,0)\}$,\\
&$\{(4,0),(21,0),(\infty_3,0)\}$,
&$\{(4,0),(22,0),(12,1)\}$,
&$\{(4,0),(13,1),(21,1)\}$,\\
&$\{(4,0),(22,1),(\infty_7,1)\}$,
&$\{(4,0),(23,1),(\infty_4,1)\}$,
&$\{(5,0),(6,0),(22,0)\}$,\\
&$\{(5,0),(14,0),(21,1)\}$,
&$\{(5,0),(15,0),(\infty_2,1)\}$,
&$\{(5,0),(23,0),(14,1)\}$,\\
&$\{(5,0),(15,1),(23,1)\}$,
&$\{(6,0),(15,0),(7,1)\}$,
&$\{(6,0),(23,0),(\infty_0,1)\}$,\\
&$\{(6,0),(\infty_3,0),(23,1)\}$,
&$\{(6,0),(\infty_5,0),(14,1)\}$.
\end{longtable}
\end{center}

\section{Appendix: Codewords in the proof of Lemma \ref{[10:4]}}\label{initial codewords of [10:4]}

\begin{center}
\begin{longtable}{lll}
$(t,r)=(1,1)$:\\
$\{(0,0),(2,0),(4,0)\}$,
&$\{(0,0),(3,0),(5,0)\}$,
&$\{(0,0),(\infty_0,0),(3,n/2)\}$,\\
$\{(0,0),(1,n/2),(4,n/2)\}$,
&$\{(0,0),(2,n/2),(5,n/2)\}$,
&$\{(1,0),(2,0),(3,n/2)\}$,\\
$\{(1,0),(3,0),(4,n/2)\}$,
&$\{(1,0),(5,0),(2,n/2)\}$,
&$\{(1,0),(\infty_0,0),(5,n/2)\}$,\\
$\{(2,0),(\infty_0,0),(4,n/2)\}$,
&$\{(3,0),(4,0),(5,n/2)\}$.\\

$(t,r)=(1,2)$:\\
$\{(0,0),(2,0),(4,0)\}$,
&$\{(0,0),(3,0),(5,0)\}$,
&$\{(0,0),(\infty_0,0),(1,n/2)\}$,\\
$\{(0,0),(\infty_1,0),(3,n/2)\}$,
&$\{(0,0),(2,n/2),(5,n/2)\}$,
&$\{(0,0),(4,n/2),(\infty_0,n/2)\}$,\\
$\{(1,0),(2,0),(\infty_0,0)\}$,
&$\{(1,0),(3,0),(2,n/2)\}$,
&$\{(1,0),(4,0),(3,n/2)\}$,\\
$\{(1,0),(5,0),(4,n/2)\}$,
&$\{(1,0),(\infty_1,0),(5,n/2)\}$,
&$\{(2,0),(\infty_1,0),(4,n/2)\}$,\\
$\{(2,0),(5,n/2),(\infty_0,n/2)\}$,
&$\{(3,0),(4,0),(\infty_0,n/2)\}$,
&$\{(3,0),(\infty_0,0),(5,n/2)\}$.\\

$(t,r)=(1,4)$:\\
$\{(0,0),(2,0),(4,0)\}$,
&$\{(0,0),(3,0),(\infty_0,0)\}$,
&$\{(0,0),(5,0),(\infty_1,0)\}$,\\
$\{(0,0),(\infty_2,0),(1,n/2)\}$,
&$\{(0,0),(\infty_3,0),(3,n/2)\}$,
&$\{(0,0),(2,n/2),(\infty_0,n/2)\}$,\\
$\{(0,0),(4,n/2),(\infty_1,n/2)\}$,
&$\{(0,0),(5,n/2),(\infty_2,n/2)\}$,
&$\{(1,0),(2,0),(\infty_1,0)\}$,\\
$\{(1,0),(3,0),(5,0)\}$,
&$\{(1,0),(4,0),(\infty_0,n/2)\}$,
&$\{(1,0),(\infty_0,0),(2,n/2)\}$,\\
$\{(1,0),(\infty_2,0),(4,n/2)\}$,
&$\{(1,0),(\infty_3,0),(5,n/2)\}$,
&$\{(1,0),(3,n/2),(\infty_1,n/2)\}$,\\
$\{(2,0),(5,0),(\infty_1,n/2)\}$,
&$\{(2,0),(\infty_2,0),(5,n/2)\}$,
&$\{(2,0),(\infty_3,0),(4,n/2)\}$,\\
$\{(2,0),(3,n/2),(\infty_2,n/2)\}$,
&$\{(3,0),(4,0),(\infty_1,n/2)\}$,
&$\{(3,0),(4,n/2),(\infty_2,n/2)\}$,\\
$\{(3,0),(5,n/2),(\infty_0,n/2)\}$,
&$\{(4,0),(\infty_0,0),(5,n/2)\}$.\\
$(t,r)=(1,5)$:\\
$\{(0,0),(2,0),(\infty_0,0)\}$,
&$\{(0,0),(3,0),(\infty_1,0)\}$,
&$\{(0,0),(4,0),(\infty_2,0)\}$,\\
$\{(0,0),(5,0),(\infty_3,0)\}$,
&$\{(0,0),(\infty_4,0),(3,n/2)\}$,
&$\{(0,0),(1,n/2),(\infty_0,n/2)\}$,\\
$\{(0,0),(2,n/2),(\infty_1,n/2)\}$,
&$\{(0,0),(4,n/2),(\infty_3,n/2)\}$,
&$\{(0,0),(5,n/2),(\infty_2,n/2)\}$,\\
$\{(1,0),(2,0),(\infty_2,0)\}$,
&$\{(1,0),(3,0),(\infty_3,0)\}$,
&$\{(1,0),(4,0),(\infty_0,n/2)\}$,\\
$\{(1,0),(5,0),(\infty_1,n/2)\}$,
&$\{(1,0),(\infty_1,0),(4,n/2)\}$,
&$\{(1,0),(\infty_4,0),(5,n/2)\}$,\\
$\{(1,0),(2,n/2),(\infty_3,n/2)\}$,
&$\{(1,0),(3,n/2),(\infty_2,n/2)\}$,
&$\{(2,0),(4,0),(\infty_2,n/2)\}$,\\
$\{(2,0),(5,0),(\infty_3,n/2)\}$,
&$\{(2,0),(\infty_4,0),(4,n/2)\}$,
&$\{(2,0),(3,n/2),(\infty_0,n/2)\}$,\\
$\{(2,0),(5,n/2),(\infty_1,n/2)\}$,
&$\{(3,0),(4,0),(\infty_3,n/2)\}$,
&$\{(3,0),(5,0),(\infty_2,n/2)\}$,\\
$\{(3,0),(4,n/2),(\infty_1,n/2)\}$,
&$\{(3,0),(5,n/2),(\infty_0,n/2)\}$,
&$\{(4,0),(\infty_0,0),(5,n/2)\}$.\\
$(t,r)=(2,1)$:\\
$\{(0,0),(2,0),(4,0)\}$,
&$\{(0,0),(3,0),(5,0)\}$,
&$\{(0,0),(6,0),(8,0)\}$,\\
$\{(0,0),(7,0),(9,0)\}$,
&$\{(0,0),(10,0),(1,n/2)\}$,
&$\{(0,0),(11,0),(2,n/2)\}$,\\
$\{(0,0),(\infty_0,0),(6,n/2)\}$,
&$\{(0,0),(3,n/2),(4,n/2)\}$,
&$\{(0,0),(5,n/2),(7,n/2)\}$,\\
$\{(0,0),(8,n/2),(10,n/2)\}$,
&$\{(0,0),(9,n/2),(11,n/2)\}$,
&$\{(1,0),(2,0),(5,0)\}$,\\
$\{(1,0),(3,0),(6,0)\}$,
&$\{(1,0),(4,0),(7,0)\}$,
&$\{(1,0),(8,0),(11,0)\}$,\\
$\{(1,0),(9,0),(10,0)\}$,
&$\{(1,0),(\infty_0,0),(7,n/2)\}$,
&$\{(1,0),(2,n/2),(6,n/2)\}$,\\
$\{(1,0),(3,n/2),(8,n/2)\}$,
&$\{(1,0),(4,n/2),(9,n/2)\}$,
&$\{(1,0),(5,n/2),(11,n/2)\}$,\\
$\{(2,0),(7,0),(8,0)\}$,
&$\{(2,0),(9,0),(3,n/2)\}$,
&$\{(2,0),(10,0),(4,n/2)\}$,\\
$\{(2,0),(11,0),(5,n/2)\}$,
&$\{(2,0),(\infty_0,0),(8,n/2)\}$,
&$\{(2,0),(6,n/2),(9,n/2)\}$,\\
$\{(2,0),(7,n/2),(10,n/2)\}$,
&$\{(3,0),(7,0),(4,n/2)\}$,
&$\{(3,0),(9,0),(7,n/2)\}$,\\
$\{(3,0),(10,0),(11,n/2)\}$,
&$\{(3,0),(11,0),(6,n/2)\}$,
&$\{(3,0),(\infty_0,0),(10,n/2)\}$,\\
$\{(3,0),(5,n/2),(8,n/2)\}$,
&$\{(4,0),(6,0),(5,n/2)\}$,
&$\{(4,0),(8,0),(6,n/2)\}$,\\
$\{(4,0),(10,0),(8,n/2)\}$,
&$\{(4,0),(11,0),(9,n/2)\}$,
&$\{(4,0),(\infty_0,0),(11,n/2)\}$,\\
$\{(5,0),(6,0),(10,n/2)\}$,
&$\{(5,0),(9,0),(8,n/2)\}$,
&$\{(5,0),(10,0),(7,n/2)\}$,\\
$\{(5,0),(\infty_0,0),(9,n/2)\}$,
&$\{(6,0),(10,0),(9,n/2)\}$,
&$\{(6,0),(11,0),(7,n/2)\}$,\\
$\{(7,0),(11,0),(8,n/2)\}$.\\
$(t,r)=(2,2)$:\\
$\{(0,0),(2,0),(4,0)\}$,
&$\{(0,0),(3,0),(5,0)\}$,
&$\{(0,0),(6,0),(8,0)\}$,\\
$\{(0,0),(7,0),(9,0)\}$,
&$\{(0,0),(10,0),(\infty_0,0)\}$,
&$\{(0,0),(11,0),(1,n/2)\}$,\\
$\{(0,0),(\infty_1,0),(6,n/2)\}$,
&$\{(0,0),(2,n/2),(5,n/2)\}$,
&$\{(0,0),(3,n/2),(4,n/2)\}$,\\
$\{(0,0),(7,n/2),(8,n/2)\}$,
&$\{(0,0),(9,n/2),(10,n/2)\}$,
&$\{(0,0),(11,n/2),(\infty_0,n/2)\}$,\\
$\{(1,0),(2,0),(6,0)\}$,
&$\{(1,0),(3,0),(7,0)\}$,
&$\{(1,0),(4,0),(8,0)\}$,\\
$\{(1,0),(5,0),(9,0)\}$,
&$\{(1,0),(10,0),(2,n/2)\}$,
&$\{(1,0),(11,0),(3,n/2)\}$,\\
$\{(1,0),(\infty_0,0),(4,n/2)\}$,
&$\{(1,0),(\infty_1,0),(7,n/2)\}$,
&$\{(1,0),(5,n/2),(6,n/2)\}$,\\
$\{(1,0),(8,n/2),(10,n/2)\}$,
&$\{(1,0),(9,n/2),(\infty_0,n/2)\}$,
&$\{(2,0),(7,0),(10,0)\}$,\\
$\{(2,0),(8,0),(11,0)\}$,
&$\{(2,0),(9,0),(3,n/2)\}$,
&$\{(2,0),(\infty_0,0),(5,n/2)\}$,\\
$\{(2,0),(\infty_1,0),(8,n/2)\}$,
&$\{(2,0),(4,n/2),(6,n/2)\}$,
&$\{(2,0),(7,n/2),(\infty_0,n/2)\}$,\\
$\{(2,0),(9,n/2),(11,n/2)\}$,
&$\{(3,0),(6,0),(9,0)\}$,
&$\{(3,0),(8,0),(4,n/2)\}$,\\
$\{(3,0),(10,0),(6,n/2)\}$,
&$\{(3,0),(11,0),(7,n/2)\}$,
&$\{(3,0),(\infty_0,0),(8,n/2)\}$,\\
$\{(3,0),(\infty_1,0),(10,n/2)\}$,
&$\{(3,0),(5,n/2),(\infty_0,n/2)\}$,
&$\{(4,0),(7,0),(5,n/2)\}$,\\
$\{(4,0),(9,0),(6,n/2)\}$,
&$\{(4,0),(10,0),(7,n/2)\}$,
&$\{(4,0),(11,0),(10,n/2)\}$,\\
$\{(4,0),(\infty_0,0),(11,n/2)\}$,
&$\{(4,0),(\infty_1,0),(9,n/2)\}$,
&$\{(5,0),(7,0),(8,n/2)\}$,\\
$\{(5,0),(8,0),(10,n/2)\}$,
&$\{(5,0),(10,0),(9,n/2)\}$,
&$\{(5,0),(11,0),(6,n/2)\}$,\\
$\{(5,0),(\infty_1,0),(11,n/2)\}$,
&$\{(6,0),(10,0),(\infty_0,n/2)\}$,
&$\{(6,0),(11,0),(8,n/2)\}$,\\
$\{(6,0),(\infty_0,0),(7,n/2)\}$,
&$\{(7,0),(11,0),(9,n/2)\}$,
&$\{(8,0),(\infty_0,0),(9,n/2)\}$.\\
$(t,r)=(2,4)$:\\
$\{(0,0),(2,0),(4,0)\}$,
&$\{(0,0),(3,0),(5,0)\}$,
&$\{(0,0),(6,0),(8,0)\}$,\\
$\{(0,0),(7,0),(9,0)\}$,
&$\{(0,0),(10,0),(\infty_0,0)\}$,
&$\{(0,0),(11,0),(\infty_1,0)\}$,\\
$\{(0,0),(\infty_2,0),(1,n/2)\}$,
&$\{(0,0),(\infty_3,0),(6,n/2)\}$,
&$\{(0,0),(2,n/2),(5,n/2)\}$,\\
$\{(0,0),(3,n/2),(4,n/2)\}$,
&$\{(0,0),(7,n/2),(8,n/2)\}$,
&$\{(0,0),(9,n/2),(\infty_0,n/2)\}$,\\
$\{(0,0),(10,n/2),(\infty_1,n/2)\}$,
&$\{(0,0),(11,n/2),(\infty_2,n/2)\}$,
&$\{(1,0),(2,0),(6,0)\}$,\\
$\{(1,0),(3,0),(7,0)\}$,
&$\{(1,0),(4,0),(8,0)\}$,
&$\{(1,0),(5,0),(9,0)\}$,\\
$\{(1,0),(10,0),(\infty_2,0)\}$,
&$\{(1,0),(11,0),(\infty_0,0)\}$,
&$\{(1,0),(\infty_1,0),(2,n/2)\}$,\\
$\{(1,0),(\infty_3,0),(7,n/2)\}$,
&$\{(1,0),(3,n/2),(6,n/2)\}$,
&$\{(1,0),(4,n/2),(10,n/2)\}$,\\
$\{(1,0),(5,n/2),(11,n/2)\}$,
&$\{(1,0),(8,n/2),(\infty_0,n/2)\}$,
&$\{(1,0),(9,n/2),(\infty_1,n/2)\}$,\\
$\{(2,0),(7,0),(10,0)\}$,
&$\{(2,0),(8,0),(11,0)\}$,
&$\{(2,0),(9,0),(\infty_2,0)\}$,\\
$\{(2,0),(\infty_0,0),(3,n/2)\}$,
&$\{(2,0),(\infty_1,0),(4,n/2)\}$,
&$\{(2,0),(\infty_3,0),(8,n/2)\}$,\\
$\{(2,0),(5,n/2),(10,n/2)\}$,
&$\{(2,0),(6,n/2),(\infty_0,n/2)\}$,
&$\{(2,0),(7,n/2),(\infty_2,n/2)\}$,\\
$\{(2,0),(9,n/2),(11,n/2)\}$,
&$\{(3,0),(8,0),(10,0)\}$,
&$\{(3,0),(9,0),(4,n/2)\}$,\\
$\{(3,0),(11,0),(5,n/2)\}$,
&$\{(3,0),(\infty_0,0),(6,n/2)\}$,
&$\{(3,0),(\infty_1,0),(9,n/2)\}$,\\
$\{(3,0),(\infty_2,0),(10,n/2)\}$,
&$\{(3,0),(\infty_3,0),(11,n/2)\}$,
&$\{(3,0),(7,n/2),(\infty_1,n/2)\}$,\\
$\{(3,0),(8,n/2),(\infty_2,n/2)\}$,
&$\{(4,0),(6,0),(5,n/2)\}$,
&$\{(4,0),(7,0),(8,n/2)\}$,\\
$\{(4,0),(9,0),(\infty_0,n/2)\}$,
&$\{(4,0),(11,0),(\infty_2,n/2)\}$,
&$\{(4,0),(\infty_0,0),(11,n/2)\}$,\\
$\{(4,0),(\infty_1,0),(6,n/2)\}$,
&$\{(4,0),(\infty_2,0),(7,n/2)\}$,
&$\{(4,0),(\infty_3,0),(10,n/2)\}$,\\
$\{(5,0),(6,0),(\infty_2,n/2)\}$,
&$\{(5,0),(7,0),(\infty_1,n/2)\}$,
&$\{(5,0),(8,0),(\infty_0,n/2)\}$,\\
$\{(5,0),(\infty_0,0),(7,n/2)\}$,
&$\{(5,0),(\infty_1,0),(10,n/2)\}$,
&$\{(5,0),(\infty_2,0),(8,n/2)\}$,\\
$\{(5,0),(\infty_3,0),(9,n/2)\}$,
&$\{(6,0),(9,0),(10,n/2)\}$,
&$\{(6,0),(10,0),(11,n/2)\}$,\\
$\{(6,0),(11,0),(7,n/2)\}$,
&$\{(6,0),(\infty_1,0),(8,n/2)\}$,
&$\{(6,0),(\infty_2,0),(9,n/2)\}$,\\
$\{(7,0),(11,0),(9,n/2)\}$,
&$\{(7,0),(\infty_0,0),(10,n/2)\}$,
&$\{(8,0),(\infty_1,0),(11,n/2)\}$,\\
$\{(8,0),(9,n/2),(10,n/2)\}$.\\
$(t,r)=(2,5)$:\\
$\{(0,0),(2,0),(4,0)\}$,
&$\{(0,0),(3,0),(7,0)\}$,
&$\{(0,0),(5,0),(6,0)\}$,\\
$\{(0,0),(8,0),(10,0)\}$,
&$\{(0,0),(9,0),(\infty_2,0)\}$,
&$\{(0,0),(11,0),(10,n/2)\}$,\\
$\{(0,0),(\infty_0,0),(1,n/2)\}$,
&$\{(0,0),(\infty_1,0),(3,n/2)\}$,
&$\{(0,0),(\infty_3,0),(2,n/2)\}$,\\
$\{(0,0),(\infty_4,0),(6,n/2)\}$,
&$\{(0,0),(4,n/2),(7,n/2)\}$,
&$\{(0,0),(5,n/2),(\infty_0,n/2)\}$,\\
$\{(0,0),(8,n/2),(\infty_2,n/2)\}$,
&$\{(0,0),(9,n/2),(\infty_1,n/2)\}$,
&$\{(0,0),(11,n/2),(\infty_3,n/2)\}$,\\
$\{(1,0),(2,0),(\infty_3,0)\}$,
&$\{(1,0),(3,0),(\infty_1,0)\}$,
&$\{(1,0),(4,0),(5,n/2)\}$,\\
$\{(1,0),(5,0),(\infty_2,n/2)\}$,
&$\{(1,0),(6,0),(2,n/2)\}$,
&$\{(1,0),(7,0),(8,0)\}$,\\
$\{(1,0),(9,0),(10,0)\}$,
&$\{(1,0),(11,0),(\infty_0,0)\}$,
&$\{(1,0),(\infty_2,0),(3,n/2)\}$,\\
$\{(1,0),(\infty_4,0),(7,n/2)\}$,
&$\{(1,0),(4,n/2),(6,n/2)\}$,
&$\{(1,0),(8,n/2),(11,n/2)\}$,\\
$\{(1,0),(9,n/2),(\infty_3,n/2)\}$,
&$\{(1,0),(10,n/2),(\infty_1,n/2)\}$,
&$\{(2,0),(5,0),(11,n/2)\}$,\\
$\{(2,0),(6,0),(8,0)\}$,
&$\{(2,0),(7,0),(\infty_1,0)\}$,
&$\{(2,0),(9,0),(11,0)\}$,\\
$\{(2,0),(10,0),(5,n/2)\}$,
&$\{(2,0),(\infty_0,0),(3,n/2)\}$,
&$\{(2,0),(\infty_2,0),(4,n/2)\}$,\\
$\{(2,0),(\infty_4,0),(9,n/2)\}$,
&$\{(2,0),(7,n/2),(\infty_0,n/2)\}$,
&$\{(2,0),(8,n/2),(\infty_1,n/2)\}$,\\
$\{(2,0),(10,n/2),(\infty_2,n/2)\}$,
&$\{(3,0),(4,0),(8,0)\}$,
&$\{(3,0),(5,0),(9,0)\}$,\\
$\{(3,0),(6,0),(\infty_0,0)\}$,
&$\{(3,0),(10,0),(\infty_3,n/2)\}$,
&$\{(3,0),(11,0),(9,n/2)\}$,\\
$\{(3,0),(\infty_2,0),(6,n/2)\}$,
&$\{(3,0),(\infty_3,0),(4,n/2)\}$,
&$\{(3,0),(\infty_4,0),(10,n/2)\}$,\\
$\{(3,0),(5,n/2),(8,n/2)\}$,
&$\{(3,0),(7,n/2),(11,n/2)\}$,
&$\{(4,0),(9,0),(\infty_0,n/2)\}$,\\
$\{(4,0),(10,0),(7,n/2)\}$,
&$\{(4,0),(11,0),(\infty_2,0)\}$,
&$\{(4,0),(\infty_0,0),(10,n/2)\}$,\\
$\{(4,0),(\infty_1,0),(8,n/2)\}$,
&$\{(4,0),(\infty_3,0),(9,n/2)\}$,
&$\{(4,0),(\infty_4,0),(11,n/2)\}$,\\
$\{(4,0),(6,n/2),(\infty_1,n/2)\}$,
&$\{(5,0),(7,0),(\infty_0,n/2)\}$,
&$\{(5,0),(10,0),(\infty_1,n/2)\}$,\\
$\{(5,0),(11,0),(\infty_1,0)\}$,
&$\{(5,0),(\infty_2,0),(9,n/2)\}$,
&$\{(5,0),(\infty_3,0),(7,n/2)\}$,\\
$\{(5,0),(\infty_4,0),(8,n/2)\}$,
&$\{(5,0),(6,n/2),(\infty_3,n/2)\}$,
&$\{(6,0),(9,0),(7,n/2)\}$,\\
$\{(6,0),(10,0),(9,n/2)\}$,
&$\{(6,0),(11,0),(\infty_1,n/2)\}$,
&$\{(6,0),(\infty_2,0),(11,n/2)\}$,\\
$\{(6,0),(8,n/2),(\infty_3,n/2)\}$,
&$\{(6,0),(10,n/2),(\infty_0,n/2)\}$,
&$\{(7,0),(9,0),(\infty_1,n/2)\}$,\\
$\{(7,0),(10,0),(\infty_2,n/2)\}$,
&$\{(7,0),(\infty_2,0),(8,n/2)\}$,
&$\{(7,0),(\infty_3,0),(11,n/2)\}$,\\
$\{(8,0),(\infty_0,0),(11,n/2)\}$,
&$\{(8,0),(9,n/2),(\infty_0,n/2)\}$,
&$\{(8,0),(10,n/2),(\infty_3,n/2)\}$.
\end{longtable}
\end{center}

\section{Appendix: Codewords in the proof of Lemma \ref{[m:r],n=4}}\label{initial codewords of [m:r]}

The required $2$-D $([6t+r:r]\times4,3,1)$-OOCs are defined on $(\mathbb{Z}_{6t}\cup\{\infty_0,\infty_1,\ldots,\infty_{r-1}\})\times \mathbb{Z}_4$ with the hole $\{\infty_0,\infty_1,\ldots,\infty_{r-1}\}\times \mathbb{Z}_4$. For each given $t$ and $r$, we only list initial codewords. All the $24t^2+8tr-2t$ codewords can be obtained by developing these initial codewords by $(+2t\pmod{6t},-)$, where $\infty_i+2t=\infty_i$ for $i\in I_r$. Each of the initial codewords marked with a star only generates one codeword.

\begin{center}
\begin{longtable}{llll}
$(t,r)=(1,1)$: & $\{(0,0),(1,0),(2,0)\}$,&$\{(0,0),(3,0),(2,2)\}$,&$\{(0,0),(\infty_0,0),(1,3)\}$,\\
&$\{(0,0),(0,1),(\infty_0,2)\}$,&$\{(0,0),(1,1),(2,3)\}$,&$\{(0,0),(2,1),(5,3)\}$,\\
&$\{(0,0),(3,1),(\infty_0,3)\}$,&$\{(0,0),(5,1),(3,3)\}$,&$\{(1,0),(3,0),(1,3)\}$,\\
&$\{(1,0),(\infty_0,0),(5,1)\}$.\\

$(t,r)=(1,2)$: & $\{(0,0),(1,0),(\infty_0,0)\}$, &$\{(0,0),(3,0),(1,1)\}$, &$\{(0,0),(5,0),(\infty_1,0)\}$,\\
&$\{(0,0),(0,1),(\infty_0,3)\}$, &$\{(0,0),(2,1),(1,2)\}$, &$\{(0,0),(3,1),(2,2)\}$,\\
&$\{(0,0),(4,1),(\infty_1,2)\}$, &$\{(0,0),(\infty_0,1),(3,3)\}$, &$\{(0,0),(3,2),(5,3)\}$,\\
&$\{(0,0),(5,2),(\infty_1,3)\}$, &$\{(1,0),(1,1),(\infty_1,3)\}$, &$\{(1,0),(\infty_0,1),(3,2)\}$,\\
&$\{(0,0),(2,0),(4,0)\}^*$, &$\{(1,0),(3,0),(5,0)\}^*$.\\

$(t,r)=(1,3)$: &$\{(0,0),(1,0),(3,0)\}$,
&$\{(0,0),(5,0),(3,1)\}$,
&$\{(0,0),(\infty_0,0),(1,1)\}$,\\
&$\{(0,0),(\infty_1,0),(3,3)\}$,
&$\{(0,0),(\infty_2,0),(5,2)\}$,
&$\{(0,0),(0,1),(1,3)\}$,\\
&$\{(0,0),(2,1),(\infty_0,3)\}$,
&$\{(0,0),(4,1),(\infty_1,3)\}$,
&$\{(0,0),(5,1),(\infty_2,2)\}$,\\
&$\{(0,0),(\infty_0,1),(5,3)\}$,
&$\{(0,0),(\infty_1,1),(3,2)\}$,
&$\{(0,0),(\infty_2,1),(2,2)\}$,\\
&$\{(1,0),(\infty_0,0),(5,3)\}$,
&$\{(1,0),(\infty_1,0),(3,2)\}$,
&$\{(1,0),(\infty_2,0),(1,1)\}$,\\
&$\{(0,0),(2,0),(4,0)\}^*$.\\

$(t,r)=(1,4)$:
&$\{(0,0),(1,0),(2,0)\}$,
&$\{(0,0),(3,0),(\infty_0,0)\}$,
&$\{(0,0),(\infty_1,0),(4,3)\}$,\\
&$\{(0,0),(\infty_2,0),(3,2)\}$,
&$\{(0,0),(\infty_3,0),(1,2)\}$,
&$\{(0,0),(0,1),(\infty_2,3)\}$,\\
&$\{(0,0),(1,1),(\infty_3,1)\}$,
&$\{(0,0),(3,1),(\infty_1,2)\}$,
&$\{(0,0),(4,1),(\infty_3,3)\}$,\\
&$\{(0,0),(5,1),(\infty_0,2)\}$,
&$\{(0,0),(\infty_0,1),(4,2)\}$,
&$\{(0,0),(\infty_2,1),(5,2)\}$,\\
&$\{(0,0),(1,3),(5,3)\}$,
&$\{(0,0),(3,3),(\infty_1,3)\}$,
&$\{(1,0),(\infty_2,0),(3,3)\}$,\\
&$\{(1,0),(1,1),(\infty_1,3)\}$,
&$\{(1,0),(3,1),(\infty_0,3)\}$,
&$\{(1,0),(\infty_3,1),(5,2)\}$,\\

$(t,r)=(1,5)$: & $\{(0,0),(1,0),(\infty_1,0)\}$,&$\{(0,0),(3,0),(\infty_0,3)\}$,&$\{(0,0),(5,0),(\infty_2,1)\}$,\\
&$\{(0,0),(\infty_0,0),(3,3)\}$,&$\{(0,0),(\infty_2,0),(5,1)\}$,&$\{(0,0),(\infty_3,0),(4,3)\}$,\\
&$\{(0,0),(\infty_4,0),(2,2)\}$,&$\{(0,0),(0,1),(\infty_1,3)\}$,&$\{(0,0),(1,1),(\infty_4,1)\}$,\\
&$\{(0,0),(3,1),(\infty_2,3)\}$,&$\{(0,0),(4,1),(\infty_3,3)\}$,&$\{(0,0),(\infty_0,1),(1,3)\}$,\\
&$\{(0,0),(\infty_1,1),(5,3)\}$,&$\{(0,0),(1,2),(\infty_4,3)\}$,&$\{(0,0),(3,2),(\infty_0,2)\}$,\\
&$\{(0,0),(5,2),(\infty_2,2)\}$,&$\{(1,0),(\infty_3,0),(1,3)\}$,&$\{(1,0),(3,1),(\infty_3,3)\}$,\\
&$\{(1,0),(5,1),(\infty_4,3)\}$,&$\{(1,0),(\infty_1,1),(5,2)\}$,&$\{(0,0),(2,0),(4,0)\}^*$,\\
&$\{(1,0),(3,0),(5,0)\}^*$.\\

$(t,r)=(2,1)$:
&$\{(0,0),(1,0),(2,0)\}$,
&$\{(0,0),(3,0),(5,3)\}$,
&$\{(0,0),(4,0),(11,0)\}$,\\
&$\{(0,0),(5,0),(8,3)\}$,
&$\{(0,0),(6,0),(\infty_0,2)\}$,
&$\{(0,0),(9,0),(3,3)\}$,\\
&$\{(0,0),(10,0),(4,2)\}$,
&$\{(0,0),(\infty_0,0),(11,3)\}$,
&$\{(0,0),(0,1),(7,3)\}$,\\
&$\{(0,0),(1,1),(2,3)\}$,
&$\{(0,0),(2,1),(6,3)\}$,
&$\{(0,0),(3,1),(9,3)\}$,\\
&$\{(0,0),(5,1),(9,2)\}$,
&$\{(0,0),(6,1),(11,2)\}$,
&$\{(0,0),(7,1),(10,2)\}$,\\
&$\{(0,0),(8,1),(10,3)\}$,
&$\{(0,0),(10,1),(1,3)\}$,
&$\{(0,0),(11,1),(\infty_0,3)\}$,\\
&$\{(0,0),(\infty_0,1),(1,2)\}$,
&$\{(0,0),(3,2),(5,2)\}$,
&$\{(1,0),(3,0),(5,2)\}$,\\
&$\{(1,0),(5,0),(10,0)\}$,
&$\{(1,0),(7,0),(1,3)\}$,
&$\{(1,0),(\infty_0,0),(6,3)\}$,\\
&$\{(1,0),(2,1),(3,3)\}$,
&$\{(1,0),(3,1),(10,3)\}$,
&$\{(1,0),(6,1),(6,2)\}$,\\
&$\{(1,0),(9,1),(11,3)\}$,
&$\{(1,0),(10,1),(\infty_0,1)\}$,
&$\{(1,0),(\infty_0,2),(2,3)\}$,\\
&$\{(2,0),(3,0),(6,3)\}$,
&$\{(2,0),(6,0),(7,3)\}$,
&$\{(2,0),(7,0),(11,2)\}$,\\
&$\{(2,0),(11,0),(10,3)\}$,
&$\{(3,0),(7,0),(3,3)\}$,
&$\{(3,0),(\infty_0,0),(11,1)\}$.\\

$(t,r)=(2,2)$:
&$\{(0,0),(1,0),(2,0)\}$,
&$\{(0,0),(3,0),(2,3)\}$,
&$\{(0,0),(5,0),(6,3)\}$,\\
&$\{(0,0),(6,0),(\infty_0,0)\}$,
&$\{(0,0),(7,0),(5,1)\}$,
&$\{(0,0),(9,0),(\infty_0,1)\}$,\\
&$\{(0,0),(10,0),(\infty_1,3)\}$,
&$\{(0,0),(11,0),(11,1)\}$,
&$\{(0,0),(\infty_1,0),(1,3)\}$,\\
&$\{(0,0),(0,1),(7,3)\}$,
&$\{(0,0),(1,1),(5,3)\}$,
&$\{(0,0),(2,1),(3,3)\}$,\\
&$\{(0,0),(3,1),(4,2)\}$,
&$\{(0,0),(4,1),(1,2)\}$,
&$\{(0,0),(6,1),(6,2)\}$,\\
&$\{(0,0),(7,1),(9,3)\}$,
&$\{(0,0),(8,1),(\infty_1,2)\}$,
&$\{(0,0),(10,1),(3,2)\}$,\\
&$\{(0,0),(2,2),(\infty_0,3)\}$,
&$\{(0,0),(5,2),(\infty_0,2)\}$,
&$\{(0,0),(9,2),(10,3)\}$,\\
&$\{(0,0),(10,2),(11,2)\}$,
&$\{(1,0),(3,0),(5,0)\}$,
&$\{(1,0),(6,0),(9,3)\}$,\\
&$\{(1,0),(7,0),(6,1)\}$,
&$\{(1,0),(10,0),(6,2)\}$,
&$\{(1,0),(\infty_1,0),(10,3)\}$,\\
&$\{(1,0),(1,1),(7,3)\}$,
&$\{(1,0),(3,1),(\infty_1,3)\}$,
&$\{(1,0),(7,1),(2,2)\}$,\\
&$\{(1,0),(9,1),(6,3)\}$,
&$\{(1,0),(11,1),(\infty_0,3)\}$,
&$\{(1,0),(3,2),(\infty_0,2)\}$,\\
&$\{(1,0),(\infty_1,2),(11,3)\}$,
&$\{(2,0),(6,0),(10,3)\}$,
&$\{(2,0),(7,0),(\infty_1,0)\}$,\\
&$\{(2,0),(11,0),(7,2)\}$,
&$\{(2,0),(11,1),(\infty_1,2)\}$,
&$\{(2,0),(11,2),(\infty_0,3)\}$,\\
&$\{(2,0),(\infty_0,2),(11,3)\}$,
&$\{(3,0),(7,0),(11,3)\}$,
&$\{(0,0),(4,0),(8,0)\}^*$.\\

$(t,r)=(2,3)$: &$\{(0,0),(1,0),(2,0)\}$,
&$\{(0,0),(3,0),(\infty_2,2)\}$,
&$\{(0,0),(5,0),(7,3)\}$,\\
&$\{(0,0),(6,0),(10,2)\}$,
&$\{(0,0),(7,0),(2,3)\}$,
&$\{(0,0),(9,0),(8,2)\}$,\\
&$\{(0,0),(10,0),(\infty_2,0)\}$,
&$\{(0,0),(11,0),(\infty_2,1)\}$,
&$\{(0,0),(\infty_0,0),(2,2)\}$,\\
&$\{(0,0),(\infty_1,0),(6,2)\}$,
&$\{(0,0),(0,1),(3,2)\}$,
&$\{(0,0),(1,1),(6,1)\}$,\\
&$\{(0,0),(2,1),(5,1)\}$,
&$\{(0,0),(4,1),(1,3)\}$,
&$\{(0,0),(7,1),(9,1)\}$,\\
&$\{(0,0),(8,1),(\infty_0,2)\}$,
&$\{(0,0),(10,1),(\infty_2,3)\}$,
&$\{(0,0),(11,1),(\infty_0,3)\}$,\\
&$\{(0,0),(\infty_1,1),(11,2)\}$,
&$\{(0,0),(5,2),(3,3)\}$,
&$\{(0,0),(7,2),(10,3)\}$,\\
&$\{(0,0),(\infty_1,2),(9,3)\}$,
&$\{(0,0),(5,3),(\infty_1,3)\}$,
&$\{(0,0),(6,3),(11,3)\}$,\\
&$\{(1,0),(3,0),(7,0)\}$,
&$\{(1,0),(\infty_0,0),(10,3)\}$,
&$\{(1,0),(\infty_2,0),(10,1)\}$,\\
&$\{(1,0),(1,1),(\infty_2,2)\}$,
&$\{(1,0),(2,1),(\infty_1,2)\}$,
&$\{(1,0),(3,1),(5,3)\}$,\\
&$\{(1,0),(5,1),(2,3)\}$,
&$\{(1,0),(6,1),(7,2)\}$,
&$\{(1,0),(7,1),(11,3)\}$,\\
&$\{(1,0),(\infty_0,1),(9,2)\}$,
&$\{(1,0),(\infty_1,1),(7,3)\}$,
&$\{(1,0),(2,2),(6,3)\}$,\\
&$\{(1,0),(3,2),(\infty_0,2)\}$,
&$\{(1,0),(6,2),(\infty_2,3)\}$,
&$\{(2,0),(3,0),(\infty_0,3)\}$,\\
&$\{(2,0),(6,0),(11,2)\}$,
&$\{(2,0),(11,0),(2,3)\}$,
&$\{(2,0),(\infty_0,0),(7,3)\}$,\\
&$\{(2,0),(\infty_1,0),(10,1)\}$,
&$\{(2,0),(3,2),(3,3)\}$,
&$\{(3,0),(\infty_1,0),(7,3)\}$,\\
&$\{(3,0),(\infty_2,0),(7,1)\}$,
&$\{(0,0),(4,0),(8,0)\}^*$,
&$\{(1,0),(5,0),(9,0)\}^*$.\\

$(t,r)=(2,4)$:
&$\{(0,0),(1,0),(2,0)\}$,
&$\{(0,0),(3,0),(\infty_0,2)\}$,
&$\{(0,0),(4,0),(9,2)\}$,\\
&$\{(0,0),(5,0),(9,0)\}$,
&$\{(0,0),(6,0),(7,2)\}$,
&$\{(0,0),(7,0),(3,3)\}$,\\
&$\{(0,0),(10,0),(1,2)\}$,
&$\{(0,0),(11,0),(9,3)\}$,
&$\{(0,0),(\infty_0,0),(2,2)\}$,\\
&$\{(0,0),(\infty_1,0),(3,2)\}$,
&$\{(0,0),(\infty_2,0),(10,1)\}$,
&$\{(0,0),(\infty_3,0),(8,3)\}$,\\
&$\{(0,0),(0,1),(\infty_1,3)\}$,
&$\{(0,0),(1,1),(7,3)\}$,
&$\{(0,0),(2,1),(\infty_3,2)\}$,\\
&$\{(0,0),(3,1),(11,1)\}$,
&$\{(0,0),(5,1),(6,2)\}$,
&$\{(0,0),(6,1),(2,3)\}$,\\
&$\{(0,0),(7,1),(\infty_3,3)\}$,
&$\{(0,0),(8,1),(\infty_2,3)\}$,
&$\{(0,0),(9,1),(8,2)\}$,\\
&$\{(0,0),(\infty_0,1),(5,3)\}$,
&$\{(0,0),(\infty_1,1),(10,3)\}$,
&$\{(0,0),(\infty_2,1),(6,3)\}$,\\
&$\{(0,0),(10,2),(11,3)\}$,
&$\{(0,0),(11,2),(\infty_0,3)\}$,
&$\{(1,0),(3,0),(9,1)\}$,\\
&$\{(1,0),(6,0),(11,2)\}$,
&$\{(1,0),(7,0),(\infty_1,1)\}$,
&$\{(1,0),(10,0),(\infty_0,3)\}$,\\
&$\{(1,0),(11,0),(\infty_1,0)\}$,
&$\{(1,0),(\infty_0,0),(6,3)\}$,
&$\{(1,0),(\infty_2,0),(3,3)\}$,\\
&$\{(1,0),(\infty_3,0),(1,3)\}$,
&$\{(1,0),(5,1),(\infty_3,3)\}$,
&$\{(1,0),(6,1),(\infty_2,2)\}$,\\
&$\{(1,0),(7,1),(9,2)\}$,
&$\{(1,0),(10,1),(\infty_0,1)\}$,
&$\{(1,0),(11,1),(\infty_2,1)\}$,\\
&$\{(1,0),(2,2),(\infty_1,2)\}$,
&$\{(1,0),(3,2),(10,3)\}$,
&$\{(1,0),(6,2),(\infty_1,3)\}$,\\
&$\{(1,0),(2,3),(\infty_2,3)\}$,
&$\{(2,0),(3,0),(\infty_1,3)\}$,
&$\{(2,0),(6,0),(10,3)\}$,\\
&$\{(2,0),(7,0),(\infty_3,3)\}$,
&$\{(2,0),(11,0),(2,3)\}$,
&$\{(2,0),(\infty_3,0),(3,3)\}$,\\
&$\{(2,0),(7,1),(11,3)\}$,
&$\{(2,0),(11,2),(\infty_3,2)\}$,
&$\{(3,0),(\infty_0,0),(3,1)\}$,\\
&$\{(3,0),(11,1),(\infty_2,3)\}$.\\

$(t,r)=(2,5)$: &$\{(0,0),(1,0),(2,0)\}$,
&$\{(0,0),(3,0),(10,2)\}$,
&$\{(0,0),(5,0),(11,2)\}$,\\
&$\{(0,0),(6,0),(\infty_4,0)\}$,
&$\{(0,0),(7,0),(5,3)\}$,
&$\{(0,0),(9,0),(8,1)\}$,\\
&$\{(0,0),(10,0),(\infty_1,3)\}$,
&$\{(0,0),(11,0),(10,3)\}$,
&$\{(0,0),(\infty_0,0),(6,3)\}$,\\
&$\{(0,0),(\infty_1,0),(4,2)\}$,
&$\{(0,0),(\infty_2,0),(0,3)\}$,
&$\{(0,0),(\infty_3,0),(2,3)\}$,\\
&$\{(0,0),(1,1),(\infty_4,1)\}$,
&$\{(0,0),(2,1),(5,1)\}$,
&$\{(0,0),(3,1),(\infty_2,3)\}$,\\
&$\{(0,0),(4,1),(\infty_3,2)\}$,
&$\{(0,0),(6,1),(\infty_4,2)\}$,
&$\{(0,0),(7,1),(7,2)\}$,\\
&$\{(0,0),(9,1),(\infty_0,2)\}$,
&$\{(0,0),(10,1),(3,2)\}$,
&$\{(0,0),(11,1),(6,2)\}$,\\
&$\{(0,0),(\infty_0,1),(11,3)\}$,
&$\{(0,0),(\infty_1,1),(7,3)\}$,
&$\{(0,0),(1,2),(\infty_2,2)\}$,\\
&$\{(0,0),(2,2),(9,3)\}$,
&$\{(0,0),(5,2),(\infty_4,3)\}$,
&$\{(0,0),(9,2),(\infty_3,3)\}$,\\
&$\{(0,0),(3,3),(\infty_0,3)\}$,
&$\{(1,0),(3,0),(7,1)\}$,
&$\{(1,0),(5,0),(6,2)\}$,\\
&$\{(1,0),(6,0),(7,3)\}$,
&$\{(1,0),(7,0),(\infty_1,3)\}$,
&$\{(1,0),(11,0),(\infty_4,2)\}$,\\
&$\{(1,0),(\infty_0,0),(10,2)\}$,
&$\{(1,0),(\infty_1,0),(11,3)\}$,
&$\{(1,0),(\infty_3,0),(3,3)\}$,\\
&$\{(1,0),(1,1),(9,3)\}$,
&$\{(1,0),(2,1),(\infty_2,2)\}$,
&$\{(1,0),(6,1),(\infty_2,3)\}$,\\
&$\{(1,0),(9,1),(\infty_3,3)\}$,
&$\{(1,0),(10,1),(\infty_4,3)\}$,
&$\{(1,0),(11,1),(\infty_2,1)\}$,\\
&$\{(1,0),(\infty_1,1),(10,3)\}$,
&$\{(1,0),(3,2),(\infty_1,2)\}$,
&$\{(1,0),(11,2),(\infty_0,3)\}$,\\
&$\{(1,0),(\infty_0,2),(2,3)\}$,
&$\{(2,0),(3,0),(\infty_2,3)\}$,
&$\{(2,0),(6,0),(6,3)\}$,\\
&$\{(2,0),(7,0),(\infty_3,3)\}$,
&$\{(2,0),(11,0),(\infty_4,3)\}$,
&$\{(2,0),(\infty_0,0),(11,1)\}$,\\
&$\{(2,0),(\infty_1,0),(10,3)\}$,
&$\{(2,0),(\infty_2,0),(11,3)\}$,
&$\{(2,0),(\infty_3,0),(10,2)\}$,\\
&$\{(2,0),(3,2),(11,2)\}$,
&$\{(3,0),(\infty_3,0),(11,2)\}$,
&$\{(3,0),(\infty_4,0),(7,3)\}$,\\
&$\{(0,0),(4,0),(8,0)\}^*$.
\end{longtable}
\end{center}

\section{Appendix: Codewords in the proof of Lemma \ref{n=6}}\label{Appendix:n=6}

\begin{center}
\begin{longtable}{lllll}
$m=7$:

&$\{(0,0),(1,0),(2,0)\}$,
&$\{(0,0),(3,0),(6,5)\}$,
&$\{(0,0),(4,0),(1,4)\}$,\\
&$\{(0,0),(5,0),(6,4)\}$,
&$\{(0,0),(6,0),(5,5)\}$,
&$\{(0,0),(0,1),(3,2)\}$,\\
&$\{(0,0),(1,1),(2,2)\}$,
&$\{(0,0),(2,1),(5,1)\}$,
&$\{(0,0),(4,1),(0,2)\}$,\\
&$\{(0,0),(6,1),(6,2)\}$,
&$\{(0,0),(1,2),(2,5)\}$,
&$\{(0,0),(4,2),(5,2)\}$,\\
&$\{(0,0),(1,3),(3,3)\}$,
&$\{(0,0),(2,3),(6,3)\}$,
&$\{(0,0),(4,3),(2,4)\}$,\\
&$\{(0,0),(5,3),(4,4)\}$,
&$\{(0,0),(3,4),(1,5)\}$,
&$\{(0,0),(5,4),(3,5)\}$,\\
&$\{(1,0),(4,0),(4,5)\}$,
&$\{(1,0),(5,0),(1,4)\}$,
&$\{(1,0),(6,0),(6,4)\}$,\\
&$\{(1,0),(1,1),(5,4)\}$,
&$\{(1,0),(3,1),(6,3)\}$,
&$\{(1,0),(4,1),(6,2)\}$,\\
&$\{(1,0),(5,1),(2,4)\}$,
&$\{(1,0),(6,1),(3,4)\}$,
&$\{(1,0),(2,2),(3,2)\}$,\\
&$\{(1,0),(3,3),(4,4)\}$,
&$\{(1,0),(4,3),(2,5)\}$,
&$\{(1,0),(5,5),(6,5)\}$,\\
&$\{(2,0),(4,0),(3,1)\}$,
&$\{(2,0),(2,1),(4,2)\}$,
&$\{(2,0),(5,1),(5,2)\}$,\\
&$\{(2,0),(6,1),(3,3)\}$,
&$\{(2,0),(2,2),(6,4)\}$,
&$\{(2,0),(3,2),(5,5)\}$,\\
&$\{(2,0),(4,3),(3,5)\}$,
&$\{(2,0),(6,3),(5,4)\}$,
&$\{(2,0),(3,4),(6,5)\}$,\\
&$\{(3,0),(4,0),(3,4)\}$,
&$\{(3,0),(5,0),(3,5)\}$,
&$\{(3,0),(6,0),(4,3)\}$,\\
&$\{(3,0),(5,2),(5,4)\}$,
&$\{(4,0),(6,0),(5,4)\}$,
&$\{(4,0),(5,1),(4,4)\}$,\\
&$\{(4,0),(5,2),(6,5)\}$.\\

$m=8$:

&$\{(0,0),(1,0),(2,0)\}$,
&$\{(0,0),(3,0),(1,5)\}$,
&$\{(0,0),(4,0),(3,4)\}$,\\
&$\{(0,0),(5,0),(6,5)\}$,
&$\{(0,0),(6,0),(7,0)\}$,
&$\{(0,0),(0,1),(1,2)\}$,\\
&$\{(0,0),(2,1),(4,3)\}$,
&$\{(0,0),(3,1),(4,1)\}$,
&$\{(0,0),(5,1),(6,3)\}$,\\
&$\{(0,0),(6,1),(2,4)\}$,
&$\{(0,0),(7,1),(0,4)\}$,
&$\{(0,0),(2,2),(3,5)\}$,\\
&$\{(0,0),(3,2),(5,2)\}$,
&$\{(0,0),(4,2),(6,2)\}$,
&$\{(0,0),(7,2),(3,3)\}$,\\
&$\{(0,0),(1,3),(5,3)\}$,
&$\{(0,0),(2,3),(1,4)\}$,
&$\{(0,0),(4,4),(5,4)\}$,\\
&$\{(0,0),(6,4),(2,5)\}$,
&$\{(0,0),(7,4),(4,5)\}$,
&$\{(0,0),(5,5),(7,5)\}$,\\
&$\{(1,0),(3,0),(7,4)\}$,
&$\{(1,0),(4,0),(4,1)\}$,
&$\{(1,0),(6,0),(6,1)\}$,\\
&$\{(1,0),(7,0),(1,1)\}$,
&$\{(1,0),(2,1),(5,1)\}$,
&$\{(1,0),(7,1),(6,5)\}$,\\
&$\{(1,0),(1,2),(5,5)\}$,
&$\{(1,0),(2,2),(4,5)\}$,
&$\{(1,0),(3,2),(6,4)\}$,\\
&$\{(1,0),(4,2),(7,2)\}$,
&$\{(1,0),(5,2),(2,3)\}$,
&$\{(1,0),(6,2),(7,3)\}$,\\
&$\{(1,0),(3,3),(3,4)\}$,
&$\{(1,0),(4,3),(5,4)\}$,
&$\{(1,0),(6,3),(3,5)\}$,\\
&$\{(1,0),(2,4),(4,4)\}$,
&$\{(2,0),(3,0),(4,5)\}$,
&$\{(2,0),(6,0),(3,5)\}$,\\
&$\{(2,0),(7,0),(5,1)\}$,
&$\{(2,0),(2,1),(3,2)\}$,
&$\{(2,0),(4,1),(5,3)\}$,\\
&$\{(2,0),(6,1),(7,5)\}$,
&$\{(2,0),(7,1),(7,2)\}$,
&$\{(2,0),(2,2),(5,4)\}$,\\
&$\{(2,0),(6,2),(4,4)\}$,
&$\{(2,0),(7,3),(6,4)\}$,
&$\{(2,0),(3,4),(7,4)\}$,\\
&$\{(3,0),(6,0),(5,2)\}$,
&$\{(3,0),(4,1),(7,3)\}$,
&$\{(3,0),(5,1),(3,2)\}$,\\
&$\{(3,0),(7,1),(5,4)\}$,
&$\{(3,0),(7,2),(4,4)\}$,
&$\{(3,0),(4,3),(6,5)\}$,\\
&$\{(3,0),(5,3),(6,3)\}$,
&$\{(4,0),(6,1),(5,4)\}$,
&$\{(4,0),(7,1),(7,3)\}$,\\
&$\{(4,0),(4,2),(6,5)\}$,
&$\{(4,0),(5,3),(5,5)\}$,
&$\{(5,0),(5,1),(7,2)\}$,\\
&$\{(5,0),(6,1),(7,4)\}$.\\

$m=10$:
&$\{(0,0),(1,0),(2,0)\}$,
&$\{(0,0),(3,0),(1,4)\}$,
&$\{(0,0),(4,0),(7,4)\}$,\\
&$\{(0,0),(5,0),(4,3)\}$,
&$\{(0,0),(6,0),(0,4)\}$,
&$\{(0,0),(7,0),(9,2)\}$,\\
&$\{(0,0),(8,0),(9,0)\}$,
&$\{(0,0),(0,1),(1,2)\}$,
&$\{(0,0),(2,1),(8,1)\}$,\\
&$\{(0,0),(3,1),(2,2)\}$,
&$\{(0,0),(4,1),(6,4)\}$,
&$\{(0,0),(5,1),(6,1)\}$,\\
&$\{(0,0),(7,1),(9,1)\}$,
&$\{(0,0),(3,2),(4,2)\}$,
&$\{(0,0),(5,2),(7,2)\}$,\\
&$\{(0,0),(8,2),(1,3)\}$,
&$\{(0,0),(2,3),(3,3)\}$,
&$\{(0,0),(5,3),(5,5)\}$,\\
&$\{(0,0),(6,3),(7,3)\}$,
&$\{(0,0),(8,3),(2,4)\}$,
&$\{(0,0),(9,3),(3,4)\}$,\\
&$\{(0,0),(4,4),(1,5)\}$,
&$\{(0,0),(5,4),(8,4)\}$,
&$\{(0,0),(9,4),(2,5)\}$,\\
&$\{(0,0),(3,5),(6,5)\}$,
&$\{(0,0),(4,5),(9,5)\}$,
&$\{(0,0),(7,5),(8,5)\}$,\\
&$\{(1,0),(3,0),(8,3)\}$,
&$\{(1,0),(4,0),(7,3)\}$,
&$\{(1,0),(5,0),(7,4)\}$,\\
&$\{(1,0),(6,0),(8,0)\}$,
&$\{(1,0),(7,0),(1,1)\}$,
&$\{(1,0),(9,0),(8,1)\}$,\\
&$\{(1,0),(2,1),(4,1)\}$,
&$\{(1,0),(3,1),(2,5)\}$,
&$\{(1,0),(5,1),(8,4)\}$,\\
&$\{(1,0),(6,1),(9,1)\}$,
&$\{(1,0),(7,1),(5,5)\}$,
&$\{(1,0),(1,2),(6,5)\}$,\\
&$\{(1,0),(2,2),(5,2)\}$,
&$\{(1,0),(4,2),(3,5)\}$,
&$\{(1,0),(6,2),(2,3)\}$,\\
&$\{(1,0),(7,2),(3,3)\}$,
&$\{(1,0),(8,2),(4,3)\}$,
&$\{(1,0),(9,2),(4,4)\}$,\\
&$\{(1,0),(5,3),(2,4)\}$,
&$\{(1,0),(9,3),(5,4)\}$,
&$\{(1,0),(3,4),(9,4)\}$,\\
&$\{(1,0),(6,4),(9,5)\}$,
&$\{(2,0),(6,0),(2,5)\}$,
&$\{(2,0),(7,0),(7,1)\}$,\\
&$\{(2,0),(9,0),(5,3)\}$,
&$\{(2,0),(3,1),(4,5)\}$,
&$\{(2,0),(4,1),(8,4)\}$,\\
&$\{(2,0),(5,1),(6,4)\}$,
&$\{(2,0),(8,1),(2,4)\}$,
&$\{(2,0),(9,1),(4,2)\}$,\\
&$\{(2,0),(5,2),(9,2)\}$,
&$\{(2,0),(6,2),(3,3)\}$,
&$\{(2,0),(7,2),(4,3)\}$,\\
&$\{(2,0),(8,2),(7,4)\}$,
&$\{(2,0),(6,3),(4,4)\}$,
&$\{(2,0),(7,3),(5,4)\}$,\\
&$\{(2,0),(9,3),(9,4)\}$,
&$\{(2,0),(3,4),(7,5)\}$,
&$\{(3,0),(5,0),(3,1)\}$,\\
&$\{(3,0),(7,0),(3,2)\}$,
&$\{(3,0),(8,0),(4,2)\}$,
&$\{(3,0),(4,1),(5,3)\}$,\\
&$\{(3,0),(5,1),(6,3)\}$,
&$\{(3,0),(6,1),(5,2)\}$,
&$\{(3,0),(8,1),(9,4)\}$,\\
&$\{(3,0),(9,1),(6,2)\}$,
&$\{(3,0),(7,2),(9,3)\}$,
&$\{(3,0),(8,2),(7,3)\}$,\\
&$\{(3,0),(9,2),(6,4)\}$,
&$\{(3,0),(5,4),(4,5)\}$,
&$\{(3,0),(8,4),(8,5)\}$,\\
&$\{(4,0),(5,0),(4,5)\}$,
&$\{(4,0),(6,0),(6,4)\}$,
&$\{(4,0),(7,0),(8,1)\}$,\\
&$\{(4,0),(8,0),(4,4)\}$,
&$\{(4,0),(6,1),(7,2)\}$,
&$\{(4,0),(7,1),(6,2)\}$,\\
&$\{(4,0),(9,1),(9,3)\}$,
&$\{(4,0),(9,2),(5,4)\}$,
&$\{(5,0),(5,1),(9,2)\}$,\\
&$\{(5,0),(6,1),(7,3)\}$,
&$\{(5,0),(7,1),(8,4)\}$,
&$\{(5,0),(8,1),(8,5)\}$,\\
&$\{(5,0),(8,2),(6,4)\}$,
&$\{(6,0),(6,1),(8,3)\}$,
&$\{(6,0),(8,1),(9,3)\}$,\\
&$\{(6,0),(9,2),(7,4)\}$,
&$\{(6,0),(7,3),(8,5)\}$,
&$\{(7,0),(7,2),(9,5)\}$.\\
$m=11$:
&$\{(0,0),(1,0),(2,0)\}$,
&$\{(0,0),(3,0),(8,3)\}$,
&$\{(0,0),(4,0),(9,4)\}$,\\
&$\{(0,0),(5,0),(9,5)\}$,
&$\{(0,0),(6,0),(1,3)\}$,
&$\{(0,0),(7,0),(7,2)\}$,\\
&$\{(0,0),(8,0),(5,1)\}$,
&$\{(0,0),(9,0),(5,5)\}$,
&$\{(0,0),(10,0),(10,4)\}$,\\
&$\{(0,0),(0,1),(2,3)\}$,
&$\{(0,0),(1,1),(2,4)\}$,
&$\{(0,0),(2,1),(8,2)\}$,\\
&$\{(0,0),(3,1),(9,2)\}$,
&$\{(0,0),(4,1),(5,4)\}$,
&$\{(0,0),(6,1),(8,5)\}$,\\
&$\{(0,0),(7,1),(6,5)\}$,
&$\{(0,0),(8,1),(10,2)\}$,
&$\{(0,0),(9,1),(3,2)\}$,\\
&$\{(0,0),(10,1),(7,5)\}$,
&$\{(0,0),(0,2),(1,4)\}$,
&$\{(0,0),(4,2),(5,2)\}$,\\
&$\{(0,0),(6,2),(8,4)\}$,
&$\{(0,0),(3,3),(4,3)\}$,
&$\{(0,0),(5,3),(6,3)\}$,\\
&$\{(0,0),(7,3),(9,3)\}$,
&$\{(0,0),(10,3),(3,4)\}$,
&$\{(0,0),(4,4),(6,4)\}$,\\
&$\{(0,0),(7,4),(10,5)\}$,
&$\{(0,0),(1,5),(3,5)\}$,
&$\{(0,0),(2,5),(4,5)\}$,\\
&$\{(1,0),(4,0),(7,2)\}$,
&$\{(1,0),(5,0),(3,3)\}$,
&$\{(1,0),(6,0),(1,1)\}$,\\
&$\{(1,0),(7,0),(9,4)\}$,
&$\{(1,0),(8,0),(10,4)\}$,
&$\{(1,0),(9,0),(10,1)\}$,\\
&$\{(1,0),(10,0),(2,4)\}$,
&$\{(1,0),(2,1),(4,5)\}$,
&$\{(1,0),(3,1),(10,2)\}$,\\
&$\{(1,0),(4,1),(7,1)\}$,
&$\{(1,0),(5,1),(5,5)\}$,
&$\{(1,0),(6,1),(7,5)\}$,\\
&$\{(1,0),(8,1),(7,4)\}$,
&$\{(1,0),(9,1),(5,3)\}$,
&$\{(1,0),(1,2),(3,4)\}$,\\
&$\{(1,0),(2,2),(8,4)\}$,
&$\{(1,0),(4,2),(4,4)\}$,
&$\{(1,0),(5,2),(8,2)\}$,\\
&$\{(1,0),(6,2),(2,5)\}$,
&$\{(1,0),(9,2),(4,3)\}$,
&$\{(1,0),(7,3),(8,3)\}$,\\
&$\{(1,0),(9,3),(10,3)\}$,
&$\{(1,0),(5,4),(3,5)\}$,
&$\{(1,0),(6,4),(9,5)\}$,\\
&$\{(1,0),(8,5),(10,5)\}$,
&$\{(2,0),(3,0),(2,5)\}$,
&$\{(2,0),(5,0),(4,1)\}$,\\
&$\{(2,0),(6,0),(7,3)\}$,
&$\{(2,0),(7,0),(10,0)\}$,
&$\{(2,0),(8,0),(8,4)\}$,\\
&$\{(2,0),(9,0),(8,3)\}$,
&$\{(2,0),(5,1),(10,5)\}$,
&$\{(2,0),(6,1),(7,1)\}$,\\
&$\{(2,0),(9,1),(5,5)\}$,
&$\{(2,0),(10,1),(4,2)\}$,
&$\{(2,0),(2,2),(3,4)\}$,\\
&$\{(2,0),(5,2),(9,5)\}$,
&$\{(2,0),(6,2),(3,3)\}$,
&$\{(2,0),(7,2),(3,5)\}$,\\
&$\{(2,0),(9,2),(9,3)\}$,
&$\{(2,0),(4,3),(6,5)\}$,
&$\{(2,0),(5,3),(5,4)\}$,\\
&$\{(2,0),(10,3),(7,4)\}$,
&$\{(2,0),(6,4),(10,4)\}$,
&$\{(2,0),(9,4),(7,5)\}$,\\
&$\{(2,0),(4,5),(8,5)\}$,
&$\{(3,0),(5,0),(6,1)\}$,
&$\{(3,0),(6,0),(7,5)\}$,\\
&$\{(3,0),(7,0),(3,2)\}$,
&$\{(3,0),(8,0),(3,1)\}$,
&$\{(3,0),(9,0),(6,2)\}$,\\
&$\{(3,0),(10,0),(5,1)\}$,
&$\{(3,0),(4,1),(6,4)\}$,
&$\{(3,0),(7,1),(4,2)\}$,\\
&$\{(3,0),(8,1),(9,3)\}$,
&$\{(3,0),(5,2),(4,4)\}$,
&$\{(3,0),(7,2),(5,4)\}$,\\
&$\{(3,0),(8,2),(9,2)\}$,
&$\{(3,0),(10,2),(10,3)\}$,
&$\{(3,0),(4,3),(9,4)\}$,\\
&$\{(3,0),(6,3),(10,4)\}$,
&$\{(3,0),(8,4),(4,5)\}$,
&$\{(4,0),(9,0),(8,1)\}$,\\
&$\{(4,0),(10,0),(9,2)\}$,
&$\{(4,0),(4,1),(10,2)\}$,
&$\{(4,0),(5,1),(7,4)\}$,\\
&$\{(4,0),(6,1),(9,3)\}$,
&$\{(4,0),(7,1),(10,4)\}$,
&$\{(4,0),(5,2),(7,3)\}$,\\
&$\{(4,0),(8,2),(6,5)\}$,
&$\{(4,0),(8,3),(8,4)\}$,
&$\{(4,0),(10,3),(6,4)\}$,\\
&$\{(5,0),(7,0),(7,5)\}$,
&$\{(5,0),(9,0),(6,3)\}$,
&$\{(5,0),(10,0),(8,3)\}$,\\
&$\{(5,0),(8,1),(10,3)\}$,
&$\{(5,0),(10,1),(8,2)\}$,
&$\{(5,0),(6,2),(6,4)\}$,\\
&$\{(5,0),(7,2),(8,4)\}$,
&$\{(5,0),(10,2),(6,5)\}$,
&$\{(6,0),(8,0),(6,5)\}$,\\
&$\{(6,0),(9,0),(10,2)\}$,
&$\{(6,0),(7,1),(8,5)\}$,
&$\{(6,0),(10,4),(9,5)\}$,\\
&$\{(7,0),(8,1),(9,2)\}$,
&$\{(7,0),(9,1),(10,4)\}$,
&$\{(7,0),(9,3),(8,5)\}$.
\end{longtable}
\end{center}

\section{Appendix: Codewords in the proof of Lemma \ref{n=10}}\label{Appendix:n=10}

\begin{center}
\begin{longtable}{lllll}
$m=7$:
&$\{(0,0),(1,0),(2,0)\}$,
&$\{(0,0),(3,0),(6,5)\}$,
&$\{(0,0),(4,0),(1,8)\}$,\\
&$\{(0,0),(5,0),(6,4)\}$,
&$\{(0,0),(6,0),(5,5)\}$,
&$\{(0,0),(0,1),(3,8)\}$,\\
&$\{(0,0),(1,1),(2,6)\}$,
&$\{(0,0),(2,1),(5,9)\}$,
&$\{(0,0),(3,1),(1,4)\}$,\\
&$\{(0,0),(4,1),(5,1)\}$,
&$\{(0,0),(6,1),(1,5)\}$,
&$\{(0,0),(0,2),(2,5)\}$,\\
&$\{(0,0),(1,2),(3,2)\}$,
&$\{(0,0),(2,2),(1,6)\}$,
&$\{(0,0),(4,2),(6,3)\}$,\\
&$\{(0,0),(5,2),(1,3)\}$,
&$\{(0,0),(6,2),(0,3)\}$,
&$\{(0,0),(3,3),(6,7)\}$,\\
&$\{(0,0),(4,3),(0,6)\}$,
&$\{(0,0),(5,3),(4,9)\}$,
&$\{(0,0),(2,4),(4,6)\}$,\\
&$\{(0,0),(3,4),(2,7)\}$,
&$\{(0,0),(4,4),(3,6)\}$,
&$\{(0,0),(5,4),(5,7)\}$,\\
&$\{(0,0),(3,5),(4,5)\}$,
&$\{(0,0),(5,6),(6,6)\}$,
&$\{(0,0),(1,7),(6,8)\}$,\\
&$\{(0,0),(2,8),(5,8)\}$,
&$\{(0,0),(4,8),(1,9)\}$,
&$\{(0,0),(2,9),(3,9)\}$,\\
&$\{(1,0),(4,0),(1,2)\}$,
&$\{(1,0),(5,0),(2,8)\}$,
&$\{(1,0),(6,0),(5,3)\}$,\\
&$\{(1,0),(1,1),(2,2)\}$,
&$\{(1,0),(3,1),(3,8)\}$,
&$\{(1,0),(4,1),(1,7)\}$,\\
&$\{(1,0),(5,1),(6,4)\}$,
&$\{(1,0),(3,2),(5,6)\}$,
&$\{(1,0),(5,2),(2,7)\}$,\\
&$\{(1,0),(6,2),(4,3)\}$,
&$\{(1,0),(2,3),(6,3)\}$,
&$\{(1,0),(3,3),(1,4)\}$,\\
&$\{(1,0),(2,4),(3,5)\}$,
&$\{(1,0),(3,4),(5,4)\}$,
&$\{(1,0),(4,5),(6,5)\}$,\\
&$\{(1,0),(5,5),(3,6)\}$,
&$\{(1,0),(4,6),(4,7)\}$,
&$\{(1,0),(5,7),(5,8)\}$,\\
&$\{(1,0),(6,7),(2,9)\}$,
&$\{(1,0),(6,8),(6,9)\}$,
&$\{(2,0),(4,0),(5,9)\}$,\\
&$\{(2,0),(2,1),(3,9)\}$,
&$\{(2,0),(4,1),(2,6)\}$,
&$\{(2,0),(5,1),(4,9)\}$,\\
&$\{(2,0),(6,1),(3,2)\}$,
&$\{(2,0),(2,2),(6,7)\}$,
&$\{(2,0),(6,2),(3,6)\}$,\\
&$\{(2,0),(2,3),(4,6)\}$,
&$\{(2,0),(3,3),(6,3)\}$,
&$\{(2,0),(5,3),(3,5)\}$,\\
&$\{(2,0),(3,4),(5,6)\}$,
&$\{(2,0),(4,4),(6,6)\}$,
&$\{(2,0),(5,4),(4,7)\}$,\\
&$\{(2,0),(6,4),(4,8)\}$,
&$\{(2,0),(5,7),(6,9)\}$,
&$\{(3,0),(3,1),(4,2)\}$,\\
&$\{(3,0),(5,1),(5,7)\}$,
&$\{(3,0),(6,1),(5,3)\}$,
&$\{(3,0),(3,2),(4,6)\}$,\\
&$\{(3,0),(6,2),(5,6)\}$,
&$\{(3,0),(4,3),(3,6)\}$,
&$\{(3,0),(6,3),(6,7)\}$,\\
&$\{(3,0),(4,5),(6,8)\}$,
&$\{(3,0),(5,5),(4,9)\}$,
&$\{(4,0),(5,1),(5,3)\}$,\\
&$\{(4,0),(4,2),(6,7)\}$,
&$\{(4,0),(4,3),(5,8)\}$,
&$\{(4,0),(4,4),(6,8)\}$,\\
&$\{(5,0),(6,1),(6,9)\}$.\\

$m=10$:
&$\{(0,0),(1,0),(2,0)\}$,
&$\{(0,0),(3,0),(1,4)\}$,
&$\{(0,0),(4,0),(7,6)\}$,\\
&$\{(0,0),(5,0),(4,3)\}$,
&$\{(0,0),(6,0),(7,0)\}$,
&$\{(0,0),(8,0),(9,6)\}$,\\
&$\{(0,0),(9,0),(4,2)\}$,
&$\{(0,0),(0,1),(8,7)\}$,
&$\{(0,0),(1,1),(8,5)\}$,\\
&$\{(0,0),(2,1),(3,6)\}$,
&$\{(0,0),(3,1),(4,6)\}$,
&$\{(0,0),(4,1),(5,1)\}$,\\
&$\{(0,0),(6,1),(5,4)\}$,
&$\{(0,0),(7,1),(1,8)\}$,
&$\{(0,0),(8,1),(7,2)\}$,\\
&$\{(0,0),(9,1),(1,6)\}$,
&$\{(0,0),(0,2),(1,9)\}$,
&$\{(0,0),(1,2),(5,9)\}$,\\
&$\{(0,0),(2,2),(2,4)\}$,
&$\{(0,0),(3,2),(8,2)\}$,
&$\{(0,0),(5,2),(6,3)\}$,\\
&$\{(0,0),(6,2),(2,9)\}$,
&$\{(0,0),(9,2),(0,3)\}$,
&$\{(0,0),(1,3),(3,3)\}$,\\
&$\{(0,0),(2,3),(3,5)\}$,
&$\{(0,0),(5,3),(3,9)\}$,
&$\{(0,0),(7,3),(2,8)\}$,\\
&$\{(0,0),(8,3),(9,3)\}$,
&$\{(0,0),(0,4),(3,8)\}$,
&$\{(0,0),(4,4),(6,4)\}$,\\
&$\{(0,0),(7,4),(6,9)\}$,
&$\{(0,0),(8,4),(1,5)\}$,
&$\{(0,0),(9,4),(2,5)\}$,\\
&$\{(0,0),(4,5),(2,7)\}$,
&$\{(0,0),(5,5),(6,5)\}$,
&$\{(0,0),(7,5),(3,7)\}$,\\
&$\{(0,0),(9,5),(5,6)\}$,
&$\{(0,0),(2,6),(5,7)\}$,
&$\{(0,0),(6,6),(4,7)\}$,\\
&$\{(0,0),(6,7),(4,9)\}$,
&$\{(0,0),(7,7),(9,7)\}$,
&$\{(0,0),(4,8),(7,8)\}$,\\
&$\{(0,0),(5,8),(8,8)\}$,
&$\{(0,0),(6,8),(9,8)\}$,
&$\{(0,0),(7,9),(8,9)\}$,\\
&$\{(1,0),(4,0),(4,6)\}$,
&$\{(1,0),(5,0),(1,4)\}$,
&$\{(1,0),(6,0),(1,1)\}$,\\
&$\{(1,0),(7,0),(3,5)\}$,
&$\{(1,0),(8,0),(8,6)\}$,
&$\{(1,0),(9,0),(7,4)\}$,\\
&$\{(1,0),(2,1),(4,4)\}$,
&$\{(1,0),(3,1),(6,6)\}$,
&$\{(1,0),(4,1),(7,5)\}$,\\
&$\{(1,0),(5,1),(8,3)\}$,
&$\{(1,0),(6,1),(2,6)\}$,
&$\{(1,0),(7,1),(3,2)\}$,\\
&$\{(1,0),(8,1),(2,4)\}$,
&$\{(1,0),(9,1),(9,2)\}$,
&$\{(1,0),(1,2),(8,7)\}$,\\
&$\{(1,0),(2,2),(4,2)\}$,
&$\{(1,0),(5,2),(5,3)\}$,
&$\{(1,0),(6,2),(3,9)\}$,\\
&$\{(1,0),(7,2),(5,4)\}$,
&$\{(1,0),(8,2),(2,9)\}$,
&$\{(1,0),(1,3),(9,6)\}$,\\
&$\{(1,0),(2,3),(3,4)\}$,
&$\{(1,0),(3,3),(6,4)\}$,
&$\{(1,0),(4,3),(7,6)\}$,\\
&$\{(1,0),(6,3),(9,4)\}$,
&$\{(1,0),(2,5),(5,5)\}$,
&$\{(1,0),(4,5),(3,7)\}$,\\
&$\{(1,0),(6,5),(4,9)\}$,
&$\{(1,0),(2,7),(6,7)\}$,
&$\{(1,0),(4,7),(9,7)\}$,\\
&$\{(1,0),(7,7),(2,8)\}$,
&$\{(1,0),(3,8),(4,8)\}$,
&$\{(1,0),(5,8),(7,8)\}$,\\
&$\{(1,0),(6,8),(5,9)\}$,
&$\{(1,0),(8,8),(9,9)\}$,
&$\{(1,0),(9,8),(7,9)\}$,\\
&$\{(2,0),(3,0),(4,2)\}$,
&$\{(2,0),(7,0),(5,8)\}$,
&$\{(2,0),(8,0),(4,5)\}$,\\
&$\{(2,0),(9,0),(8,5)\}$,
&$\{(2,0),(2,1),(2,4)\}$,
&$\{(2,0),(4,1),(6,6)\}$,\\
&$\{(2,0),(6,1),(6,7)\}$,
&$\{(2,0),(7,1),(7,3)\}$,
&$\{(2,0),(8,1),(5,2)\}$,\\
&$\{(2,0),(9,1),(4,4)\}$,
&$\{(2,0),(6,2),(3,4)\}$,
&$\{(2,0),(7,2),(5,3)\}$,\\
&$\{(2,0),(8,2),(3,3)\}$,
&$\{(2,0),(9,2),(5,4)\}$,
&$\{(2,0),(9,3),(6,4)\}$,\\
&$\{(2,0),(7,4),(9,5)\}$,
&$\{(2,0),(8,4),(3,8)\}$,
&$\{(2,0),(9,4),(3,6)\}$,\\
&$\{(2,0),(5,5),(4,6)\}$,
&$\{(2,0),(5,6),(9,6)\}$,
&$\{(2,0),(7,6),(4,7)\}$,\\
&$\{(2,0),(8,6),(6,8)\}$,
&$\{(2,0),(3,7),(7,7)\}$,
&$\{(2,0),(5,7),(7,8)\}$,\\
&$\{(2,0),(9,7),(8,8)\}$,
&$\{(2,0),(9,8),(4,9)\}$,
&$\{(2,0),(3,9),(5,9)\}$,\\
&$\{(2,0),(6,9),(8,9)\}$,
&$\{(3,0),(6,0),(3,1)\}$,
&$\{(3,0),(9,0),(9,6)\}$,\\
&$\{(3,0),(4,1),(8,1)\}$,
&$\{(3,0),(5,1),(4,9)\}$,
&$\{(3,0),(7,1),(4,6)\}$,\\
&$\{(3,0),(9,1),(5,5)\}$,
&$\{(3,0),(3,2),(6,4)\}$,
&$\{(3,0),(5,2),(9,5)\}$,\\
&$\{(3,0),(7,2),(7,3)\}$,
&$\{(3,0),(8,2),(8,3)\}$,
&$\{(3,0),(9,2),(8,4)\}$,\\
&$\{(3,0),(3,3),(8,8)\}$,
&$\{(3,0),(4,3),(5,9)\}$,
&$\{(3,0),(5,3),(9,4)\}$,\\
&$\{(3,0),(9,3),(3,6)\}$,
&$\{(3,0),(4,4),(9,9)\}$,
&$\{(3,0),(7,4),(5,7)\}$,\\
&$\{(3,0),(5,6),(8,7)\}$,
&$\{(3,0),(6,6),(7,7)\}$,
&$\{(3,0),(7,6),(6,7)\}$,\\
&$\{(3,0),(4,7),(5,8)\}$,
&$\{(4,0),(4,1),(8,9)\}$,
&$\{(4,0),(6,1),(7,8)\}$,\\
&$\{(4,0),(7,1),(6,3)\}$,
&$\{(4,0),(8,1),(9,4)\}$,
&$\{(4,0),(9,1),(5,4)\}$,\\
&$\{(4,0),(4,2),(8,4)\}$,
&$\{(4,0),(6,2),(6,4)\}$,
&$\{(4,0),(7,2),(5,8)\}$,\\
&$\{(4,0),(9,2),(7,7)\}$,
&$\{(4,0),(4,3),(8,6)\}$,
&$\{(4,0),(5,3),(8,7)\}$,\\
&$\{(4,0),(9,3),(9,6)\}$,
&$\{(4,0),(5,5),(6,7)\}$,
&$\{(5,0),(5,2),(9,4)\}$,\\
&$\{(5,0),(5,3),(6,6)\}$,
&$\{(5,0),(7,3),(8,5)\}$,
&$\{(5,0),(8,3),(5,6)\}$,\\
&$\{(5,0),(6,4),(7,6)\}$,
&$\{(5,0),(6,5),(8,6)\}$,
&$\{(5,0),(7,5),(8,8)\}$,\\
&$\{(5,0),(9,5),(6,8)\}$,
&$\{(6,0),(6,1),(8,3)\}$,
&$\{(6,0),(9,2),(6,7)\}$,\\
&$\{(6,0),(7,3),(8,7)\}$,
&$\{(6,0),(9,3),(8,9)\}$,
&$\{(6,0),(7,4),(8,5)\}$,\\
&$\{(6,0),(8,4),(7,6)\}$,
&$\{(6,0),(9,4),(9,6)\}$,
&$\{(6,0),(8,6),(9,8)\}$,\\
&$\{(7,0),(9,2),(7,4)\}$,
&$\{(7,0),(7,3),(9,7)\}$,
&$\{(7,0),(9,3),(8,6)\}$,\\
&$\{(7,0),(8,5),(8,7)\}$.
\end{longtable}
\end{center}

\section{Appendix: The fourth part in the proof of Lemma \ref{m<12,n=2mod8}}\label{fourthpart m=5}

All the codewords are of the form $\{(\infty,0),(\infty,a),(\infty,b)\}$. We only list the second component of
each point of these codewords. For $n=26$, the fourth part consists of $\{0,2,8\}$ and $\{0,4,14\}$. For $n\neq26$,
the fourth part consists of $\{0,2,6\}$ and the following $\lfloor(n-16)/6\rfloor$ codewords. Note that $\lfloor(n-16)/6\rfloor=0$
if $n=18,20$.
\begin{center}
\begin{longtable}{lllllll}
$n=50$:
&$\{0,8,17\}$,
&$\{0,10,21\}$,
&$\{0,12,27\}$,
&$\{0,13,31\}$,
&$\{0,14,30\}$.\\
$n=74$:
&$\{0,8,17\}$,
&$\{0,10,21\}$,
&$\{0,12,25\}$,
&$\{0,14,38\}$,
&$\{0,15,42\}$,
&$\{0,16,44\}$,\\
&$\{0,18,40\}$,
&$\{0,19,45\}$,
&$\{0,20,43\}$.\\
$n=98$:
&$\{0,8,17\}$,
&$\{0,10,21\}$,
&$\{0,12,25\}$,
&$\{0,14,29\}$,
&$\{0,16,43\}$,
&$\{0,18,51\}$,\\
&$\{0,19,53\}$,
&$\{0,20,56\}$,
&$\{0,22,54\}$,
&$\{0,23,60\}$,
&$\{0,24,59\}$,
&$\{0,26,57\}$,\\
&$\{0,28,58\}$.\\
$n=44$:
&$\{0,8,17\}$,
&$\{0,10,21\}$,
&$\{0,12,26\}$,
&$\{0,13,28\}$.\\
$n=68$:
&$\{0,8,17\}$,
&$\{0,10,21\}$,
&$\{0,12,26\}$,
&$\{0,13,37\}$,
&$\{0,15,38\}$,
&$\{0,16,41\}$,\\
&$\{0,18,40\}$,
&$\{0,19,39\}$.\\
$n=92$:
&$\{0,8,17\}$,
&$\{0,10,21\}$,
&$\{0,12,25\}$,
&$\{0,14,29\}$,
&$\{0,16,47\}$,
&$\{0,18,50\}$,\\
&$\{0,19,52\}$,
&$\{0,20,55\}$,
&$\{0,22,56\}$,
&$\{0,23,51\}$,
&$\{0,24,54\}$,
&$\{0,26,53\}$.\\
$n=42$:
&$\{0,8,18\}$,
&$\{0,9,23\}$,
&$\{0,11,26\}$,
&$\{0,12,25\}$.\\
$n=66$:
&$\{0,8,17\}$,
&$\{0,10,21\}$,
&$\{0,12,31\}$,
&$\{0,13,36\}$,
&$\{0,14,39\}$,
&$\{0,15,37\}$,\\
&$\{0,16,40\}$,
&$\{0,18,38\}$.\\
$n=90$:
&$\{0,8,17\}$,
&$\{0,10,21\}$,
&$\{0,12,25\}$,
&$\{0,14,36\}$,
&$\{0,15,47\}$,
&$\{0,16,49\}$,\\
&$\{0,18,48\}$,
&$\{0,19,53\}$,
&$\{0,20,51\}$,
&$\{0,23,50\}$,
&$\{0,24,52\}$,
&$\{0,26,55\}$.\\
$n=36$:
&$\{0,8,20\}$,
&$\{0,9,22\}$,
&$\{0,10,21\}$.\\
$n=60$:
&$\{0,8,17\}$,
&$\{0,10,22\}$,
&$\{0,11,32\}$,
&$\{0,13,36\}$,
&$\{0,14,33\}$,
&$\{0,15,35\}$,\\
&$\{0,16,34\}$.\\
$n=84$:
&$\{0,8,17\}$,
&$\{0,10,21\}$,
&$\{0,12,25\}$,
&$\{0,14,40\}$,
&$\{0,15,43\}$,
&$\{0,16,46\}$,\\
&$\{0,18,50\}$,
&$\{0,19,48\}$,
&$\{0,20,51\}$,
&$\{0,22,49\}$,
&$\{0,23,47\}$.\\
$n=34$:
&$\{0,7,16\}$,
&$\{0,8,20\}$,
&$\{0,10,21\}$.\\
$n=58$:
&$\{0,7,15\}$,
&$\{0,9,19\}$,
&$\{0,11,32\}$,
&$\{0,12,30\}$,
&$\{0,13,35\}$,
&$\{0,14,34\}$,\\
&$\{0,16,33\}$.\\
$n=82$:
&$\{0,7,15\}$,
&$\{0,9,19\}$,
&$\{0,11,23\}$,
&$\{0,13,37\}$,
&$\{0,14,40\}$,
&$\{0,16,44\}$,\\
&$\{0,17,46\}$,
&$\{0,18,49\}$,
&$\{0,20,50\}$,
&$\{0,21,48\}$,
&$\{0,22,47\}$.\\
\end{longtable}
\end{center}


\begin{thebibliography}{99}

\bibitem{ab}R. J. R. Abel, M. Buratti, Some progress on $(v, 4,
1)$ difference families and optical orthogonal codes, J. Combin.
Theory, Ser. A, 106 (2004), 59-75.


\bibitem{am4}
T.L. Alderson, Optical orthogonal codes and arcs in PG$(d,q)$, Finite Fields Appl., 13
(2007), 762-768.

\bibitem{am3}T. L. Alderson, K. E. Mellinger, Constructions of optical
orthogonal codes from finite geometry, SIAM J. Discrete Math., 21
(2007), 785-793.

\bibitem{am1}T. L. Alderson, K. E. Mellinger, Families of optimal OOCs with $\lambda=2$, IEEE Trans. Inf. Theory, 54 (2008),
3722-3724.

\bibitem{am2}T. L. Alderson, K. E. Mellinger, Geometric constructions of optimal
optical orthogonal codes, Adv. Math. Commun., 2 (2008), 451-467.

\bibitem{am}
T.L.~Alderson and K.E.~Mellinger, $2$-Dimensional optical orthogonal codes from Singer groups,
Discrete Appl. Math., 157 (2009), 3008-3019.



\bibitem{Baker}
C.A. Baker, Extended Skolem sequences, J. Combin. Des., 3 (1995), 363-379.

\bibitem{bj}J. Bao, L. Ji, Constructions of strictly $m$-cyclic and semi-cyclic H$(m, n, 4, 3)$, J. Combin. Des., 24 (2016), 249-264.

\bibitem{bt}
S. Bitan and T. Etzion, The last packing number of quadruple, and cyclic SQS, Des. Codes
Cryptogr., 3 (1993), 283-313.

\bibitem{be}S. Bitan, T. Etzion, Constructions for optimal constant weight cyclically permutable codes and difference families,
IEEE Trans. Inf. Theory, 41 (1995), 77-87.

\bibitem{b}
R.C.~Bose, On the construction of balanced incomplete block designs,
Ann. Eugenics., 9 (1939), 353-399.

\bibitem{bw}
E.F. Brickell and V.K. Wei, Optical orthogonal codes and cyclic block designs, Congr. Numer., 58 (1987), 175-192.

\bibitem{b2}
M. Buratti, Old and new designs via difference multisets and strong difference families, J.
Combin. Des., 7 (1999), 406-425.

\bibitem{b3}
M. Buratti, On silver and golden optical orthogonal codes, Art Discrete Appl. Math., 1
(2018), \#P2.02.

\bibitem{b4}
M. Buratti, Cyclic designs with block size 4 and related
optimal optical orthogonal codes, Des. Codes Cryptogr., 26 (2002),
111-125.

\bibitem{bp}M. Buratti and A. Pasotti, Combinatorial designs and the theorem of Weil on multiplicative
character sums, Finite Fields Appl., 15 (2009), 332-344.

\bibitem{bp1}M. Buratti, A. Pasotti, Further progress on difference families
with block size 4 or 5, Des. Codes Cryptogr., 56 (2010), 1-20.

\bibitem{bs}M. Buratti and D.R. Stinson, New results on modular Golomb rulers, optical orthogonal
codes and related structures, Ars Math. Contemp., 20 (2021), 1-27.

\bibitem{cw}
H. Cao and R. Wei, Combinatorial constructions for optimal two-dimensional optical
orthogonal codes, IEEE Trans. Inf. Theory, 55 (2009), 1387-1394.

\bibitem{csfw}Y. Chang, S. Costa, T. Feng, and X.Wang, Strong difference families with special patterns,
Discrete Math., 343 (2020), 111776.

\bibitem{csfw1}Y. Chang, S. Costa, T. Feng, and X. Wang, Partitionable sets, almost partitionable sets
and their applications, J. Combin. Des., 28 (2020), 783-813.

\bibitem{cfm}Y. Chang, R. Fuji-Hara and Y. Miao, Combinatorial constructions
of optimal optical orthogonal codes with weight 4, IEEE Trans. Inf.
Theory, 49 (2003), 1283-1292.

\bibitem{cj} Y. Chang, L. Ji, Optimal $(4up, 5, 1)$ optical orthogonal
codes, J. Combin. Des., 12 (2004), 346-361.

\bibitem{cm1} Y. Chang, Y. Miao, Constructions for optimal optical orthogonal
codes, Discrete Math., 261 (2003), 127-139.

\bibitem{cy} Y. Chang, J. Yin, Further results on optimal optical
orthogonal codes with weight 4, Discrete Math., 279 (2004), 135-151.

\bibitem{cgz}
K. Chen, G. Ge, and L. Zhu, Starters and related codes, J. Statist. Plann. Inference, 86
(2000), 379-395.

\bibitem{cc1}
W. Chu and C.J. Colbourn, Optimal $(v,4,2)$-OOC of small orders, Discrete Math., 279
(2004), 163-172.

\bibitem{cc}
W. Chu, C. J. Colbourn, Recursive constructions for optimal $(n, 4,
2)$-OOCs, J. Combin. Des., 12 (2004), 333-345.

\bibitem{cc2}
W. Chu and S.W. Golomb, A new recursive construction for optical orthogonal codes,
IEEE Trans. Inf. Theory, 49 (2003), 3072-3076.

\bibitem{ck}
H. Chung and P.V. Kumar, Optical orthogonal codes - new bounds and an optimal construction,
IEEE Trans. Inf. Theory, 36 (1990), 866-873.

\bibitem{csw}
F.R.K.~Chung, J.A.~Salehi, and V.K.~Wei, Optical orthogonal
codes: design, analysis and applications, IEEE Trans. Inf.
Theory, 35 (1989), 595-604.


\bibitem{dx}
C. Ding and C. Xing, Several classes of $(2^m-1,w,2)$ optical orthogonal codes, Discrete
Appl. Math., 128 (2003), 103-120.


\bibitem{fc}
T. Feng, Y. Chang, Combinatorial constructions for optimal
two-dimensional optical orthogonal codes with $\lambda=2$, IEEE
Trans. Inf. Theory, 57 (2011), 6796-6819.

\bibitem{fcj}
T. Feng, Y. Chang and L. Ji, Constructions for strictly cyclic
3-designs and applications to optimal OOCs with $\lambda=2$, J.
Combin. Theory, Ser. A, 115 (2008), 1527-1551.

\bibitem{fcj1}
T. Feng, Y. Chang, and L. Ji, Constructions for rotational Steiner quadruple systems, J.
Combin. Des., 17 (2009), 353-368.

\bibitem{fwwzh}
T.~Feng, L.~Wang, X.~Wang, and Y.~Zhao, Optimal two dimensional optical orthogonal
codes with the best cross-correlation constraint, J. Combin. Des., 25 (2017), 349-380.

\bibitem{fww}
T.~Feng, L.~Wang, and X.~Wang, Optimal $2$-D $(n\times m,3,2,1)$ optical orthogonal
codes and related equi-difference conflict avoiding codes, Des. Codes Cryptogr., 87 (2019), 1499-1520.

\bibitem{fwc}
T.~Feng, X.~Wang, and Y.~Chang, Semi-cyclic holey group divisible designs with block size
three, Des. Codes Cryptogr., 74 (2015), 301-324.

\bibitem{fwwei}
T.~Feng, X.~Wang, and R.~Wei, Semi-cyclic holey group divisible designs and applications to sampling designs and optical orthogonal codes, J. Combin. Des., 24 (2016), 201-222.

\bibitem{fm}R. Fuji-Hara and Y. Miao, Optical orthogonal codes: their bounds and new optimal constructions,
IEEE Trans. Inf. Theory, 46 (2000), 2396-2406.

\bibitem{fmy}R. Fuji-Hara, Y. Miao, and J. Yin, Optimal $(9v,4,1)$ optical orthogonal codes, SIAM J.
Discrete Math., 14 (2001), 256-266.

\bibitem{gjl}
R.P. Gallant, Z. Jiang, A.C.H. Ling, The spectrum of cyclic group divisible designs with
block size three. J. Combin. Des., 7 (1999), 95-105.

\bibitem{gms}
G. Ge, Y. Miao, and X. Sun, Perfect difference families, perfect difference matrices, and
related combinatorial structures, J. Combin. Des., 18 (2010), 415-449.

\bibitem{gy} 
G. Ge and J. Yin, Constructions for optimal $(v,4,1)$ optical orthogonal codes, IEEE Trans.
Inf. Theory, 47 (2001), 2998-3004.

\bibitem{hc} 
Y. Huang and Y. Chang, Two classes of optimal two-dimensional OOCs, Des.
Codes Cryptogr., 63 (2012), 357-363.

\bibitem{km}S. Kageyama and Y. Miao, Optical orthogonal codes derived from difference triangle sets,
Congr. Numer., 143 (2000), 129-139.

\bibitem{kpp}
W.C.~Kwong, P.A.~Perrier, and P.R.~Prucnal, Performance
comparison of asynchronous and synchronous code-division
multiple-access techniques for fiber-optic local area networks,
IEEE Trans. Commun., 39 (1991), 1625-1634.

\bibitem{lwl}
Y. Liu, L. Wang, and J. Lei, Constructions of optimal two-dimensional optical orthogonal codes with AM-OPPW restriction for $\lambda=2$, Discrete Math., 346 (2023), 113424.

\bibitem{mc1}S. Ma, Y. Chang, A new class of optimal optical orthogonal
codes with weight five, IEEE Trans. Inf. Theory, 50 (2004),
1848-1850.

\bibitem{mc2} S. Ma, Y. Chang,
Constructions of optimal optical orthogonal codes with weight five,
J. Combin. Des., 13 (2005), 54-69.

\bibitem{ml}
S.V.~Maric and V.K.N.~Lau, Multirate fiber-optic CDMA:
System design and performance analysis, J. Lightwave Technol.,
16 (1998), 9-17.

\bibitem{mmc}
S.V.~Maric, O.~Moreno, and C.~Corrada, Multimedia transmission
in fiber-optic LANs using optical CDMA, J. Lightwave Technol.,
14 (1996), 2149-2153.

\bibitem{ms}
S. Mashhadi and J.A. Salehi, Code-division multiple-access techniques in optical fiber
networks-part III: optical AND logic gate receiver structure with generalized optical orthogonal
codes, IEEE Trans. Commun., 54 (2006), 1457-1468.

\bibitem{mms}
N. Miyamoto, H. Mizuno, and S. Shinohara, Optical orthogonal codes obtained from conics
on finite projective planes, Finite Fields Appl., 10 (2004), 405-411.

\bibitem{okeb}R. Omrani, G. Garg, P.V. Kumar, P. Elia and P. Bhambhani, Large families of asymptotically optimal two-dimensional optical orthogonal codes, IEEE Trans. Inf. Theory 58, (2012), 1163-1185.


\bibitem{rs}
R. Rees and D.R. Stinson, Frames with block size four, Can. J. Math., 44 (1992), 1030-1049.


\bibitem{sa}
J.A.~Salehi, Code division multiple-access techniques in
optical fiber networks-Part I: Fundamental principles, IEEE Trans.
Commun., 37 (1989), 824-833.

\bibitem{sb}J.A.~Salehi and C.A.~Brackett, Code division multiple-access techniques in optical fiber networks-Part II:
Systems performance analysis, IEEE Trans. Commun., 37 (1989), 834-842.

\bibitem{s}J.E. Simpson, Langford sequence: perfect and hooked, Discrete Math., 44
(1983), 97-104.

\bibitem{s2}N. Shalaby, The existence of near-Skolem and hooked near-Skolem
sequences, Discrete Math., 135 (1994), 303-319.

\bibitem{s3}
N. Shalaby, Skolem and Langford Sequences, CRC Handbook of Combinatorial Designs (C.J.~Colbourn and J.H.~Dinitz, eds.), CRC Press, Boca Raton, FL, 2007, 612-616.



\bibitem{wsy}
J.~Wang, X.~Shan, and J.~Yin, On constructions for optimal two-dimentional optical orthogonal codes, Des.
Codes Cryptogr., 54 (2010), 43-60.


\bibitem{wc}
L.~Wang and Y.~Chang, Combinatorial constructions of optimal three-dimensional optical
orthogonal codes, IEEE Trans. Inf. Theory, 61 (2015), 671-687.

\bibitem{wc3}
L.~Wang, Y.~Chang, Determination of sizes of optimal three-dimensional optical orthogonal
codes of weight three with AM-OPP property, J. Combin. Des., 25 (2017), 310-334.

\bibitem{wfpw}
L.~Wang, T.~Feng, R.~Pan, and X.~Wang, The spectrum of semicyclic holey group divisible designs with
block size three, J. Combin. Des., 28 (2020), 49-74.

\bibitem{wxc}X. Wang and Y. Chang, The spectrum of $(gv,g,3,\lambda)$-DF in $\mathbb{Z}_{gv}$, Science in China (A), 52 (2009), 1004-1016.

\bibitem{wc2}X. Wang and Y. Chang, Further results on $(v,4,1)$-perfect difference families, Discrete
Math., 310 (2010), 1995-2006.

\bibitem{w}
R.~Wei, Group divisible designs with equal-sized holes, Ars Combin., 35 (1993), 315-323.

\bibitem{yang}
Y. Yang, New enumeration results about the optical orthogonal codes, Inf. Process.
Lett., 40 (1991), 85-87.


\bibitem{yk}
G.C.~Yang and W.C.~Kwong, Performance comparison of multiwavelength CDMA and WDMA+CDMA for fiber-optic networks, IEEE Trans. Commun., 45 (1997), 1426-1436.

\bibitem{yin}
J. Yin, Some combinatorial constructions for optical orthogonal codes, Discrete Math., 185
(1998), 201-219.

\bibitem{yyl} J. Yin, X. Yang, and Y. Li, Some 20-regular CDP$(5,1,20u)$ and their applications, Finite
Fields Appl., 17 (2011), 317-328.

\bibitem{zc}
J. Zhang, Y. Chang, The spectrum of cyclic BSA$(v,3,\lambda;\alpha)$ with
$\alpha=2,3$, J. Combin. Des., 13 (2005), 313-335.

\bibitem{zfw}
M. Zhang, T. Feng, and X. Wang, The existence of cyclic $(v,4,1)$-designs, Des. Codes,
Cryptogr., 90 (2022), 1611-1628.

\bibitem{zcf}
C. Zhao, Y. Chang, T. Feng,
The existence of optimal $(v,4,1)$ optical orthogonal codes
achieving the Johnson bound, IEEE Trans. Inf. Theory, 70 (2024), 8746-8757.

\bibitem{zzcfwz}
C. Zhao, B. Zhao, Y. Chang, T. Feng, X. Wang, and M. Zhang,
Cyclic relative difference families with block size four and their applications, J. Combin. Thorey, Ser. A, 206 (2024),105890

\end{thebibliography}
\end{document}